\documentclass[11pt,final]{article}
\usepackage{amsmath,amsthm,amsfonts,amssymb,graphicx,hyperref,enumerate,psfrag}
\usepackage[a4paper,margin=2.5cm]{geometry}
\usepackage{fullpage}
\usepackage{enumitem}
\usepackage{bbm,exscale,stmaryrd,wasysym}

\usepackage{tikz}
\usetikzlibrary{positioning}
\tikzset{main node/.style={circle,fill=blue!20,draw,minimum size=1cm,inner sep=0pt},
        }
\tikzstyle{every loop}=[]
\usepackage[utf8]{inputenc} 

\renewcommand{\subset}{\subseteq}

\newcommand{\cz}{K_0}
\newcommand{\ci}{K_1}
\newcommand{\cii}{K_2}
\newcommand{\ciii}{K_3}
\newcommand{\civ}{K_5}
\newcommand{\cv}{K_6}
\newcommand{\cvi}{K_4}
\newcommand{\cvii}{K_7}
\newcommand{\cviii}{K_{12}}

\newcommand{\cx}{K_{14}}
\newcommand{\cxi}{K_{15}}
\newcommand{\cxii}{K_{16}}
\newcommand{\cxiii}{K_{8}}
\newcommand{\cxiv}{K_{9}}
\newcommand{\cxv}{K_{10}}
\newcommand{\cxvi}{K_{11}}

\newcommand{\Eh}{E_{\phid}^{\geq h}}

\newcommand{\Enh}{E_{\psin}^{\geq h}}
\newcommand{\Emnh}{E_{\psin}^{\geq h}}
\newcommand{\Ch}{\cC_{\circ}^{h}}

\newcommand{\Chplus}{\cC_{\circ}^{h,+}}

\newcommand{\Chminus}{\cC_{\oc}^{h,-}}

\newcommand{\Chdplus}{\mathcal{C}_{\circ}^{ h+\delta}}

\newcommand{\Chlnplus}{\mathcal{C}_{\circ}^{ h+\log^{-1}n}}
\newcommand{\Chlnminus}{\mathcal{C}_{\circ}^{ h-\log^{-1}n}}

\newcommand{\dtd}{d_{\Td}}
\newcommand{\oc}{\overline{\circ}}
\newcommand{\oX}{\overline{X}}

\newcommand{\ER}{\text{ER}}
\newcommand{\Bin}{\emph{Bin}}

\newcommand{\stdom}{\underset{st.}{\leq}}
\newcommand{\stdomi}{\underset{st.}{\geq}}

\newcommand{\cvpann}{\overset{\dP_{ann}}{\vspace{3mm}\longrightarrow}}
\newcommand{\bK}{\mathbf{K}}

\newcommand{\psimnt}{\widehat{\psi_{\cM_n}}}

\newtheorem{theorem}{Theorem}[section]

\newtheorem{lemma}[theorem]{Lemma}

\newtheorem{remark}[theorem]{Remark}

\newtheorem{proposition}[theorem]{Proposition}

\newtheorem{corollary}[theorem]{Corollary}

\newcommand{\dE}{\mathbb {E}}
\newcommand{\dP}{\mathbb {P}}
\newcommand{\dZ}{\mathbb {Z}}
\newcommand{\dN}{\mathbb {N}}
\newcommand{\dR}{\mathbb {R}}

\newcommand{\cA}{\mathcal {A}}

\newcommand{\cC}{\mathcal {C}}

\newcommand{\cE}{\mathcal {E}}
\newcommand{\cF}{\mathcal {F}}
\newcommand{\cG}{\mathcal {G}}

\newcommand{\cM}{\mathcal {M}}
\newcommand{\cN}{\mathcal {N}}

\newcommand{\cS}{\mathcal {S}}

\newcommand{\cZ}{\mathcal {Z}}

\newcommand{\bP}{\mathbf{P}}
\newcommand{\bE}{\mathbf{E}}
\newcommand{\bC}{\mathbf{C}}

\newcommand{\kT}{\mathfrak{T}}

\newcommand{\lv}{\left\vert}
\newcommand{\rv}{\right\vert}

\newcommand{\tx}{\text{tx}}
\newcommand{\fT}{\mathfrak{T}}

\newcommand{\psin}{\psi_{\cG_n}}
\newcommand{\psimn}{\psi_{\cM_n}}
\newcommand{\phid}{\varphi_{\mathbb{T}_d}}

\newcommand{\Td}{\mathbb{T}_d}
\newcommand{\ggn}{G_{\cG_n}}
\newcommand{\gmn}{G_{\cM_n}}
\newcommand{\gtd}{G_{\Td}}
\newcommand{\Var}{\text{Var}}
\newcommand{\Cov}{\text{Cov}}

\title{Anatomy of a Gaussian giant: supercritical level-sets of the free field on regular graphs}
\author{Guillaume Conchon-\,\hspace{-0.3mm}-Kerjan\footnote{University of Bath, United Kingdom. E-mail: galck20@bath.ac.uk}}
\date{}
\begin{document}
\maketitle

\textbf{Abstract.} We study the level-set of the zero-average Gaussian Free Field on a uniform random $d$-regular graph above an arbitrary level $h\in (-\infty, h_{\star})$, where $h_{\star}$ is the level-set percolation threshold of the GFF on the $d$-regular tree $\Td$. We prove that w.h.p as the number $n$ of vertices of the graph diverges, the GFF has a unique giant connected component $\cC_1^{(n)}$ of size $\eta(h) n+o(n)$, where $\eta(h)$ is the probability that the root percolates in the corresponding GFF level-set on $\Td$.
This gives a positive answer to the conjecture of \cite{ACregulgraphs} for most regular graphs. We also prove that the second largest component has size $\Theta(\log n)$.
\\
Moreover, we show that $\cC_1^{(n)}$ shares the following similarities with the giant component of the supercritical Erd\H{o}s-R\'enyi random graph. First, the diameter and the typical distance between vertices are $\Theta(\log n)$. Second, the $2$-core and the kernel encompass a given positive proportion of the vertices. Third, the local structure is a branching process conditioned to survive, namely the level-set percolation cluster of the root in $\Td$ (in the Erd\H{o}s-R\'enyi case, it is known to be a Galton-Watson tree with a Poisson distribution for the offspring). 
\\
\\
\textbf{Keywords:} Gaussian free field, percolation, random graphs.
\\
\textbf{MSC2020 subject classifications:} 60K35, 60G15, 60C05, 05C80.


\section{Introduction}
\subsection{Overview}

The Gaussian Free Field (GFF) on a transient graph $\cG$ is a Gaussian process indexed by the vertices. Its covariance is given by the Green function of the simple random walk on $\cG$, hence the GFF carries a lot of information on the structure of $\cG$ and on the behaviour of random walks, giving a base motivation for its study.
\\
Level-set percolation of the GFF has been investigated since the 1980s (\cite{BricmontLebowitzMaes, MolchanovStepanovPerco}). Lately, one important incentive  has been to gain information on the vacant set of random interlacements (\cite{Lupu, sznitman2012}), via Dynkin-type isomorphism theorems (\cite{eisenbaum2000, SabotTarres}). 
It was subject to much attention in the last decade on $\dZ^d$ (\cite{DrewitzPrevostRodriguez,DrewitzPrevostRodriguez2,DuminilAlGFF1,DuminilAlGFF2, Lupu, PopovRath, Sznitman12RIandGFF}). On such a lattice where the Green function decays polynomially with the distance between vertices, it provides a percolation model with long-range interactions. 
\\
More recently, level-set percolation was studied on transient rooted trees (\cite{ACregultrees, AS2018, SZ2016}). 
 There is a phase transition at a critical threshold $h_{\star}\in \dR$: if $h<h_{\star}$, the connected component of the root in the level-set above $h$ of the GFF has a positive probability to be infinite, and if $h>h_{\star}$, this probability is zero. 
\\
One can define an analogous field on a finite connected graph, the \textbf{zero-average Gaussian Free Field}, whose covariance is given by the \textbf{zero-average Green function} (see Section~\ref{subsec:defs}). A natural question is whether some characteristics of the GFF on an infinite graph $\cG$ can be transferred to a sequence of finite graphs $(\cG_n)_{n\geq 0}$ whose local limit is $\cG$. For instance, one might ask whether a phase transition for the existence of an infinite connected component of the level-set in $\cG$ corresponds to a phase transition for the emergence of a ``macroscopic'' component of size $\Theta(\vert \cG_n\vert)$\footnote{for two sequences $(a_n)_{n\geq 1}$, $(b_n)_{n\geq 1}$, we say that $a_n=\Theta(b_n)$ if there exists $c>0$ such that for every large enough $n$, $c<a_n/b_n<1/c$} in the level set in $\cG_n$. For $\cG=\dZ^d$, a sharp phase transition was shown in~\cite{DuminilAlGFF1}, and Abächerli~\cite{Abacherli} studied the zero-average GFF on the torus. 
\\
Abächerli and  \v{C}ern\'y recently investigated the GFF on the $d$-regular tree $\Td$ \cite{ACregultrees}, and the zero-average GFF on some $d$-regular graphs  (large girth expanders) in a companion paper \cite{ACregulgraphs}. In this setting, many essential questions (such as the value of $h_{\star}$, or the sharpness of the phase transition at $h_{\star}$ for the zero-average GFF) remain open. In this paper, we answer some of them, and relate the percolation level-sets to other classical random graphs, in particular the Erd\H{o}s-R\'enyi model (Section~\ref{subsec:openquestions}).

\subsection{Setting}\label{subsec:defs}
\noindent
In all this work, we fix an integer $d\geq 3$. We denote $\Td$ the infinite $d$-regular tree rooted at an arbitrary vertex $\circ$, 
and $\cG_n$ a uniform $d$-regular random graph of size $n$ for $n\geq 1$ (if $d$ is odd, consider only even $n$). Let $V_n$ be its vertex set and $\pi_n$ be the uniform measure on $V_n$, i.e. $\pi_n(x)=1/n$ for every $x\in V_n$.
\\

\noindent
\textbf{Gaussian Free Field on regular trees}

\noindent
The GFF $\phid$ on $\Td$ is a centred Gaussian field $(\phid(x))_{x\in \Td}$ indexed by the vertices of $\Td$, and with covariances given by the Green function $\gtd$: for all vertices $x,y\in \Td$, we set $\text{Cov}(\phid(x),\phid(y))=\gtd(x,y)$. Recall that 
\[
\gtd(x,y)= \bE_x^{\Td}\left[\sum_{k\geq 0}\mathbf{1}_{\{X_k=y\}}\right]\]
where $(X_k)_{k\geq 0}$ is a discrete-time SRW (Simple Random Walk) on $\Td$. In general, we will denote $\bP^{\cG}_{\mu}$ the law of a SRW on a graph $\cG$ with initial distribution $\mu$, and $\bE^{\cG}_{\mu}$ the corresponding expectation.
\\

\noindent
\textbf{Gaussian Free Field on $d$-regular graphs}

\noindent
If $\cG_n$ is connected, the zero-average GFF $\psin$ on $\cG_n$ is a centred Gaussian field $(\psin(x))_{x\in \cG_n}$ indexed by the vertices of $\cG_n$, and with covariances given by the zero-average Green function $\ggn$ on $\cG_n$: for all $x,y\in \cG_n$, we set
\[
\text{Cov}(\psin(x),\psin(y))=\ggn(x,y):=\bE_x^{\cG_n}\left[ \int_{0}^{+\infty}\left( \mathbf{1}_{\{\overline{X}_t=y\}}-\frac{1}{n}\right) dt\right]
\]
\noindent
where $(\oX_t)_{t\geq 0}$ is a continuous time SRW on $\cG_n$ started at $x$ with Exp(1) independent jump-times. Precisely, let $(\zeta_i)_{i\geq 1}$ be a sequence of independent exponential variables of parameter 1. Let $(X_k)_{k\geq 0}$ be a SRW started at $x$, independent of $(\xi_i)_{i\geq 1}$. Then for all $t\geq 0$, we define $\oX_t:=X_{k(t)}$, with $k(t):=\sup\{k\geq 0,\, \sum_{i=1}^k\zeta_i \leq t\}$. 
\\
The function $\ggn$ is symmetric, finite and positive semidefinite. This ensures that $\psin$ is well-defined (see \cite{Abacherli} for details, in particular Remark 1.2).
\\
If $\cG_n$ is not connected, then we set arbitrarily $\psin(x)=0$ a.s.~for all $x\in \cG_n$ (this will not play a significant role for our purpose, as $\cG_n$ is connected w.h.p., see Proposition~\ref{prop:goodgraph}).
\\
\\
\textbf{Two layers of randomness}
\\
Denote $\dP_{ann}$ and $\dE_{ann}$ the annealed law and expectation for the joint realization of $\cG_n$ and of $\psin$ on it. For a fixed realization of $\cG_n$, denote $\dP^{\cG_n}$ and $\dE^{\cG_n}$ the quenched law and expectation.

\subsection{Results}\label{subsec:results}
Define the level set $\Eh:=\{x\in \Td \, \vert \, \phid(x)\geq h\}$. Let $\Ch$ be the connected component of $\Eh$ containing the root $\circ$. Similarly, define the level sets $\Enh:=\{x\in \cG_n \, \vert \, \psin(x)\geq h\}$ for $n\geq 1$. For $i\geq 1$, let $\cC_i^{(n)}$ be the $i$-th largest connected component of $\Enh$, by number of vertices. Break ties arbitrarily if several components have the same size (this will not play a significant role in the paper).
In \cite{SZ2016}, Sznitman showed that there exists a constant $h_{\star}\in (0, \infty)$ such that 
\begin{equation}
\text{if $h> h_{\star}$, $\eta(h):=\dP^{\Td}(\vert\Ch\vert = +\infty)=0$, and if $h<h_{\star}$, $\eta(h)>0$.}
\end{equation}
\noindent
In \cite{ACregultrees} (Theorems 4.3 and 5.1), Abächerli and  \v{C}ern\'y showed that if $h>h_{\star}$, the size of $\Ch$ has exponential moments, and if $h<h_{\star}$, $\Ch$ has a positive probability to grow exponentially. In \cite{ACregulgraphs} (Theorems 3.1 and 4.1), they found that if $\cG_n$ satisfies some deterministic conditions, which hold w.h.p.~(with high probability), then the following events hold $\dP^{\cG_n}$-w.h.p.: if $h>h_{\star}$, $\vert \cC_1^{(n)}\vert=O(\log n)$, and if $h<h_{\star}$, at least $\xi n$ vertices of $\Enh$ are in components of size at least $n^{\delta}$, for some
constants $\delta, \xi>0$ depending on $h$. 

\noindent
Thus, in the supercritical case $h<h_{\star}$, a positive proportion of the vertices is in at least ``mesoscopic'' components (there is no explicit lower bound for $\delta$). 
\\
This work focusses exclusively on the supercritical case; recall that $d$ can be any fixed integer larger than 2. In the following results (and in the remainder of the paper), the constants $K_i$, $i\geq 0$ depend only on $d$ and $h$.
\\
We first prove the existence of a giant component:
\begin{theorem}\label{thm:maingff}
Let $h<h_{\star}$. The following holds:
\begin{equation}\label{eqn:mainthm}
\frac{\vert \cC_1^{(n)}\vert}{n}\cvpann \eta(h),
\end{equation}
where $\cvpann$ stands for convergence in $\dP_{ann}$-probability as $n\rightarrow +\infty$. Moreover, there exists $\cz>0$ such that
\begin{equation}\label{eqn:secondcompo}
\dP_{ann}\left(\cz^{-1}\log n\leq \vert \cC_2^{(n)} \vert\leq \cz\log n\right)\underset{n\rightarrow +\infty}{\longrightarrow} 1.
\end{equation}
\end{theorem}

\noindent
Note that by Markov's inequality, for any $\varepsilon >0$ and any sequence of events $(\cE_n)_{n\geq 1}$ such that $\dP_{ann}(\cE_n)\rightarrow 1$, w.h.p. $\cG_n$ is such that $\dP^{\cG_n}(\cE_n)\geq 1-\varepsilon$. Thus, w.h.p. on $\cG_n$, the conclusions of Theorem~\ref{thm:maingff} hold with arbitrarily large $\dP^{\cG_n}$-probability. After a first preprint of this work, \v{C}ern\'y~\cite{Cernygiantcompo} proved via a different approach that (\ref{eqn:mainthm}) also holds under the deterministic conditions of~\cite{ACregultrees} and~\cite{ACregulgraphs}, see Section~\ref{subsec:openquestions} for a discussion.
\\
\\
We also establish some structural properties of $\cC_1^{(n)}$. Let $\mathbf{C}^{(n)}$ be the \textbf{2-core} of $\cC_1^{(n)}$, obtained by deleting recursively the vertices of degree $1$ of $\cC_1^{(n)}$ and their edges. Let $\mathbf{K}^{(n)}$ be the \textbf{kernel} of $\cC_1^{(n)}$, i.e. $\mathbf{C}^{(n)}$ where simple paths are contracted to a single edge, so that the vertices of $\mathbf{K}^{(n)}$ are those of $\mathbf{C}^{(n)}$ with degree at least~$3$. 

\begin{theorem}{\textbf{\emph{Global structure of $\cC_1^{(n)}$}}}\label{thm:globalstruct}
\\
Fix $h<h_{\star}$. There exist  $\ci,\cii >0$ such that
\begin{equation}\label{eqn:core}
\frac{\vert \mathbf{C}^{(n)}\vert}{n}\cvpann \ci
\end{equation}
and 
\begin{equation}\label{eqn:kernel}
\frac{\vert \mathbf{K}^{(n)}\vert}{n}\cvpann \cii.
\end{equation}
\noindent
Moreover, there exists $\ciii>0$ such that if $D_1^{(n)}$ is the diameter of $\cC_1^{(n)}$, then
\begin{equation}\label{eqn:diameter}
\dP_{ann}(D_1^{(n)} \leq \ciii\log n)\underset{n\rightarrow +\infty}{\longrightarrow} 1.
\end{equation}

\noindent
Last, there exists $\lambda_h>1$ such that for every $\varepsilon >0$,
\begin{equation}\label{eqn:typicaldistance}
\pi_{2,n}(\{(x,y)\in (\cC_1^{(n)})^2 , \,(1-\varepsilon)\log_{\lambda_h}n \leq d_{\cC_1^{(n)}}(x,y)\leq (1+\varepsilon)\log_{\lambda_h}n\})\cvpann 1,
\end{equation}
where $\pi_{2,n}$ is the uniform measure on $(\cC_1^{(n)})^2$ and $d_{\cC_1^{(n)}}$ the usual graph distance on $\cC_1^{(n)}$. 
\noindent
In other words, the typical distance between vertices of $\cC_1^{(n)}$ is $\log_{\lambda_h}n$. 
\end{theorem}

\noindent
We will see in Section~\ref{sec:Td} that $\lambda_h$ is the growth rate of $\Ch$ conditioned on being infinite.
\\
Say that a random graph $G$ is the \textbf{local limit} of the random graph sequence $(G_n)_{n\geq 1}$ if $G_n$ converges to $G$ in distribution w.r.t to the local topology (see for instance the lecture notes of Curien~\cite{CurienLoccv} for a precise definition). We prove that the local limit of $\cC_1^{(n)}$ is $\Ch$ conditioned to be infinite. Say that two rooted trees $T,T'$ are \textbf{isomorphic} if there is a bijection $\Phi:T\rightarrow T'$ preserving the root, and such that for all vertices $x,y\in T$, there is an edge between $x$ and $y$ if and only if there is an edge between $\Phi(x)$ and $\Phi(y)$.

\begin{theorem}{\textbf{\emph{Local limit of $\cC_1^{(n)}$}}}\label{thm:locallimit}
\\
Fix $h<h_{\star}$. For every radius $k\geq 1$, for every rooted tree $T$ of height $k$, let $V_n^{(T)}:=\hspace{-1mm}\{x\in\cC_1^{(n)},B_{\cC_1^{(n)}}(x,k)\text{ is a tree isomorphic to }T\}$ and $p_T:=\dP^{\Td}(B_{\Ch}(\circ,k)=T\,\vert\,\,\vert \Ch\vert=+\infty)$. Then
$$\frac{\vert V_n^{(T)}\vert }{\vert \cC_1^{(n)}\vert}\cvpann  p_T$$
\end{theorem}


\subsection{Discussion and open questions}\label{subsec:openquestions}
\textbf{GFF percolation versus bond percolation} 
\\
The graph $\Enh$ undergoes the same phase transition as some classical bond percolation models for the size of the largest connected component. 
We draw a comparison with the Erd\H{o}s-Rényi random graph (i.e. bond percolation on the complete graph), introduced by Gilbert in~\cite{GilbertER}: for a constant $c>0$ and $n\in \dN$, $\ER(n,c/n)$ is the graph on $n$ vertices such that for every pair of vertices $x,y$, there is an edge between $x$ and $y$ with probability $c/n$, independently of all other pairs of vertices. Erd\H{o}s and R\'enyi~\cite{ERoriginal} showed that the supercritical regime corresponds to $c>1$ and the subcritical regime to $c<1$. Theorems~\ref{thm:maingff}, \ref{thm:globalstruct} and \ref{thm:locallimit} hold for $\ER(n,c/n)$ as $n\rightarrow +\infty$, for any fixed $c>1$, the tree $\Ch$ being replaced by a Galton-Watson tree whose offspring distribution is Poisson with parameter $c$, and $\lambda_h$ being replaced by $c$.
\\
As for Bernoulli bond percolation on $\cG_n$ (each edge of $\cG_n$ is deleted with probability $1-p$, independently of the others), the same phase transition holds for the size of the largest connected component, the critical threshold being $p=1/(d-1)$ (Theorem 3.2 of \cite{ABS04}).
\\
\\
\textbf{The structure of $\cC_1^{(n)}$}
\\
It was shown recently in \cite{DingLubetzkyPeres} that the distribution of the giant component of $\ER(n,c/n)$ is continuous w.r.t. to a random graph which can be explicitly described. Its kernel is a configuration model whose vertices have i.i.d. degrees with a Poisson distribution (conditioned on being at least 3). In particular, it is an expander. The lengths of the simple paths in the $2$-core are i.i.d. geometric random variables. See Theorem 1 of~\cite{DingLubetzkyPeres} for details. This implies a result analogous to Theorem~\ref{thm:globalstruct} for $\ER(n,c/n)$. 
\\
We conjecture that the kernel $\bK^{(n)}$ is an expander for every $h<h_{\star}$. The main obstacle to gathering information on its global structure is that if $\psin$ is revealed on a positive proportion of the vertices of $\bK^{(n)}$ (and hence of $\cG_n$), then it could affect substantially $\psin$ on the remaining vertices. In particular, if $h>0$ is large enough, we could imagine that the average of $\psin$ on the discovered vertices is positive. But by (\ref{eqn:zeroaverage}), the average of the GFF on the remaining vertices would be negative, hence below the threshold $h$.
\\
\\
\textbf{Deterministic regular graphs}

\noindent
The results of \cite{ACregulgraphs} and \cite{ABS04} hold in fact for any deterministic sequence of large-girth expanders (conditions \ref{expander1} and \ref{expander2} in Proposition~\ref{prop:goodgraph}), which is w.h.p. the case for $\cG_n$. Very recently, \v{C}ern\'y~\cite{Cernygiantcompo} gave another proof of (\ref{eqn:mainthm}) that holds under these deterministic conditions. He also showed that $\vert\cC_2^{(n)} \vert=o(n)$ w.h.p. His approach is very different, and uses notably a decomposition of the GFF as an infinite sum of fields with finite range interactions, introduced in~\cite{DuminilAlGFF2}. 
\\
In our proofs, averaging on the randomness of $\cG_n$ is a crucial ingredient to control the presence of cycles on large subgraphs of $\cG_n$, and allows us to extend some arguments of~\cite{ACregulgraphs}, where $\psin$ is locally approximated by $\phid$.
\\
We conjecture that those deterministic conditions are not sufficient for (\ref{eqn:secondcompo}) to hold. This was shown for the Bernoulli bond percolation in~\cite{krivelevichlubetzkysudakov} (Theorem 2): for every $a\in (0,1)$, one can build a sequence $(G_n)_{n\geq 1}$ satisfying \ref{expander1} and \ref{expander2} such that the second largest connected component has at least $n^a$ vertices (the second largest component first grows exponentially on a tree-like ball until it has a polynomial size, and then is ``trapped'' in zones where the expansion of the graph is close to an arbitrarily small constant).

\subsection{Proof outline}\label{subsec:proofstrategy}
\noindent
Our proofs rely on two main arguments:
\begin{enumerate}[label=\arabic*)]
\item\label{mainargt1} An annealed exploration of $\Emnh$ (Proposition~\ref{prop:gffann}), where the structure of $\cG_n$ is progressively revealed. There is a standard sequential construction of a uniform $d$-regular multigraph $\cM_n$, which, conditionned to be simple, yields $\cG_n$ (this conditionning has a non vanishing probability, see Section~\ref{sec:basicprop}). Each newly discovered vertex is given an independent standard normal variable. Then $\psin$ is built via a recursive procedure, using these Gaussian variables (Proposition~\ref{prop:condgff}).
\item\label{mainargt2} A comparison of $\psin$ and $\phid$ (Proposition~\ref{prop:couplinggffsexplo}): on a tree-like subgraph $T$ of $\cG_n$, such that there are no cycles in $\cG_n$ at distance $\kappa\log\log n$ of $T$ for a large enough constant $\kappa$, there is a bijective map $\Phi$ between $T$ and an isomorphic subtree of $\Td$ and a coupling of $\psin$ and $\phid$ 
so that
\[
\sup_{y\in T}\vert \psin(y)-\phid(\Phi(y))\vert \leq \log^{-1}n.
\]
\end{enumerate}

\noindent
We stress the fact that we reveal $\psin$ only \textit{after} having explored $\cG_n$: if we reveal $\psin$ at a given vertex, it conditions the structure of $\cG_n$ and thus the pairings of the still unmatched half-edges, so that we cannot use the sequential construction any more to further explore the graph. Hence, during the exploration, we will need to build an approximate version of $\psin$, depending on the Gaussian variables of \ref{mainargt1}. This makes some proofs tedious, in particular that of (\ref{eqn:secondcompo}).
\\
From Section~\ref{sec:basicprop} onwards, we will work exclusively on $\cM_n$ (and with the GFF on it, $\psimn$), proving that the statements of Theorems~\ref{thm:maingff}, \ref{thm:globalstruct} and \ref{thm:locallimit} hold w.h.p.~on $\cM_n$. Since the law of $\cG_n$ is that of $\cM_n$ under a non-vanishing conditioning, these statements hold w.h.p.~on $\cG_n$. For convenience, we continue this Section with $\cG_n$ and $\psin$ instead of $\cM_n$ and $\psimn$.
\\
\\
\textbf{The base exploration} \\
The exploration that we will perform in all proofs, with some modifications, is as follows: pick $x\in V_n$, and reveal its connected component $\cC_x^{\cG_n,h}$ in $\Enh$ in a breadth-first way, as well as its neighbourhood up to distance $a_n=\kappa\log\log n$. Until we meet a cycle, the explored zone is a tree $T_x$, growing at least like $\Chlnplus$, and at most like $\Chlnminus$ by \ref{mainargt2}.
\\
On one hand, $\Chlnplus$ has a probability $\simeq \eta(h+\log^{-1}n)=\eta(h)+o(1)$ to be infinite, with a growth rate $\lambda_h>1$ (Section~\ref{subsec:expogrowthCh}). On the other hand, the probability to create a cycle is $o(1)$ as long as we reveal $o(\sqrt{n})$ vertices (since we perform $o(\sqrt{n})$ pairings of half-edges having each a probability $o(\sqrt{n})/n$ to involve two already discovered vertices). Thus, with $\dP_{ann}$-probability at least $\eta(h)+o(1)$, $T_x$ and $\partial T_x$ will reach a size $\Theta(\sqrt{n}\log^{-\kappa'}n)$ for some constant $\kappa'>0$ (Proposition~\ref{prop:explo1vertex}). 
\\
Conversely, $\Chlnminus$ has a probability $1-\eta(h-\log^{-1}n)=1-\eta(h)+o(1)$ to be finite, and with $\dP_{ann}$-probability $1-\eta(h)+o(1)$, $\vert\cC_x^{\cG_n,h}\vert =o(\sqrt{n})$ (Proposition~\ref{prop:explo1vertexaborted}).
\\

\noindent
\textbf{Proof of (\ref{eqn:mainthm})}.

\noindent
First, we show that for any two vertices $x,y\in V_n$, there is a $\dP_{ann}$-probability $\eta(h)^2+o(1)$ that they are connected in $\Enh$. To do so, we explore $\cC_x^{\cG_n,h}$ and $\cC_y^{\cG_n,h} $, that we couple with independent copies of $\Chlnplus$, so that with probability $\eta(h)^2+o(1)$, $\partial T_x$ and $\partial T_y$ have $\Theta(\sqrt{n}\log^{-\kappa'}n)$ vertices. The explorations from $x$ and $y$ are disjoint with probability $1-o(1)$, since $o(\sqrt{n})$ vertices have been explored. Then, we draw multiple paths between $T_x$ and $T_y$ (with an ``envelope'' of radius $\Theta(\log\log n)$ around each of them to allow the use of the approximation \ref{mainargt2}), the \textbf{joining balls} (Section~\ref{subsec:connexion}). The probability that $\Enh$ percolates through at least one of these paths is $1-o(1)$. 
\\
Second, we prove by a second moment argument that $\dP_{ann}$-w.h.p., the number of couples $(x,y)\in V_n^2$ such that $y\in \cC_x^{\cG_n,h}$ is $(\eta(h)^2+o(1))n^2$ (Lemma~\ref{lem:connexionslowerbound}). 
\\ 
Third, knowing that $\vert\cC_x^{\cG_n,h}\vert =o(\sqrt{n})$ with $\dP_{ann}$-probability $1-\eta(h)+o(1)$, we deduce in the same way that at least $(1-\eta(h)+o(1))n$ vertices are in connected components of size $o(\sqrt{n})$ (Lemma~\ref{lem:connexionsupperbound}). 
\\
Those two facts together force the existence of a connected component of size $(\eta(h)+o(1))n$.
\\
\\
\textbf{Proof of (\ref{eqn:secondcompo}).}
\\
The most difficult part is the upper bound. We show that for $\cz$ large enough, for $x\in V_n$, $\dP_{ann}(\cz\log n\leq \vert \cC_x^{\cG_n,h}\vert \leq \cz^{-1}n )=o(1/n)$, and conclude by a union bound on $x$ and a corollary of the proof of (\ref{eqn:mainthm}), namely that $\vert\cC_2^{(n)}\vert /n \cvpann 0$. 
\\
The greater precision $o(1/n)$ requires three additional ingredients:
\begin{itemize}
\item the size of $\Ch$ conditioned on being finite has exponential moments (Proposition~\ref{prop:expomomentsCh}), in particular, $\dP^{\Td}(\vert\Ch\vert\geq c\log n, \, \vert \Ch\vert <+\infty)=o(1/n)$ for a large enough constant $c$;
\item when exploring $k$ vertices around $x$, there is a probability $\Theta(k^2/n)$ that a cycle arises, so that we will need to handle at least one cycle to fully explore $\cC_x^{\cG_n,h}$;
\item we need a better approximation of $\psin$ than $\log^{-1}n$ in \ref{mainargt2}: with probability at least $\Theta(1/n)$, we will meet too many vertices with an approximate value of $\psin$ that are in $[h-\log^{-1}n, h+\log^{-1}n]$, so that we can not tell whether they are in $\cC_x^{\cG_n,h}$ or not before the end of the exploration. To remedy this, we replace the ``security radius'' $a_n$ in \ref{mainargt2} by some $r_n=\Theta(\log n)$, so that we approximate $\psin$ up to a difference $n^{-\Theta(1)}$. 
\end{itemize}

\noindent
\textbf{Other proofs}.

\noindent
The proofs of Theorems~\ref{thm:globalstruct} and \ref{thm:locallimit} are based on slightly modified explorations, and are much simpler.

\subsection{Plan of the rest of the paper}
In Section~\ref{sec:basicprop}, we review some basic properties of $\cG_n$. In Section~\ref{sec:Td}, we study the GFF on $\Td$. In Section~\ref{sec:coupling}, we establish a coupling between recursive constructions of the GFF on $\Td$ and on a tree-like neighbourhood of $\cG_n$. As these three sections consist of preparatory work, most of their proofs are deferred to the \hyperref[appn]{Appendix}. The core ideas of the paper are in Sections~\ref{sec:exploration} and~\ref{sec:giant}. In Section~\ref{sec:exploration}, we explore the connected component of a vertex in $\Emnh$. In Section~\ref{sec:giant}, we prove \eqref{eqn:mainthm}. The most technical part of this article is Section~\ref{sec:uniqueness}, in which we prove (\ref{eqn:secondcompo}). In Section~\ref{sec:properties}, we prove Theorems~\ref{thm:globalstruct} and \ref{thm:locallimit}. 

\subsection{Further definitions}
In this paper, 
graphs are undirected.
For a graph $G$, denote $d_G$ the usual graph distance on its vertex set $V$, and for every vertex $x$ and every integer $R\geq 0$, let $B_G(x,R):=\{y, \, d_G(x,y)\leq R\}$ and $\partial B_G(x,R+1)=B_G(x,R+1)\setminus B_G(x,R)$. For any $S\subseteq V$, let similarly $B_G(S,R):=\cup_{x\in S}B_G(x,R)$ and $\partial B_G(S,R+1)=B_G(S,R+1)\setminus B_G(S,R)$. If $A$ is a subgraph of $G$ with vertex set $S$, let $B_G(A,R):=B_G(S,R)$. If $x$ and $y$ are neighbours, we denote $B_G(x,y,R)$ the subgraph of $G$ obtained by taking all paths of length $R$ starting at $x$ and not going through $y$.
\\
The \textbf{tree excess} of a finite graph $G$ is $\tx(G)=e-v+1$, where $v:=\vert V\vert$ and $e$ is the number of edges in $G$, see Section~\ref{subsec:AppendixSection4} for elementary facts on the tree excess that will be useful throughout the paper. 
\\
A \textbf{rooted tree} is a tree $T$ with a distinguished vertex $\circ$, the \textbf{root}. 
The \textbf{height} $\mathfrak{h}_T(x)$ of a vertex $x$ in $T$ is $d_T(\circ,x)$. If $T$ is finite, its \textbf{boundary} $\partial T$ is the set of vertices of maximal height. 
The \textbf{subtree from $x$} is the subtree made of the vertices $y$ such that $x$ is on any path from $\circ$ to $y$. The \textbf{offspring} of $x$ is the set of vertices of its subtree. For $r\geq 0$, the \textbf{$r$-offspring} of $x$ is its offspring at distance $r$ of $x$, and its \textbf{offspring up to generation $r$} is its offspring at distance at most $r$. If $y$ is in the $1$-offspring of $x$, then $y$ is a \textbf{child} of $x$, and $x$ is its \textbf{parent}. In this case, write $x=\overline{y}$.
\\
If $x,y$ are neighbours in $T$, the \textbf{cone from $x$ out of $y$} is the rooted subtree of $T$ with root $x$ and vertex set $\{z\in T\, \vert\, \text{$y$ is not on the shortest path from $x$ to $z$}\}$. 
\\
Unless mention of the contrary, all random walks are in discrete time. We will write $T_A$ (resp. $H_A$) for the first exit (resp. hitting) time of a set $A$ by a SRW. 
\\
For two probability distributions $\mu,\mu'$ on $\dR$, we write $\mu \stdom\mu'$ (or $\mu'\stdomi \mu$) if $\mu'$ dominates stochastically $\mu$, i.e. there exist two random variables $X\sim \mu$ and $X'\sim \mu'$ on the same probability space such that $X\leq X'$ a.s.

\section{Basic properties}\label{sec:basicprop}
\subsection{From $\cG_n$ to the multigraph $\cM_n$}\label{subsec:basicstructureGreen}
\noindent
The graph $\cG_n$ can be generated sequentially as follows: attach $d$ half-edges to each vertex of $V_n$. Pick an arbitrary half-edge, and match it to another half-edge chosen uniformly at random. Then, choose a remaining half-edge and match it to another unpaired half-edge chosen uniformly at random among the remaining half-edges, and so on until all half-edges have been paired. The resulting multi-graph $\cM_n$ is not necessary \textbf{simple}, i.e. it might have loops and multiple edges. The probability that $\cM_n$ is simple has a positive limit as $n\rightarrow +\infty$, and conditionally on $\{\cM_n\text{ is simple}\}$, $\cM_n$ is distributed as $\cG_n$ (see for instance  Section 7 of \cite{RemcoBook}, in particular Proposition 7.13 for a reference).
\\
In particular, an event true w.h.p. on $\cM_n$ is also true w.h.p. on $\cG_n$, so that it is enough to prove all our results on $\cM_n$. In the rest of the paper, we will exclusively work on $\cM_n$. 
\\
Note that $\cM_n$ is not necessarily connected, but that it is $d$-regular. The SRW on $\cM_n$ is as follows:  if a vertex $x$ has a loop, then a SRW starting at $x$ goes through the loop with probability $2/d$, and if there is an edge with multiplicity $\ell\geq 1$ between $x$ and an other vertex $y$, the SRW moves to $y$ with probability $\ell/d$.  The uniform measure $\pi_n$ on its vertices is still invariant for the SRW, and is the unique such probability measure if $\cM_n$ is connected. Also, the SRW is still reversible. These observations allow to transpose readily the definitions in Section~\ref{subsec:defs} from $\cG_n$ to $\cM_n$. When $\cM_n$ is connected, denote $\gmn$ the Green function (which is still symmetric, finite and positive semidefinite), and $\psimn$ the GFF. When $\cM_n$ is not connected, we also impose $\psimn(x)=0$ for all $x\in V_n$. 

\noindent
The following proposition is the main result of this Section (recall that the constants $K_i$'s implicitely depend on $d$ and $h$). We postpone its proof to the \hyperref[appn]{Appendix}.
\begin{proposition}\label{prop:goodgraph}
There exists $\ciii>0$ such that w.h.p. as $n\rightarrow +\infty$, $\cM_n$ satisfies:
\begin{enumerate}[label=(\Roman*)]
\item \label{expander1} $\cM_n$ is a \textbf{$\ciii$-expander}, i.e. the spectral gap $\lambda_{\cM_n}$ of $\cM_n$ is at least $\ciii$ (the spectral gap is $1-\lambda_2/d$, where $\lambda_2$ is the second largest eigenvalue of the adjacency matrix of $\cM_n$),
\item \label{expander2} for all $x\in \cM_n$, $B_{\cM_n}(x,\lfloor\ciii\log n\rfloor)$ contains at most one cycle.
\end{enumerate}
\noindent
Moreover, there exists $\cvi>0$ such that w.h.p. on $\cM_n$, the following holds: for all $x\in V_n$ such that \emph{tx}$(B_{\cM_n}(x,\lfloor \cvi\log\log n\rfloor))=0$, 
\begin{equation}\label{eqn:greenfunctionGnapprox1}
\left\vert\gmn(x,x)-\frac{d-1}{d-2}\right\vert\leq \log^{-6}n.
\end{equation}
If moreover $y$ is a neighbour of $x$,
\begin{equation}\label{eqn:greenfunctionGnapprox2}
\left\vert\gmn(x,y)-\frac{1}{d-2}\right\vert\leq \log^{-6}n.
\end{equation}
\end{proposition} 

\noindent
Say that a given realization of $\cM_n$ is a \textbf{good graph} when \ref{expander1}, \ref{expander2}, 
(\ref{eqn:greenfunctionGnapprox1}) and (\ref{eqn:greenfunctionGnapprox2}) hold. Remark that a good graph is necessarily connected, as the spectral gap of a non-connected graph is 0 (see Section~2.3 of~\cite{Hoory}). In particular, $\cM_n$ (and thus $\cG_n$) is connected w.h.p. 
The equations (\ref{eqn:greenfunctionGnapprox1}) and (\ref{eqn:greenfunctionGnapprox2}) illustrate the fact that $G_{\cM_n}$ is close to $G_{\Td}$ on a tree-like neighbourhood: it is well-known that for all $x,y\in \Td$, 
\begin{equation}\label{eqn:greenonTd}
G_{\Td}(x,y)= \frac{(d-1)^{1-\dtd(x,y)}}{d-2}.
\end{equation}
A quick computation can be found in \cite{Woess}, Lemma 1.24.

\noindent
By Proposition~1.1 of \cite{ACregulgraphs} (whose proof also works with loops and multiple edges) \ref{expander1} and \ref{expander2} imply that for some $\civ,\cv>0$ and for $n$ large enough, for all $x,y\in V_n$, 
\begin{equation}\label{eqn:greenfunctionGn}
\vert\gmn(x,y)\vert\leq \frac{\civ}{(d-1)^{d_{\cM_n}(x,y)}} \vee n^{-\cv}.
\end{equation}
\noindent
Throughout this paper, we will often make binomial estimates, because the number of edges between two sets of vertices in $\cM_n$ is close to a binomial random variable, as highlighted in the Lemma below. We will use repeatedly the following classical inequalities: for $n\geq m\geq 0$ and $p\in (0,1)$, if $Z\sim \text{Bin}(n,p)$, one has
\begin{equation}\label{eqn:binomiallittlethings}
\dP(Z\geq m)\leq {n\choose m}p^m, \,\,\dP(Z\leq m)\leq {n\choose m}(1-p)^m,\,\,{n\choose m}\leq\frac{n^m}{m!} \leq n^m.
\end{equation} 

\noindent
The following Lemma is an important consequence of the sequential construction of $\cM_n$.

\begin{lemma}[\textbf{Binomial number of connections}]\label{lem:matchings}
Let $m\in \dN$, let $W_0,W_1$ be disjoint subsets of $V_n$. Write $m_0:=\vert W_0\vert$ and $m_1:=\vert W_1\vert$. Suppose the only information we have on $\cM_n$ is a set $E$ of its edges that has been revealed. Let $m_E:=\vert E\vert$ and denote $\dP^E$ the law of $\cM_n$ conditionally on this information. 
Repeat the following operation $m$ times: pick an arbitrary vertex $v\in W_0$ having at least one unmatched half-edge, and pair it with an other half-edge. Add its other endpoint $v'$ in $W_0$, if it was not already in it. Let $s$ be the number of times that $v'\in W_1$. Suppose that $m_E+m+m_1<n$. Then
\begin{equation}\label{eqn:binomdifferent}
s \stdom\Bin\left(m, \frac{m_1}{n-(m_E+m)}\right).
\end{equation}
In particular,
\\
a) for any fixed $k\in\dN$, there exists $C(k)>0$ so that for $n$ large enough, if $m_E+m<n/2$,
\begin{equation}\label{eqn:binomdominationdiff}
\dP^E(s\geq k)\leq C(k)\left(\frac{m_1m}{n}\right)^k.
\end{equation}

\noindent
b) for $k=k(n)\rightarrow +\infty$ and $n$ large enough, if we have $m_E+m<n/2$ and $kn >6( m_1+m_0+m_E)m$, then
\begin{equation}\label{eqn:binomdominationatinfinitydiff}
\dP^E(s\geq k)\leq 0.99^k.
\end{equation}

\end{lemma}

\begin{proof}
Pick $v\in W_0$, such that $v$ has an unmatched half-edge $e$. There are at most $m_1$ vertices in $W_1$, so that there are at most $dm_1$ unmatched half-edges that belong to its vertices. And the total number of unmatched half-edges is at least $dn-2(\vert E\vert +m)\geq d(n-m_E-m)$. Thus, the probability that $e$ is matched with a half-edge belonging to a vertex of $W_1$ is not greater than $\frac{dm_1}{d(n-m_E-m)}=\frac{m_1}{n-(m_E+m)}$, and this bound does not depend on the outcome of the previous matchings. (\ref{eqn:binomdifferent}) follows. 
\\
Let $Z\sim \Bin\left(m, \frac{m_1}{n-(m_E+m)}\right)$. By (\ref{eqn:binomiallittlethings}), for $k\in \dN$, we have
\begin{center}
$\dP^E(Z\geq k)\leq {m\choose k}\left(\frac{m_1}{n-(m_E+m)}\right)^k\leq {m\choose k}\left(\frac{m_1}{n/2}\right)^k\leq \frac{2^k}{k!}\frac{m_1^{k}m^k}{n^k}$. 
\end{center}
This yields (\ref{eqn:binomdominationdiff}). Moreover, if $k\rightarrow +\infty$ as $n\rightarrow +\infty$ and $kn >6( m_1+m_0+m_E)m$, by Stirling's formula, we have $\dP^E(Z\geq k)\leq \left(\frac{(2e+0.1)m_1m}{kn}\right)^k<0.99^k$  for $n$ large enough, and (\ref{eqn:binomdominationatinfinitydiff}) follows.
\end{proof}

\noindent
It is straightforward to adapt this when $s$ counts the number of times that $v'$ was in $W_0$ (and there is no set $W_1$). $m_1$ is replaced by $m_0+m$ in (\ref{eqn:binomdifferent}) and (\ref{eqn:binomdominationdiff}), and (\ref{eqn:binomdominationatinfinitydiff}) does not change. Throughout this paper, we will refer to these equations without mentioning explicitly if we count the connections from $W_0$ to $W_1$ or from $W_0$ to itself.

\subsection{GFF on $\cM_n$}

\noindent
The name ``zero-average'' for the GFF on $\cM_n$ (or $\cG_n$) comes from the fact that a.s., 
\begin{equation}\label{eqn:zeroaverage}
\sum_{x\in V_n}\psimn(x)=0
\end{equation}
since $\text{Var}\left(\sum_{x\in V_n}\psimn(x)\right)=\sum_{x,y\in V_n}\gmn(x,y)=0.$
\\
Hence, there is no domain Markov property. However, there is a recursive construction of $\psimn$:

\begin{proposition}[Lemma 2.6 in \cite{ACregulgraphs}]\label{prop:condgff}
Let $A\subsetneq V_n$, $x\in V_n\setminus A$. Write $\sigma(A):=\sigma(\{\psimn(y),\, y\in A\})$. Let $(X_k)_{k\geq 0}$ be a SRW on $\cM_n$ and let $H_A$ be the hitting time of $A$. Conditionally on $\sigma(A)$, $\psimn(x)$ is a Gaussian variable, such that
\begin{equation}\label{eqn:condexpect}
\dE^{\cM_n}[\psimn(x)\vert \sigma(A)]=\bE_x^{\cM_n}[\psimn(X_{H_A})]-\frac{\bE_x^{\cM_n}[H_A]}{\bE_{\pi_n}^{\cM_n}[H_A]}\bE^{\cM_n}_{\pi_n}[\psimn(X_{H_A})]
\end{equation}
and 
\begin{equation}\label{eqn:condvar}
\emph{Var}^{\cM_n}(\psimn(x)\vert \sigma(A))=\gmn(x,x)-\bE_x^{\cM_n}[\gmn(x,X_{H_A})]+\frac{\bE_x^{\cM_n}[H_A]}{\bE_{\pi_n}^{\cM_n}[H_A]}\bE_{\pi_n}^{\cM_n}[\gmn(x,X_{H_A})].
\end{equation}
\end{proposition}

\noindent
Combining this Lemma and the sequential construction $\cM_n$, we obtain the following.

\begin{proposition}[\textbf{Joint realization of $\cM_n$ and $\psimn$}]\label{prop:gffann}
A realization of $(\cM_n,\psimn)$ is given by the following process.
Let $(\xi_i)_{i\geq 1}$ be a sequence of i.i.d.~$\cN(0,1)$ variables. A \textbf{move} consists~in:
\begin{itemize}
\item choosing an unpaired half-edge $e$ and matching it to another unpaired half-edge chosen uniformly at random (independently of $(\xi_i)_{i\geq 1}$), or
\item choosing $x\in V_n$ and $k\in \dN$ so that $\xi_k$ has not yet been attributed, and attributing $\xi_k$ to $x$.
\end{itemize}
At each move, the choice of $e,x$ or $k$ might depend in an arbitrary way on the previous moves, i.e. on the matchings and on the value of the normal variables attributed before, but \textbf{not} on the value of the remaining normal variables.
Perform moves until all half-edges are paired, and every vertex $x\in V_n$ has received a normal variable, that we denote $\xi_x$. 
\\
To generate $\psimn$, let $x_1, \ldots, x_n$ be the vertices of $V_n$, listed in the order in which they received their normal variable. Let $\psimn(x_1):=\sqrt{\gmn(x_1,x_1)}\xi_{x_1}$. For $i=2, \ldots, n$ successively, define $A_i:=\{x_1, \ldots,x_{i-1}\}$. Recall that we write $\sigma(A_i)$ for $\sigma(\{\psimn(y), y\in A_i \})$. Let 
\begin{center}
$\psimn(x_i):=\dE^{\cM_n}[\psimn(x_{i})\vert \sigma(A_{i})] + \xi_{x_{i}}\sqrt{\emph{Var}(\psimn(x_{i})\vert \sigma(A_{i}))}$.
\end{center}

%
\end{proposition}
\noindent
It might be confusing that $\psimn(x_{i})$ appears on both sides of the equation. Note that the conditional expectation and variance on the RHS are $\sigma(A_i)$-measurable random variables.

\begin{proof}
Clearly, the graph obtained after pairing all the half-edges is distributed as $\cM_n$. 
For every $i\geq 1$, $\xi_{x_{i}}$ is a standard normal variable independent of the realization of $\cM_n$, and of $\sigma(A_i)$ for $i\geq 1$, so that we can conclude by Proposition~\ref{prop:condgff}.
\end{proof}

\noindent
Last, we prove that the maximum of $\vert \psimn\vert$ on $\cM_n$ has a subexponential tail.
\begin{lemma}[\textbf{Tail for the maximum of $\vert\psimn\vert$}]\label{lem:maxgff}
Suppose that $\max_{x\in V_n}G_{\cM_n}(x,x)\leq \civ$. Then for all $\Delta>0$, if $n$ is large enough,
\begin{equation}\label{eqn:maxgff}
\dP^{\cM_n}\left(\max_{x\in V_n}\vert\psimn(x)\vert\geq \log^{2/3} n\right)\leq n^{-\Delta}.
\end{equation} 
In particular, by Proposition~\ref{prop:goodgraph} and (\ref{eqn:greenfunctionGn}), w.h.p. $\cM_n$ satisfies (\ref{eqn:maxgff}).
\end{lemma}

\begin{proof}
Let $N\sim \cN(0,\civ)$. If $n$ is large enough, then for all $x\in V_n$, 
\[
\dP^{\cM_n}\left(\vert\psimn(x)\vert\geq  \log^{2/3} n\right)\leq \dP^{\cM_n}\left(\vert N\vert  \geq \log^{2/3} n\right)\leq 2\exp\left(-\frac{\log^{4/3} n}{4\civ}\right)\leq n^{-\Delta-1}
\] 
by Markov's inequality applied to the function $u\mapsto \exp\left(\frac{\log^{2/3}n}{2\civ}u\right)$. By a union bound on all $x\in V_n$, we get 
$\dP^{\cM_n}\left(\max_{x\in \cM_n}\vert\psimn(x)\vert> \log^{2/3} n\right) \leq n^{-\Delta}.$
\end{proof}

\section{The Gaussian Free Field on $\Td$}\label{sec:Td}
\noindent
In Section \ref{subsec:branchingCh}, we characterize $\Ch$ as a branching process, with a recursive construction (Proposition~\ref{prop:recursivegfftrees}). Then, in Section~\ref{subsec:expogrowthCh}, we establish its exponential growth, conditionally on the event $\{\vert \Ch\vert = +\infty\}$. The main results are Propositions~\ref{prop:thm43adapted}, \ref{prop:expomomentsCh} and \ref{prop:Chlargedevgrowthrate}.


\subsection{$\Ch$ as a branching process}\label{subsec:branchingCh}
\noindent
There is an alternate definition of $\phid$, starting from its value at $\circ$ and expanding recursively to its neighbours. It shows that $\Ch$ is an infinite-type branching process, the type of a vertex $x$ being $\phid(x)$.

\begin{proposition}[\textbf{Recursive construction of the GFF},(1.4)-(1.9) in~\cite{ACregultrees}]\label{prop:recursivegfftrees}
Define a Gaussian field $\varphi$ on $\Td$ as follows: let $(\xi_y)_{y\in\Td}$ be a family of i.i.d. $\cN(0,1)$ random variables. Let $\varphi(\circ):=\sqrt{\frac{d-1}{d-2}}\xi_{\circ}$. For every $y\in \Td\setminus \{\circ\}$, define recursively $\varphi(y):=\sqrt{\frac{d}{d-1}}\xi_y+\frac{1}{d-1}\varphi(\overline{y}),$ where $\overline{y}$ is the parent of $y$.
Then 
\[\varphi \overset{d.}{=}\phid.
\]
\end{proposition}
\noindent
Proposition~\ref{prop:recursivegfftrees} is the corollary of a more general domain Markov property (see for instance Lemma 1.2 of \cite{RodriguezSznitman} where it is stated for $\mathbb{Z}^d$, but the proof works for any transient graph). 
\\
%
Write $\dP^{\Td}$ for the law of $\phid$, and $\dP^{\Td}_a$ for $\dP^{\Td}(\,\cdot\,\vert \phid(\circ)=a)$, $a\in \dR$ (such conditioning is well-defined, $(\phid(x))_{x\in \Td}$ being a Gaussian process). This construction gives a monotonicity property for $\phid$. A set $S\subset \dR^{\Td}$ is said to be \textbf{increasing} if for any $(\Phi^{(1)}_z)_{z\in \Td}, (\Phi^{(2)}_z)_{z\in \Td}\in \dR^{\Td}$ such that $\Phi^{(1)}_z \leq \Phi^{(2)}_z$ for all $z\in \Td$, $  (\Phi^{(1)}_z)_{z\in \Td}\in S$ only if $(\Phi^{(2)}_z)_{z\in \Td}\in S$. Say that the event $\{\phid \in S\}$ is \textbf{increasing} if $S$ is increasing.

\begin{lemma}[\textbf{Conditional monotonicity}]\label{lem:monotonicityphid}
If $E$ is an increasing event, then the map $a \mapsto \dP_a^{\Td}(E)$ is non-decreasing on $\dR$.
\end{lemma}

\begin{proof}
Let $a_1,a_2\in \dR$ such that $a_1>a_2$. It suffices to give a coupling between a GFF $\phid^{(1)}$ conditioned on $\phid^{(1)}(\circ)=a_1$ and a GFF $\phid^{(2)}$ conditioned on $\phid^{(2)}(\circ)=a_2$ such that a.s., for every $z\in \Td$, $\phid^{(1)}(z)\geq \phid^{(2)}(z)$. To do this, let $(\xi_y)_{y\in \Td}$ be i.i.d. standard normal variables, and define recursively $\phid^{(1)}$ and $\phid^{(2)}$ as in Proposition~\ref{prop:recursivegfftrees}. Then for every $z\in \Td$ of height $k\geq 0$, $\phid^{(1)}(z)=\phid^{(2)}(z)+(a_1-a_2)(d-1)^{-k}$.
\end{proof}

\subsection{Exponential growth}\label{subsec:expogrowthCh}

\noindent
All proofs of this section are postponed to the \hyperref[appn]{Appendix}, Section~\ref{subsec:AppendixSection3}. 
\\
Let $\cZ_k^h:=\Ch\cap \partial B_{\Td}(\circ,k)$ be the $k$-th generation of $\Ch$. We first characterize the growth rate of $\cZ_k^h$. The first statement of the proposition below is a variant of Theorem~4.3 in~\cite{ACregultrees}.
\begin{proposition}\label{prop:thm43adapted}
There exists $\lambda_h >1$ such that
\[
\lim_{k\rightarrow +\infty}\dP^{\Td}(\vert\cZ_k^h\vert>\lambda_{h}^k/k^2)= \eta(h)
\] 
and
\[
\lim_{k\rightarrow +\infty}\dP^{\Td}(\vert\cZ_k^h\vert < k\lambda_{h}^k)=1.
\] 
\noindent
Moreover, $h\mapsto \lambda_h$ is a decreasing homeomorphism from $(-\infty,h_{\star})$ to $(1,d-1)$. 
\end{proposition}

\noindent
We now give finer results on the growth rate of $\vert \Ch\vert$. It turns out that $\vert \Ch\vert$ conditioned to be finite has exponential moments: 
\begin{proposition}\label{prop:expomomentsCh}
There exists a constant $\cvii>0$ such that as $k\rightarrow +\infty$,
\begin{equation}\label{eqn:expomoments}
\max_{a\geq h}\dP^{\Td}_a(k\leq \vert \Ch \vert <+\infty)=o(\exp(-\cvii k)).
\end{equation}
\end{proposition}

\noindent
Since $\{\cZ_k^h\neq \emptyset\} \subset \{\vert \Ch\vert \geq k\} $, we have the following straightforward consequence:
\begin{corollary}\label{cor:expomomentsheight}
For $k$ large enough, for every $a\geq h$,\\
$\dP^{\Td}_a(\vert\Ch\vert=+\infty)\leq \dP^{\Td}_a(\cZ_k^h\neq \emptyset)\leq \dP^{\Td}_a(\vert\Ch\vert=+\infty)+e^{-\cvii k}$.
\end{corollary}

\noindent
In addition, there are large deviation bounds for the growth rate of $\cZ_k^h$:
\begin{proposition}\label{prop:Chlargedevgrowthrate}
For every $\varepsilon >0$, there exists $C> 0$ such that for every $k\in \dN$ large enough, 
\begin{equation}\label{eqn:expomomentssize}
\max_{a\geq h}\dP^{\Td}_a(k^{-1}\log\vert \cZ_k^h\vert \not\in [\log(\lambda_h-\varepsilon), \log(\lambda_h +\varepsilon)+k^{-1}\log\chi_h(a)]\,\vert \,\cZ_k^h \neq \emptyset)\leq \exp(-Ck). 
\end{equation}
\end{proposition}
Let $\oc$ be an arbitrary neighbour of $\circ$. Let $\Td^+$ be the cone from $\circ$ out of $\oc$. Write $\Chplus:=\Ch\cap \Td^+$. For $k\geq 1$, let $\cZ_k^{h,+}:=\Chplus \cap \partial B_{\Td^+}(\circ,k)$.  The last proposition also holds when replacing $\Ch$ by $\Chplus$, and $\cZ_k^h$ by $\cZ_k^{h,+}$.

\section{Approximate recursive construction of $\psimn$}\label{sec:coupling}

\noindent
Let $\kappa >0$ be a constant, and let
\begin{equation}\label{eqn:andef}
a_n:=\lfloor\kappa \log_{d-1} \log n\rfloor.
\end{equation}
The following statement is the main result of this section. 
It shows that a recursive construction of $\psimn$, under some assumptions on the subset $A\subset V_n$ of vertices where $\psimn$ is already known, is very close to the construction of $\phid$ in Proposition~\ref{prop:recursivegfftrees}. It is a crucial tool for comparing $\psimn$ and $\phid$ in the exploration in the next section.
It is analogous to Proposition~2.7 of \cite{ACregulgraphs}, where the assumptions on $A$ are slightly different: they are suited to a deterministic $d$-regular graph satisfying \ref{expander1} and \ref{expander2}, while ours will be adapted to an annealed exploration, where the randomness of $\cM_n$ plays a role. The proof is postponed to the \hyperref[appn]{Appendix}, Section~\ref{subsec:AppendixSection4}.

\begin{proposition}\label{prop:couplinggffsexplo}
If the constant $\kappa$ from (\ref{eqn:andef}) is large enough, then the following holds for $n$ large enough. Assume that $\cM_n$ is a good graph as defined in Proposition~\ref{prop:goodgraph}, and that $A\subset V_n$ satisfies
\begin{itemize}
\item $\vert A\vert \leq n\log^{-8}n$,
\item $\text{\emph{tx}}(B_{\cM_n}(A,a_n))=\text{\emph{tx}}(A)$, and
\item $ \max_{z\in A}\vert\psimn(z)\vert\leq \log^{2/3}n$.
\end{itemize}
For every $y\in \partial B_{\cM_n}(A,1)$, writing $\overline{y}$ for the unique neighbour of $y$ in $A$, we have:
\begin{equation}\label{eqn:condexpapprox}
\left\vert\dE^{\cM_n}[\psimn(y)\vert \sigma(A)]-\frac{1}{d-1}\psimn(\overline{y})\right\vert \leq \log^{-3}n
\end{equation}

and

\begin{equation}\label{eqn:condvarapprox}
\left\vert \emph{\text{Var}}^{\cM_n}(\psimn(y)\vert \sigma(A))-\frac{d}{d-1}\,\,\right\vert\leq \log^{-4}n.
\end{equation}

\end{proposition}
\noindent
We stress that the result holds for a fixed realization of $\cM_n$.

\includegraphics[scale=0.5]{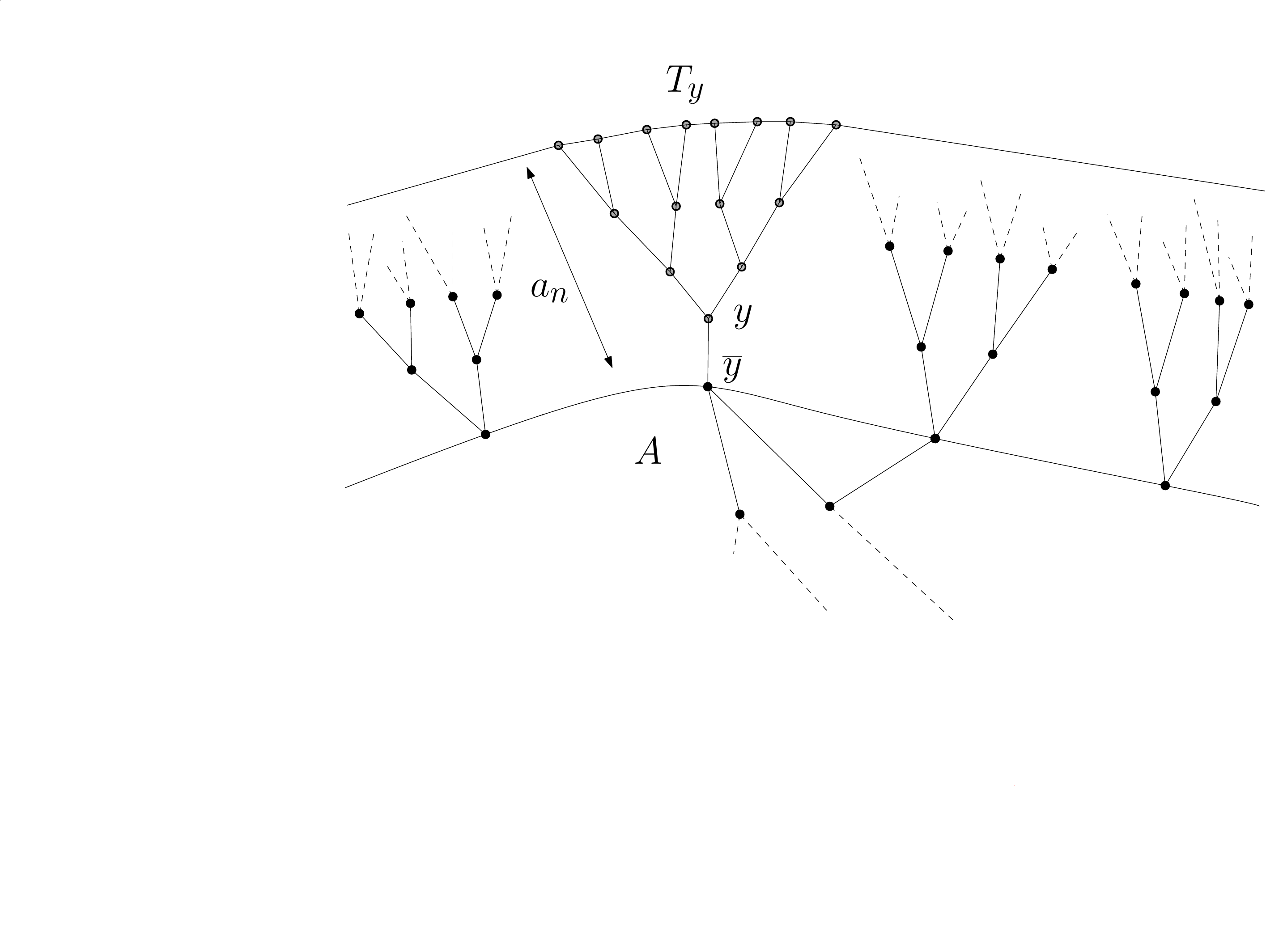}
\vspace{-28mm}
\begin{center}
Figure 1. The unicity of $\overline{y}$ comes from the fact that when building $B_{\cM_n}(A,a_n)$ from $A$, no cycle appears and no connected components of $A$ join, since $\tx(B_{\cM_n}(A,a_n))=\tx(A)$. 
\end{center}

\section{Exploration of $\psimn$ around a vertex}\label{sec:exploration} 
\subsection{Successful exploration}\label{subsec:explo1vertexsuccess}
In this section, we prove that with $\dP_{ann}$-probability arbitrarily close to $\eta(h)$ as $n\rightarrow +\infty$, $\vert \cC_x^{\cM_n,h}\vert \geq n^{1/2}\log^{-\kappa-6}n$, $\cC_x^{\cM_n,h}$ being the connected component of a given vertex $x$ in $\Enh$ (Proposition~\ref{prop:explo1vertex}), and $\kappa$ the constant defined in (\ref{eqn:andef}).
\\
To do so, we explore a tree-like neighborhood $T_x$ of $x$ in $\Enh$. We reveal $T_x$ generation by generation, and couple it with a realization of $\Chlnplus \subset \Td$ that is independent of the pairing of the half-edges of $\cM_n$. By Proposition~\ref{prop:gffann}, a realization of $\psimn$ is given by a recursive construction with the same normal variables as those of the realization of $\phid$.
When:
\begin{itemize}
\item that realization of $\Chlnplus$ is infinite (which happens with probability $\eta(h+\log^{-1}n)\simeq \eta(h)$), and 
\item the conditions of Proposition~\ref{prop:couplinggffsexplo} hold for each vertex of $T_x$, until a generation at which $\vert \partial T_x\vert \geq n^{1/2}\log^{-\kappa-6}n$ (which happens with probability $1-o(1)$, mainly because we have a probability $o(1)$ to create a cycle when pairing $o(\sqrt{n})$ half-edges),
\end{itemize}
we can apply Proposition~\ref{prop:couplinggffsexplo} to bound the difference between $\phid$ and $\psimn$ by $\log^{-1}n$, ensuring that $T_x\subseteq \cC_x^{\cM_n ,h}$, the connected component of $x$ in $\Enh$. 
%
%
\\

\noindent
\textbf{The exploration.} Fix $x\in V_n$. Let 
\begin{equation}\label{eqn:bndef}
b_n:=(d-1)^{-a_n}\log^{-6}n,
\end{equation}
where we recall the definition of $a_n$ from (\ref{eqn:andef}). Let $(\xi_{y})_{y \in \Td}$ be a family of independent variables, each of law $\cN(0,1)$, independent of the pairing of the half-edges in $\cM_n$. Define the GFF $\phid$ as in Proposition~\ref{prop:recursivegfftrees}.
\\
At every step of the exploration, $T_x $ will be a tree rooted at $x$, $\fT_x$ its respective counterpart in $\Td$, rooted at $\circ$, and $\Phi$ an isomorphism from $T_x$ to $\fT_x$. At step $k$, we will reveal the $k$-th generation of $T_x$ and $\fT_x$.
\\
Precisely, the \textbf{exploration from $x$} consists of the following steps: 
\begin{itemize}
\item at step $0$, $T_x=\{x\}$ and $\fT_x=\{\circ\}$. Reveal the pairings of the half-edges of $B_{\cM_n}(x,a_n)$. 
Stop the exploration if $\tx(B_{\cM_n}(x,a_n))>0$ or if $\phid(\circ)< h+\log^{-1}n$.

\item at step $k\geq 1$, reveal the edges of $B_{\cM_n}(T_x,a_n+1)$ that were not known at step $k-1$. Let $O_{k-1}$ be the set of the vertices of $T_x$ of height $k-1$. Stop the exploration if at least one of the following conditions holds:
\begin{enumerate}[label=C\arabic*]
\item\label{C1} a cycle appears in $B_{\cM_n}(T_x, a_n+1)$,

\item\label{C5} $\vert O_{k-1} \vert \geq n^{1/2}b_n$,

\item\label{C4} $O_{k-1} =\emptyset$ (i.e. no vertex was added to $T_x$ during the $(k-1)$-th step),
\item\label{C6} $k>\log_{\lambda_h}n$.
\end{enumerate} 
Else, denote $x_{k,1}, x_{k,2}, \ldots, x_{k,m}$ the neighbours (in $\cM_n$) of vertices of $O_{k-1}$ that are not in $O_{k-2}$, for some $m\in \dN$ (note that $m=(d-1)\vert O_{k-1}\vert$, each vertex of $O_{k-1}$ having one neighbour in $O_{k-2}$, its parent, and $d-1$ other neighbours at distance $k$ of $\circ$). 
%
%
Add to $T_x$ the vertices $x_{k,i}$ of $O_{k-1}$ such that $\phid(\Phi(x_{k,i}))\geq h+\log^{-1}n$. Add to $\kT_x$ the corresponding vertices $\Phi(x_{k,i})$.
\end{itemize}

\includegraphics[scale=1]{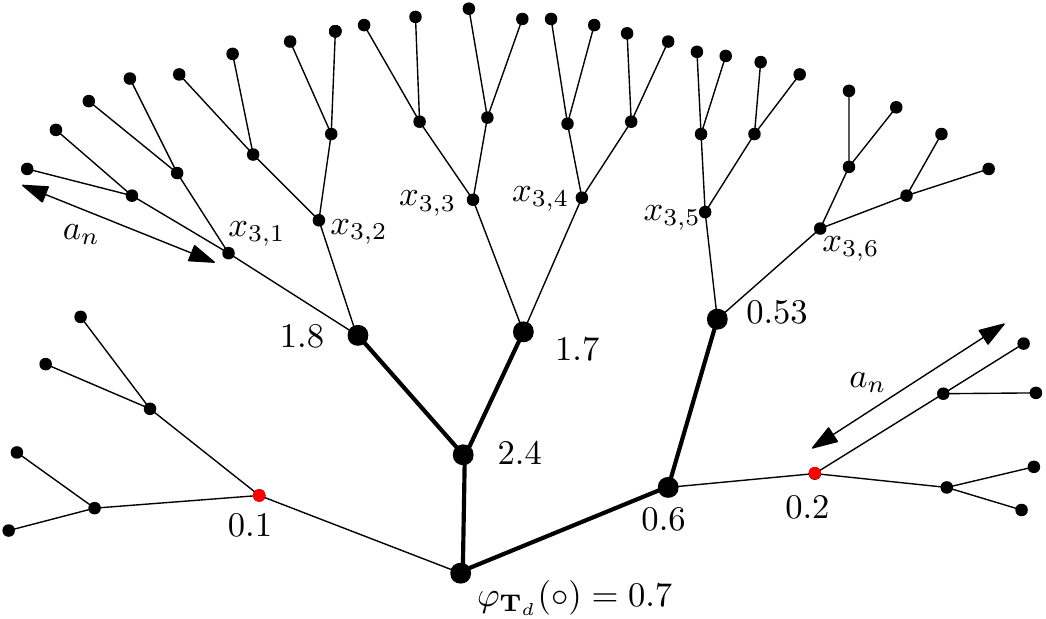}

\begin{center}
Figure 2. Illustration of the exploration with $k=3$, $a_n=2$, $h=0.3$ and $n=148$ (so that $\log^{-1}n\simeq 0.2$). Thick vertices and edges represent $T_x$ after two steps. Red vertices have not been included in $T_x$ because $\phid$ at their counterparts in $\Td$ is below $h+\log^{-1}n$. Any number near a vertex $v$ is $\phid(\Phi(v))$.
\end{center}
\vspace{2mm}

\noindent
If the exploration is stopped at some step $k$, at which only \ref{C5} is met, say that it is \textbf{successful}. In this case, by Proposition~\ref{prop:gffann}, we can sample $\psimn$ as follows: we reveal the remaining pairings of half-edges in $\cM_n$. We set $\psimn(x)=\xi_{\circ}\gmn(x,x)$. For all $k,i\geq 1$, if $A_{k,i}=\{x\}\cup\{x_{\ell,j}\vert (\ell,j) \prec (k,i)\}$ where $\prec$ is the lexicographical order on $\dN^2$, let
\begin{equation}\label{eqn:recursivegffgraph}
\psimn(x_{k,i})= \dE^{\cM_n}[\psimn(x_{k,i})\vert \sigma(A_{k,i})] + \xi_{k,i}\sqrt{\text{Var}(\psimn(x_{k,i})\vert \sigma(A_{k,i}))}.
\end{equation}
\noindent
Note that the conditional expectation and variance of the RHS are $\sigma(A_{k,i})$-measurable random variables, so that (\ref{eqn:recursivegffgraph}) makes sense, even if $\psimn(x_{k,i})$ appears on both sides of the equation.

\noindent
Let $\cS(x):=\{\text{the exploration from $x$ is successful}\}$. We prove the following:
\begin{proposition}\label{prop:explo1vertex}
\begin{equation}\label{eqn:explo1vertexsuccessful}
\dP_{ann}(\cS(x)\cap \{T_x\subseteq \cC_x^{\cM_n,h} \})\underset{n\rightarrow +\infty}{\longrightarrow} \eta(h).
\end{equation}
\end{proposition}

\begin{remark}[\textbf{Exploration size}]\label{rem:explorationsize}
Denote $R_x$ the set of vertices \emph{\textbf{seen}} during the exploration (i.e. such at least one of their half-edges has been paired). Note that for $n$ large enough, for every $x\in V_n$, by \ref{C5}, \ref{C6} and (\ref{eqn:bndef}), $T_x$ contains less than 
\begin{center}
$n^{1/2}(d-1)^{-a_n}\log^{-6}n\log_{\lambda_h}n\leq n^{1/2}(d-1)^{-a_n}\log^{-4}n$
\end{center}
vertices, so that 
\begin{center}
$\vert R_x\vert \leq n^{1/2}(d-1)^{-a_n}\log^{-4}n\times (1+(d-1)+\ldots+ (d-1)^{a_n+1})\leq  n^{1/2}\log^{-3}n$.
\end{center}
\end{remark}
\noindent
In order to prove Proposition~\ref{prop:explo1vertex}, we first show that $\Chlnplus$ either has an exponential growth at rate $>\sqrt{\lambda_h}$ with probability close to $\eta(h)$, or dies out before reaching height $\log_{\lambda_h}n$ with probability close to $1-\eta(h)$. Although the proof is slightly technical, it relies merely on Proposition~\ref{prop:thm43adapted} and on the continuity of the maps $h'\mapsto \lambda_{h'}$ and $h'\mapsto \eta(h')$. It can be skipped at first reading. Recall the definition of $\cZ_k^h$ at the beginning of Section~\ref{subsec:expogrowthCh}.

\begin{lemma}\label{lem:explo1vertex}
Let $\mathcal{F}^{(n)}_k:=\{n^{1/2}b_n\leq \vert\cZ_k^{h+\log^{-1}n}\vert\leq dn^{1/2}b_n\}$ and $\mathcal{F'}^{(n)}_k:=\{\cZ_{k-1}^{h+\log^{-1}n}=\emptyset\}$ for $k\geq 1$.
Let $k_0:= \inf \{k\geq 1, \,  \mathcal{F}^{(n)}_{k}\text{ or }\mathcal{F'}^{(n)}_{k} \text{ happens}\}$, $\cF^{*(n)}:=\{k_0\leq \log_{\lambda_h}n\}\cap \cF^{(n)}_{k_0}$ and $\mathcal{F'}^{*(n)}:=\{k_0\leq \log_{\lambda_h}n\}\cap \mathcal{F'}^{(n)}_{k_0}$. 
Then, as $n\rightarrow +\infty$:
\begin{equation}\label{eqn:growthtree}
\dP_{ann}(\cF^{*(n)})\rightarrow \eta(h) \,\text{ \emph{and} }\, \dP_{ann}(\mathcal{F'}^{*(n)})\rightarrow 1-\eta(h).
\end{equation}
\end{lemma}

\begin{proof} Remark first that by construction, $\dP_{ann}$ acts like $\dP_{\Td}$ on events that only depend on $\phid$.  Note that for every $n\geq 1$, $\mathcal{F'}^{*(n)}\cap  \cF
^{*(n)}=\emptyset$, implying $ \dP_{ann}(\mathcal{F'}^{*(n)})+ \dP_{ann}(\cF
^{*(n)})\leq 1$. Hence it is enough to prove that
\begin{equation}\label{eqn:Fstarn}
\underset{n\rightarrow +\infty}{\liminf}\, \dP_{ann}(\cF
^{*(n)})\geq \eta(h)
\end{equation}
and
\begin{equation}\label{eqn:Fstarprimen}
\liminf_{n\rightarrow +\infty} \dP_{ann}(\mathcal{F'}^{*(n)}) \geq 1-\eta(h)
\end{equation}

\noindent
Let $\varepsilon \in (0, \eta(h))$. Let $\delta >0$ be such that $\vert \eta(h+\delta) -\eta(h)\vert\leq \varepsilon$ and $\log{\lambda_{h+\delta}}>(9\log{\lambda_h})/10$. The map $h'\rightarrow \eta(h')$ is continuous on $\dR\setminus\{h_{\star}\}$ by Theorem 3.1 of~\cite{ACregultrees} and the map $h'\rightarrow \lambda_{h'}$ is an homeomorphism from $(-\infty,h_{\star})$ to $(1,d-1)$ by Proposition~\ref{prop:thm43adapted}, hence such $\delta$ exists. It is clear that
\begin{align*}
\liminf_{n\rightarrow +\infty} \dP_{ann}\left (\exists k\leq \log_{\lambda_h}n, \vert\cZ_k^{h+\log^{-1}n}\vert >n^{1/2}b_n \right)&\hspace{-1mm}\geq \liminf_{n\rightarrow +\infty} \dP_{ann}\left (\exists k\leq \log_{\lambda_h}n, \vert\cZ_k^{h+\delta}\vert \hspace{-0.5mm}>n^{1/2}b_n \right)
\\
&\hspace{-1mm}\geq \liminf_{n\rightarrow +\infty} \dP_{ann}\left ( \vert\cZ_{\lfloor \log_{\lambda_h}n\rfloor}^{h+\delta}\vert \hspace{-0.5mm}>n^{1/2}b_n \right)
\\
&\hspace{-1mm}\geq \liminf_{n\rightarrow +\infty} \dP_{ann}\left ( \vert\cZ_{\lfloor \log_{\lambda_h}n\rfloor}^{h+\delta}\vert \hspace{-0.5mm}>n^{9/10}/\log_{\lambda_h}^2n \right),
\end{align*}
hence by the first equation of Proposition~\ref{prop:thm43adapted} applied to $h+\delta$,
\begin{equation}\label{eq:expogrowthhdelta}
\liminf_{n\rightarrow +\infty} \dP_{ann}\left (\exists k\leq \log_{\lambda_h}n, \vert\cZ_k^{h+\log^{-1}n}\vert >n^{1/2}b_n \right)\geq \eta(h+\delta)\geq \eta(h)-\varepsilon.
\end{equation}
Since each vertex of $\Td$ has at most $d$ children, we have $\vert\cZ_k^{h+\log^{-1}n}\vert \hspace{-1mm}\leq \hspace{-1mm} d\vert\cZ_{k-1}^{h+\log^{-1}n}\vert$ deterministically for all $k\geq 1$. Hence, letting $k':=\inf\{k\geq 0,\vert \cZ_k^{h+\log^{-1}n}\vert \geq n^{1/2}b_n\}$ when this set is non-empty, $\cF^{(n)}_{k'}$ holds, so that $\{\exists k\leq \log_{\lambda_h}n, \vert\cZ_k^{h+\log^{-1}n}\vert >n^{1/2}b_n\}\subseteq \cup_{k\leq \log_{\lambda_h} \hspace{-1mm} n}\,\cF^{(n)}_k$. Thus 
\begin{center}
$\underset{n\rightarrow +\infty}{\liminf}\, \dP_{ann}(\cF
^{*(n)})\geq \underset{n\rightarrow +\infty}{\liminf}\, \dP_{ann}\left (\cup_{k\leq \log_{\lambda_h} \hspace{-1mm} n}\,\cF^{(n)}_k \right)\geq \eta(h)-\varepsilon$
\end{center}
by (\ref{eq:expogrowthhdelta}), and this shows (\ref{eqn:Fstarn}). 
\\
For $n\geq e^{1/\delta}$, $\Chdplus \subseteq \Chlnplus\subseteq \Ch$. Note that for $n\geq 1$, 
\begin{align*}
 &\dP_{ann}(\mathcal{F'}^{*(n)})\hspace{-1mm}
 \\
&\geq \dP_{ann}\hspace{-1mm}\left(\cZ_{\lfloor\log_{\lambda_h}n\rfloor-1}^{h+\log^{-1}n}=\emptyset\right)\hspace{-1mm}-\hspace{-1mm} \dP_{ann}\hspace{-1mm}\left(\{\cZ_{\lfloor\log_{\lambda_h}n\rfloor-1}^{h+\log^{-1}n}=\emptyset \} \cap \{\exists k\geq 1,\, \vert \cZ_k^{h+\log^{-1}n}\vert \geq n^{1/2}b_n \}\right)
\\
&\geq  \dP_{ann}\hspace{-1mm}\left(\cZ_{\lfloor\log_{\lambda_h}n\rfloor-1}^{h}=\emptyset\right)\hspace{-1mm}- \dP_{ann}\left(\{\vert \cC_{\circ}^{h+\log^{-1}n}\vert <+\infty \}  \cap \{\exists k\geq 1,\, \vert \cZ_k^{h+\log^{-1}n}\vert\hspace{-1mm} \geq n^{1/2}b_n \} \right)
\\
&\geq  \dP_{ann}\hspace{-1mm}\left(\cZ_{\lfloor\log_{\lambda_h}n\rfloor-1}^{h}=\emptyset\right) -  \dP_{ann}\left(\vert \cC_{\circ}^{h+\log^{-1}n}\vert <+\infty \,\,\vert\,\,\exists k\geq 1,\, \vert \cZ_k^{h+\log^{-1}n}\vert \geq n^{1/2}b_n\right).
\end{align*}

\noindent
The first term of the RHS converges to $1-\eta(h)$ as $n\rightarrow +\infty$. For any $k\geq 1$ and for any $v\in B_{\Td}(\circ,k)$, denoting $T_v$ the possible subtree from $v$ in $\cC_{\circ}^{h+\log^{-1}n}$ (if $v\in\cZ_k^{h+\log^{-1}n}$) and $\cC_{\circ}(h,\delta)$ the connected component of $\circ$ in $(\{\circ\}\cup E_{\phid}^{\geq h+\delta})\cap \Td^+$,
\begin{center}
$ \dP_{ann}(\vert T_v\vert <+\infty  \,\vert v\in\cZ_k^{h+\log^{-1}n} ) \leq \dP^{\Td}_{h}(\vert\cC_{\circ}(h,\delta)\vert <+\infty)$
\end{center} 
by Lemma~\ref{lem:monotonicityphid}, independently of the other vertices of $ \cZ_k^{h+\log^{-1}n}$. Thus, 
$$
 \dP_{ann}(\vert \Ch\vert <+\infty \,\,\vert\,\,\exists k\geq 1,\, \vert \cZ_k^{h}\vert \geq n^{1/2}b_n)\leq \dP^{\Td}_{h}(\vert\cC_{\circ}(h,\delta)\vert <+\infty)^{n^{1/2}b_n}.$$
\noindent
By a straightforward adaptation of Remark~\ref{rem:uniformsupercriticsurvival}, we have  $\dP^{\Td}_{h}(\vert\cC_{\circ}(h,\delta)\vert <+\infty)<1$, and \eqref{eqn:Fstarprimen} follows.
\end{proof}

\begin{proof}[Proof of Proposition~\ref{prop:explo1vertex}.]
We first establish that \ref{C1} happens with $\dP_{ann}$-probability $o(1)$. Then, if there is no cycle in $B_{\cM_n}(T_x,a_n)$, we can apply Proposition~\ref{prop:couplinggffsexplo}, to bound the difference between $\phid$ and $\psimn$.

\noindent
By Remark~\ref{rem:explorationsize}, at most $dn^{1/2}\log^{-3}n$ matchings of half-edges are performed during the exploration. By (\ref{eqn:binomdominationdiff}) with $k=m_0=1$, $m_E=0$ and $m\leq dn^{1/2}\log^{-3}n$, the probability to create at least one cycle during these matchings is less than $\log^{-1}n$ for large enough $n$. Therefore, 
\begin{equation}\label{eqn:C1happensexplo}
\dP_{ann}(\text{\ref{C1} happens})\rightarrow 0.
\end{equation}

\noindent
Note that if \ref{C1} does not happen, then on $\cF^{*(n)}$, (resp. $\mathcal{F'}^{*(n)}$), \ref{C4} (resp. \ref{C5}) is satisfied, but not \ref{C6}. Moreover, on $\cF^{*(n)}$, (resp. $\mathcal{F'}^{*(n)}$), \ref{C5} (resp. \ref{C4}) does not hold, so that we have 
$\dP_{ann}(\cS(x))\geq \dP_{ann}(\cF^{*(n)})-\dP_{ann}(\text{\ref{C1} happens}) 
$. Since $\mathcal{F}^{*(n)}\cap \mathcal{F'}^{*(n)} =\emptyset$, we get that

\begin{center}
$\dP_{ann}(\cS(x))
\leq \dP_{ann}((\mathcal{F}^{*(n)})^c)+\dP_{ann}(\text{\ref{C1} happens})\leq 1- \dP_{ann}(\mathcal{F'}^{*(n)})+\dP_{ann}(\text{\ref{C1} happens}).$
\end{center}
\noindent
Thus, by (\ref{eqn:growthtree}) and (\ref{eqn:C1happensexplo}),
\begin{equation}\label{eqn:growthsuccessful}
\dP_{ann}(\cS(x))\rightarrow \eta(h).
\end{equation}

\noindent
Suppose now that the exploration is over and that $\cS(x)$ holds. We compare $\psimn$ with $\phid$. Note that by \ref{C6}, $T_x$ has a maximal height $\log_{\lambda_h}n$ so that by the triangle inequality, 
\begin{center}
$\{ T_x\not\subseteq \cC_x^{\cM_n,h}\}\subseteq \{\exists y\in T_x,\, \vert \psimn(y)-\phid(\Phi(y)) \vert \geq \log^{-1}n\} \subseteq \cup_{y\in T_x}\cE(y)$,
\end{center}
\noindent
where $\cE(x):=\{\vert\psimn(x)-\phid(\Phi(x))\vert\geq \log^{-3}n\}$ and 
\begin{center}
$\cE(y):=\{\vert \psimn(y)-\phid(\Phi(y))\vert \geq \vert \psimn(\overline{y})-\phid(\overline{\Phi(y)})\vert +2\log^{-3}n\}$ for $y\neq x$.
\end{center}
Suppose that $\cM_n$ is a good graph. For $x_{k,i}\in T_x\setminus\{x\}$, we can apply Proposition~\ref{prop:couplinggffsexplo} on the event  $\cE_{k,i}^{(n)}:=\{\max_{y'\in A_{k,i}}\vert\psimn(y')\vert\hspace{-1mm}<\hspace{-1mm}\log^{2/3}\hspace{-0.5mm}n~\}$,  (since $\tx(B_{\cM_n}(A_{k,i},a_n))=\tx(A_{k,i})$ by \ref{C1} and since $\vert A_{k,i}\vert\leq n^{1/2}$ by Remark~\ref{rem:explorationsize}). Writing $y=x_{k,i}$ and $\xi=\xi_{\Phi(x_{k,i})}$, we get for $n$ large enough:
\begin{align*}
\vert \psimn(y)-\phid(\Phi(y))\vert \leq &\lv \frac{\psimn(\overline{y})-\phid(\overline{\Phi(y)})}{d-1}\rv +\log^{-3}n
\\
&+ \lv\left(\sqrt{\text{Var}^{\cM_n}(\psimn(y)\vert \sigma(A_{k,i}))}-\sqrt{\frac{d-1}{d}}\right)\xi\rv 
\end{align*}

\noindent
and 
\begin{align*}
\lv\sqrt{\text{Var}^{\cM_n}(\psimn(y)\vert \sigma(A_{k,i}))}-\sqrt{\frac{d-1}{d}}\rv&\leq \lv \sqrt{\frac{d-1}{d}-\log^{-4}n}-\sqrt{\frac{d-1}{d}}\rv 
\\
&\leq \frac{11}{10}\sqrt{\frac{d-1}{d}}\frac{d}{2(d-1)}\log^{-4}n
\\
&\leq\log^{-4}n.
\end{align*}
\noindent
Let $\mathcal{E}'(y):= \cE(y)\cap\cE_{k,i}^{(n)}$. We have
\begin{equation}\label{eqn:diffgffgraphtree}
\dP^{\cM_n}(\mathcal{E}'(y))\leq \dP(\vert \xi\vert\log^{-4}n\geq \log^{-3}n)\leq n^{-3}
\end{equation}
by the exponential Markov inequality used as in the proof of Lemma~\ref{lem:maxgff}. Moreover, by (\ref{eqn:greenfunctionGnapprox1}), if $\kappa$ is large enough, we obtain in the same manner:
\begin{equation}\label{eqn:diffgffgraphtreeroot}
\dP^{\cM_n} (\cE(x))\leq \dP^{\cM_n}(\vert \xi_{\circ}\vert \log^{-4}n\geq \log^{-3}n)\leq n^{-3}.
\end{equation}
By Remark~\ref{rem:explorationsize}, we have $\vert T_x\vert \leq n^{1/2}$. By (\ref{eqn:diffgffgraphtree}), (\ref{eqn:diffgffgraphtreeroot}) and a union bound on $y\in T_x$, we have that $\dP^{\cM_n}(\cup_{y\in T_x}\cE'(y))\leq n^{-5/2}$ with $\cE'(x):=\cE(x)$. And $\dP^{\cM_n}(\cup_{(k,i):x_{k,i}\in T_x}\cE^{(n)}_{k,i})\leq n^{-2}$ for large enough $n$, by Lemma~\ref{lem:maxgff}.
\\
Thus for $n$ large enough, if $\cM_n$ is a good graph,
\[
\dP^{\cM_n}(T_x\not\subseteq \cC_x^{\cM_n,h})\leq \dP^{\cM_n}(\cup_{y\in T_x}\cE(y))\leq n^{-5/2}+n^{-2}\leq n^{-1},
\]
\noindent
so that by Proposition~\ref{prop:goodgraph} and (\ref{eqn:growthsuccessful}):
\begin{center}
$\dP_{ann}(\cS(x)\cap \{T_x\subseteq \cC_x^{\cM_n,h} \})\rightarrow\eta(h).$
\end{center}
\end{proof}

\subsection{Aborted exploration}\label{subsec:explo1vertexaborted}
\noindent
For $x\in V_n$, the \textbf{lower exploration} is the exploration of Section~\ref{subsec:explo1vertexsuccess}, modified by replacing $h+\log^{-1}n$ by $h-\log^{-1}n$, so that we compare $T_x$ and $\Chlnminus$. If it is stopped at some step $k$ at which only \ref{C4} is met, say that it is \textbf{aborted}. Write
\\
$\cA(x):=\{\text{the lower exploration from $x$ is aborted}\}\cap \{ \cC_x^{\cM_n,h}\subseteq T_x\}$.

\begin{proposition}\label{prop:explo1vertexaborted}
\begin{equation}\label{eqn:explo1vertexaborted}
\dP_{ann}(\cA(x))\underset{n\rightarrow +\infty}{\longrightarrow} 1-\eta (h).
\end{equation}
\end{proposition}

\noindent
The proof follows from a direct adaptation of Lemma~\ref{lem:explo1vertex} and Proposition~\ref{prop:explo1vertex}. Note in particular that $\dP_{ann}(\text{\ref{C1} happens})=o(1)$, and that $\dP_{ann}(\cA(x))=\dP_{ann}(\cZ_{k_0-1}^{h-\log^{-1}n}=\emptyset)+o(1)=1-\eta(h)+o(1)$.

\section{Existence of a giant component}\label{sec:giant}
\noindent
In Section~\ref{subsec:connexion}, we prove that two vertices $x,y\in V_n$ are in the same connected component of $\Enh$ with $\dP_{ann}$-probability $\underset{n\rightarrow +\infty}{\longrightarrow}\eta(h)^2$ (Proposition~\ref{prop:exploendjoin}). Then in Section~\ref{subsec:averageconnec}, we use a second moment argument to get concentration and to show (\ref{eqn:mainthm}).

\subsection{Connecting two successful explorations}\label{subsec:connexion}
\noindent
Let us describe our strategy to establish Proposition~\ref{prop:exploendjoin}. We perform explorations as in Section~\ref{subsec:explo1vertexsuccess} from $x$ and $y$. If they are both successful and do not meet (which happens with probability $\simeq \eta(h)^2$), we develop disjoint balls, denoted ``joining balls'', from $\partial T_x$ to $\partial T_y$ (Section~\ref{subsubsec:jointexplo}). Each of them is rooted at a vertex of $\partial T_x$, hits $\partial T_y$ at exactly one vertex, and has a ``security radius'' of depth $a_n$ around its path from $\partial T_x$ to $\partial T_y$ (see Figure 3). The construction of the joining balls only depends on the structure of $\cM_n$, and not on the values of $\psimn$. Then, we realize $\psimn$ on $T_x, T_y$ and those balls (Section~\ref{subsubsec:jointexploGFF}). If they are all disjoint and tree-like, once we have revealed $\psimn$ on $T_x$ and $T_y$, this security radius allows us to apply Proposition~\ref{prop:couplinggffsexplo} to approximate $\psimn$ on the paths from $\partial T_x$ to $\partial T_y$ by $\phid$.
\\
Let us explain with a back-of-the-envelope computation how the joining balls allow to connect $T_x$ and $T_y$ in $\Enh$. Since $\vert T_y\vert\simeq n^{1/2}b_n$ by \ref{C5}, the probability that for a given $z\in \partial T_x$, exactly one of the vertices at distance $\lfloor\gamma\log_{d-1}\log n\rfloor$ (and no vertex at distance $< \lfloor\gamma\log_{d-1}\log n\rfloor$) from $z$ is in $\partial T_y$ is
\begin{center}
$\simeq \dP(\text{Bin}((d-1)^{\gamma\log_{d-1}\log n},\frac{n^{1/2}b_n}{dn})=1)\simeq \log^{\gamma}n\times n^{-1/2}b_n$.
\end{center}
And there are $\simeq n^{1/2}b_n$ vertices in $\partial T_x$, hence we can expect that that the number of joining balls is at least $\simeq n^{1/2}b_n\times \log^{\gamma}n\times n^{-1/2}b_n= b_n^2\log^{\gamma}n$, provided that we can control some undesirable events (such as an intersection between balls, or a cycle in a ball). This is the purpose of Lemma~\ref{lem:joiningball}. 
\\
Moreover, we know that for large $r\in \dN$ and $v\in \partial B_{\Td}(\circ,r)$, $\dP^{\Td}(v\in \Ch)$ is of order $(\lambda_h/(d-1))^r$, by Proposition~\ref{prop:Chlargedevgrowthrate}. Taking $r=\gamma \log_{d-1}\log n$, the probability that $\Enh$ percolates from $\partial T_x$ to $\partial T_y$ through a given joining ball is $\simeq\log^{\gamma(\log_{d-1}\lambda_h\,-1)}n$, if we can approximate $\psimn$ by $\phid$. For $\gamma$ large enough w.r.t $\kappa$ (recall (\ref{eqn:bndef}) and (\ref{eqn:andef}), and recall that $\lambda_h>1$), 
\begin{center}
$b_n^2\log^{\gamma}n\times\log^{\gamma(\log_{d-1}\lambda_h\,-1)}n\geq \log^{-2\kappa-13}n \,\times\log^{\gamma\log_{d-1}\lambda_h}n >>1$,
\end{center}
so that with high probability, $\Enh$ percolates through at least one joining ball from $\partial T_x$ to~$\partial T_y$. 

\subsubsection{The joint exploration}\label{subsubsec:jointexplo}
For $x,y\in V_n$, write $x\overset{h}{\leftrightarrow}y$ if $y\in \cC_x^{\cM_n, h}$. Let $(\xi_{z,v})_{z\in \{x,y\},v\in \Td}$ be an array of i.i.d. standard normal variables independent from everything else. Define the \textbf{joint exploration from $x$ and $y$} as the exploration from $x$ (with the $(\xi_{x,v})$'s), then the exploration from $y$ (with the $(\xi_{y,v})$'s), as in Section~\ref{subsec:explo1vertexsuccess}, with the additional condition
\begin{enumerate}[label=C5]
\item \label{C7} the exploration is stopped as soon as $R_x\cap R_y\neq \emptyset$,
\end{enumerate}
\noindent
where $R_x$ (resp. $R_y$) is the set of vertices seen during the exploration from $x$ (resp. from $y$), as defined in Remark~\ref{rem:explorationsize}. Note that the families $(\xi_{x,v})_{v\in \Td}$ and $(\xi_{x,v})_{v\in \Td}$ generate two independent copies of $\phid$.
\\
\\
If both explorations are successful and \ref{C7} does not happen (denote $\mathcal{S}(x,y)$ this event), we add the following steps to the joint exploration. Let $\gamma>0$ and 
\begin{equation}\label{eqn:aprimendef}
a'_n:=\lfloor \gamma\log_{d-1}\log n\rfloor. 
\end{equation}
Denote $z_1, \ldots, z_{\vert \partial T_x\vert}$ the vertices of $\partial T_x$. 
For $j=1, 2, \ldots, \vert \partial T_x\vert$ successively, build $B^*(z_j,a'_n)$ the subgraph of $\cM_n$ obtained as follows (see Figure 3 for an illustration). Write
\begin{equation}\label{eqn:BjRjdef}
B^*_j:=\cup_{j'<j}B^*(z_{j'},a'_n)\text{ and }Q_j:=R_x\cup R_y \cup B^*_j,
\end{equation}
so that $Q_j$ is the set of vertices seen in the exploration before the construction of $B^*(z_j,a'_n)$.
\\
Let initially $B^*(z_j,a'_n)$ be the subtree from $z_j$ of height $a_n$ in the tree $B_{\cM_n}(T_x,a_n)$ (in blue in Figure 3). For $k\leq a'_n$, write $B^*(z_j,k):=B^*(z_j,a'_n)\cap B_{\cM_n}(z_j,k)$. If $B^*(z_j, a_n)\cap B^*_j\neq \emptyset$, say that $j$ is \textbf{spoiled}, and the construction of $B^*(z_j,a'_n)$ stops.
\\
Else, for $k=a_n, a_n+1, \ldots, a'_n-2a_n-2$ successively, while $\tx(B^*(z_j, k)\cup Q_j) =\tx(Q_j)$ (i.e. no cycle has been discovered) and $B^*(z_j, k)\cap B_{\cM_n}(T_y,a_n)=\emptyset$, add to $B^*(z_j,a'_n)$ the neighbours of $B^*(z_j,k)$ and the corresponding edges (in red in Figure 3). If for some $k\in \{a_n, \ldots, a'_n-2a_n-2\}$, $\tx(B^*(z_j, k)\cup Q_j) >\tx(Q_j)$ (i.e. at least one cycle is discovered) or $B^*(z_j, k)\cap B_{\cM_n}(T_y,a_n)\neq \emptyset$, the construction of $B^*(z_j, a'_n)$ stops.
\\
If the construction has not been stopped for some $k<a'_n-2a_n-1$, add the neighbours of $B^*(z_j, a'_n-2a_n-1)$ to $B^*(z_j,a'_n)$ (also in red in Figure 3). If 
\begin{center}
$\vert B^*(z_j, a'_n-2a_n)\cap B_{\cM_n}(T_y,a_n)\vert \neq 1$, 
\end{center}
the construction of $B^*(z_j, a'_n)$ stops.
\\
Else, let $v_j(0)$ be the unique vertex of $B^*(z_j, a'_n-2a_n)\cap B_{\cM_n}(T_y,a_n)$. If 
\begin{center}
$\tx((B^*(z_j, a'_n-2a_n)\cup Q_j)\setminus\{v_j(0)\}) >\tx(Q_j)$, 
\end{center}
the construction of $B^*(z_j, a'_n)$ stops.
\\
Else, for $k=a'_n-2a_n, \ldots, a'_n-1$ successively, while $\tx(B^*(z_j, k)\cup Q_j) =\tx(B^*(z_j, a'_n-2a_n)\cup Q_j)$, add the neighbours of $B^*(z_j,k)$ to $B^*(z_j,a'_n)$ (in green in Figure 3). Then, the construction of $B^*(z_j, a'_n)$ is completed. In this case only, and if 
\begin{center}
$\tx(B^*(z_j, a'_n)\cup Q_j)=\tx(B^*(z_j, a'_n-2a_n)\cup Q_j)$, 
\end{center}
say that $B^*(z_j,a'_n)$ is a \textbf{joining ball}. In other words, we obtain a joining ball if, revealing the offspring up to generation $a'_n$ of $z_j$, the $(a'_n-2a_n)$ offspring of $z_j$ intersects $\partial T_y$ at a unique vertex $v_j(0)$, and no cycle is discovered in the whole construction (except when $B^*(z_j,a'_n)$ reaches $\partial B_{\cM_n}(T_y,a_n)$ at $v_j(0)$, if $T_x$ and $T_y$ were already connected in $Q_j$ by $B^*(z_{j'},a'_n)$ for some $j'<j$).
\\
Write $J:=\{j\leq \vert \partial T_x\vert, \text{$B^*(z_j,a'_n)$ is a joining ball}\}$ and $\cS'(x,y):=\cS(x,y)\cap \{\vert J\vert \geq \log^{\gamma -3\kappa -18}n\}$ ($\cS(x,y)$ was define above (\ref{eqn:aprimendef})).
\\

\includegraphics[scale=0.8]{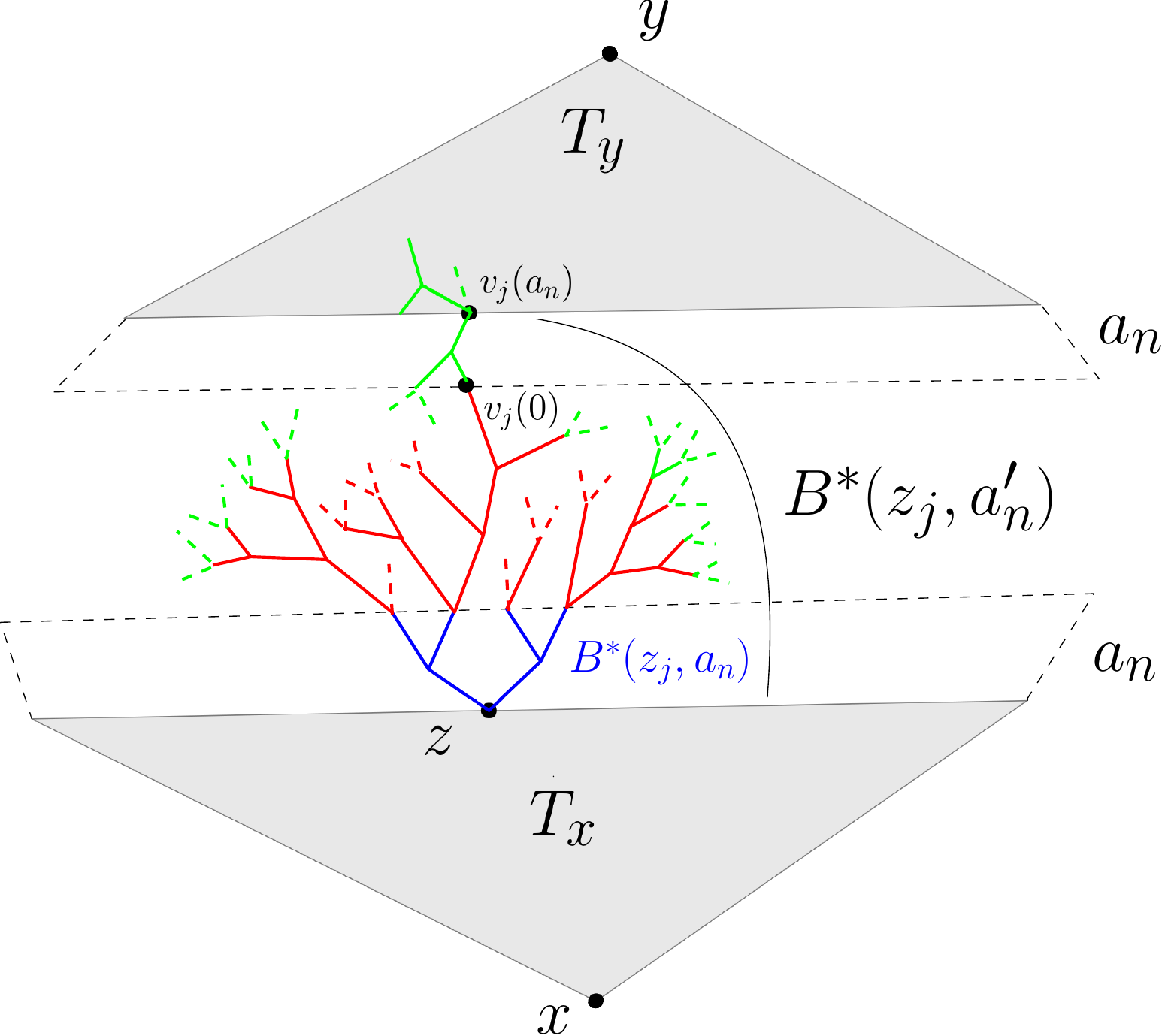}

\begin{center}
Figure 3. Illustration of a joining ball $B^*(z_j,a'_n)$. Here $a_n=2$ and $a'_n=9$. Dashed lines represent subtrees that have not been fully pictured. The blue tree has total height $a_n$, the red trees $a'_n-3a_n$, and the green trees $2a_n$.
\end{center}

\begin{lemma}\label{lem:joiningball}
Fix $\gamma >3\kappa +18$.
$$ \dP_{ann}(\cS'(x,y))\underset{n\rightarrow +\infty}{\longrightarrow}\eta(h)^2.$$
\end{lemma}

\begin{proof}
Denote $\cF^{*(n)}(x)$ (resp. $\cF^{*(n)}(y)$) the event $\cF^{*(n)}$ for $x$ (resp. $y$). Remark that the realization of $\cF^{*(n)}(x)$ (resp. $\cF^{*(n)}(y)$) only depends on the version of $\phid$ defined by $(\xi_{x,v})_{v\in \Td}$ (resp. by $(\xi_{y,v})_{,v\in \Td}$) and not on the pairings of $\cM_n$. Hence, $\cF^{*(n)}(x)$ and $\cF^{*(n)}(y)$ are independent. 
As in the proof of Lemma~\ref{lem:explo1vertex}, we get that 
\begin{center}
$ \dP_{ann}(\cF^{*(n)}(x)\cap \cF^{*(n)}(y))=  \dP_{ann}(\cF^{*(n)}(x)) \dP_{ann}( \cF^{*(n)}(y))\rightarrow \eta(h)^2$ 
\end{center}
and $\dP_{ann}(\mathcal{F'}^{*(n)}(x)\cup \mathcal{F'}^{*(n)}(y))\rightarrow 1-\eta(h)^2$. 
\\
Moreover, $\dP_{ann}(\text{\ref{C1} or \ref{C7} happens})\rightarrow 0$. Indeed, by Remark~\ref{rem:explorationsize}, less than $2dn^{1/2}\log^{-3} n$ half-edges are revealed during the explorations from $x$ and $y$, which allows to control \ref{C7} as we did for \ref{C1} in (\ref{eqn:C1happensexplo}). Thus, 
\begin{align*}
\limsup_{n\rightarrow +\infty}\vert \dP_{ann}(\cS'(x,y)) -\eta(h)^2 \vert \leq & \limsup_{n\rightarrow +\infty}\dP_{ann}(\cS(x,y)\cap\{\vert J\vert < \log^{\gamma-3\kappa-18}n \})
\end{align*}

\noindent


\noindent
and it remains to prove that 
\begin{equation}\label{eqn:joiningballs}
\limsup_{n\rightarrow +\infty}\dP_{ann}(\cS(x,y)\cap\{\vert J\vert <\log^{\gamma-3\kappa-18}n\}) =0.
\end{equation}
We proceed in two steps: in step 1, we control the number of spoiled vertices, and the number of vertices of $\partial B_{\cM_n}(T_y,a_n)$ that are hit when building the $B^*(z_j,a'_n)$'s (if a large proportion of those vertices are in $Q_j$, then it significantly affects the probability that $B^*(z_j,a'_n)$ is a joining ball). In step 2, we estimate the probability that for a given $j$, $B^*(z_j,a'_n)$ is a joining ball, provided that the bounds of step 1 hold. This gives a binomial lower bound for $\vert J\vert$. 
\\
\textbf{Step 1.} By Remark~\ref{rem:explorationsize} and (\ref{eqn:andef}), $\vert \partial B_{\cM_n}(T_x,a_n)\vert+\vert \partial B_{\cM_n}(T_y,a_n)\vert\leq 2n^{1/2}\log^{-1}n$. 
Note also that for every $j\leq \vert \partial T_x\vert$, $B^*(z_j,a'_n)$ contains less than $(d-1)^{a'_n}\leq \log^{\gamma}n$ half-edges. Hence:
\begin{equation}\label{eqn:explototaledges}
\text{at every moment of the exploration, less than $n^{1/2}\log^{\gamma}n$ half-edges have been seen.}
\end{equation}
\noindent
Let 
\begin{equation}\label{eqn:defBstar}
B^*:= \cup_{j\leq \vert \partial T_x\vert}B^*(z_j,a'_n). 
\end{equation}
To reveal the edges of $B^*$, one proceeds to at most $n^{1/2}\log^{\gamma}n$ pairings of half-edges by (\ref{eqn:explototaledges}). Any pairing that results in an edge $e$ between some $B^*(z_j,k)$ and $B_{\cM_n}(T_y,a_n)$ then leads to at most 
\begin{center}
$1+(d-1)+\ldots +(d-1)^{2a_n}\leq 3(d-1)^{2a_n}\leq \log^{2\kappa +1}n \leq \log^{\gamma -1}n$
\end{center}
vertices of $B^*(z_j,a'_n)\cap B_{\cM_n}(T_y,a_n)$, since the construction of $B^*(z_j,a'_n)$ stops if such an edge happens at distance less than $a'_n-2a_n$ of $z_j$ (and recall that we choose $\gamma >3\kappa +18>2\kappa +2$).

\noindent
Thus, by (\ref{eqn:binomdominationatinfinitydiff}) with $k=\lfloor \log^{2\gamma+1}n\rfloor$, $m<n^{1/2}\log^{\gamma}n$ and $m_1+m_0+m_E< n^{1/2}\log^{\gamma}n$ (due to \eqref{eqn:explototaledges}), for $n$ large enough:
\begin{equation}\label{eqn:saturationjoiningballs}
\dP_{ann}(\cS(x,y)\cap\{\vert  B_{\cM_n}(T_y,a_n) \cap B^* \vert \geq \log^{3\gamma}n\})\leq 0.99^{\log^2n}\leq n^{-3}.
\end{equation}

\noindent
Let $N$ be the total number of spoiled vertices. By (\ref{eqn:binomdominationatinfinitydiff}) with the same parameters,
\begin{equation}\label{eqn:spoliationjoiningballs}
\dP_{ann}(\cS(x,y)\cap \{N \geq \log^{3\gamma}n\})\leq n^{-3}.
\end{equation}

\noindent
\textbf{Step 2.} Recall the definition of $B^*_j$ from (\ref{eqn:BjRjdef}). For $j\leq m$, denote
\begin{equation}\label{eqn:defSj}
\cS_j:=\cS(x,y)\cap \{ \vert B_{\cM_n}(T_y,a_n) \cap B^*_j \vert \leq \log^{3\gamma }n\} \cap \{z_j\text{ is not spoiled}\},
\end{equation}
\noindent
and let $\cF_j$ be the sigma-algebra of the whole exploration until $B^*_{j-1}$ has been constructed.
Suppose that for every $j\geq 1$ and every $\cF_j$-measurable event $\cE_j\subseteq \cS_j$, 
\begin{equation}\label{eqn:p1p2p3}
\dP_{ann}(B^*(z_j,a'_n)\text{ is a joining ball }\vert\, \cE_j)\geq n^{-1/2}\log^{\gamma-2\kappa-10}n.
\end{equation}
On $\cE:=\{\vert  B_{\cM_n}(T_y,a_n) \cap B^* \vert < \log^{3\gamma}n\} \cap \{N < \log^{3\gamma}n\}$, the number of $j$'s such that $\cS_j$ holds is at least 
\begin{center}
$\vert \partial T_x\vert - \log^{3\gamma}n\geq n^{1/2}\log^{-\kappa-7}n$
\end{center}
by (\ref{C5}) and (\ref{eqn:bndef}). Thus, if $Z\sim \emph{\text{\emph{Bin}}}\left(\lfloor n^{1/2}\log^{-\kappa-7}n\rfloor,n^{-1/2}\log^{\gamma-2\kappa-10}n\right)$,
\[
\dP_{ann}\left(\cS(x,y)\cap \{\vert J \vert \leq \log^{\gamma-3\kappa-18}n\}  \right)\leq \dP(Z\leq \log^{\gamma-3\kappa-18}n)+\dP_{ann}(\cS(x,y)\cap \cE^c).\]
For large enough $n$, $\dP_{ann}(\cS(x,y)\cap \cE^c)=o(n^{-2})$ by (\ref{eqn:saturationjoiningballs}) and (\ref{eqn:spoliationjoiningballs}). Moreover, one checks easily (using $\gamma >3\kappa +18$ and (\ref{eqn:binomiallittlethings})) that for $n$ large enough, for all $k\leq \log^{\gamma-3\kappa-18}n$:
\begin{align*}
\dP(Z=k)&\leq (\lfloor n^{1/2}\log^{-\kappa-7}n\rfloor)^k (n^{-1/2}\log^{\gamma-2\kappa-10}n)^k(1-n^{-1/2}\log^{\gamma-2\kappa-10}n)^{ n^{1/2}\log^{-\kappa-7}n\,\,/2} 
\\
&\leq \exp\left(k\log(\log^{\gamma-3\kappa-17}n)- (\log^{\gamma-3\kappa-17}n)/2 \right)
\\
&\leq 1/n.
\end{align*}
This yields (\ref{eqn:joiningballs}). Hence, it only remains to prove (\ref{eqn:p1p2p3}).

Remark that $\dP_{ann}(B^*(z_j,a'_n)\text{ is a joining ball }\vert \, \cE_j) \geq p_1p_2p_3$ where:
\begin{itemize}
\item $p_1:=\dP_{ann}(\widehat{\cE}_1\vert \cE_j)$ and $\widehat{\cE}_1:={\cE_j}\cap\{$no cycle is created and no connection to $Q_j$ is made when revealing $B^*(z_j,a'_n-2a_n-1)\}$,

\item $p_2:=\dP_{ann}(\widehat{\cE}_2\vert \widehat{\cE}_1)$ where $\widehat{\cE}_2:={\widehat{\cE}_1}\cap\{$exactly one edge connects $B^*(z_j,a'_n-2a_n-1)$ and $D:=\partial B_{\cM_n}(T_y,a_n)\setminus \{B_{\cM_n}(B_{\cM_n}(T_y,a_n)\cap B^*_j,2a_n)\}\,\}\cap\{$no cycle is created and no connection to $\partial B_{\cM_n}(T_y,a_n)\cup B^*_j$ is made when revealing the other edges of $B^*(z_j,a'_n-2a_n)\}$,

\item $p_3:=\dP_{ann}(\widehat{\cE}_3\vert \widehat{\cE}_2)$ where we set $\widehat{\cE}_3:={\widehat{\cE}_2}\cap\{$no cycle is created and no connection to $\partial B_{\cM_n}(T_y,a_n)\cup B^*_j$ is made when revealing the remaining edges of $B^*(z_j,a'_n)\}$.
\end{itemize}
\noindent
This definition of $D$ guarantees that $B^*(z_j,a'_n)$ will not intersect a previously realized joining ball when growing the subtree from $v_j(0)$ in $B_{\cM_n}(T_y,a_n)$.

\noindent 
(\ref{eqn:binomdominationdiff}) with $k=1$, $m_0,m\leq \log^{\gamma}n$ and $m_E,m_1\leq n^{1/2}\log^{\gamma}n$ due to (\ref{eqn:explototaledges}) yields for $n$ large enough:
\begin{equation}\label{eqn:p1p3}
p_i\geq 1-C(1)\log^{\gamma}n\,\frac{\max(n^{1/2}\log^{\gamma}n,2\log^{\gamma}n)}{n}\geq 1-n^{-1/3}
\end{equation}
for $i\in \{1,3\}$. Therefore, $p_1p_3\geq 1/2$ for $n$ large enough.
\\
On $\widehat{\cE}_1$, reveal the pairings of the half-edges of $\partial B^*(z_j,a'_n-2a_n-1)$ one by one. $\widehat{\cE}_2$ holds if:
\begin{itemize}
\item a given half-edge is matched to a half-edge of $D$, which has probability at least $ \frac{\vert D\vert}{dn}$, and
\item each other half-edge is matched to a half-edge that had not been seen before (by (\ref{eqn:explototaledges}), for each half-edge this happens with probability at least $1-\frac{n^{1/2}\log^{\gamma}n}{dn-n^{1/2}\log^{\gamma}n}\geq 1-\frac{n^{1/2}\log^{\gamma}n}{n} $). 
\end{itemize}
Since $ \partial B^*(z_j,a'_n-2a_n-1)$ has $(d-1) \vert\partial B^*(z_j,a'_n-2a_n-1)\vert$ unpaired half-edges, 
\begin{center}
$p_2\geq (d-1)\vert \partial B^*(z_j,a'_n-2a_n-1) \vert \frac{\vert D\vert}{dn} \left(1-\frac{n^{1/2}\log^{\gamma}n}{n}\right)^{\vert\partial B^*(z_j,a'_n-2a_n-1)\vert-1} $ 
\end{center}
\noindent
By (\ref{eqn:andef}) and (\ref{eqn:aprimendef}), one checks easily that on $\widehat{\cE}_1$,
\begin{center}
$ \log^{\gamma -2\kappa -1}n\leq  \vert\partial B^*(z_j,a'_n-2a_n-1)\vert\leq \log^{\gamma}n$,
\end{center}
and that on $\cS_j$ (defined in (\ref{eqn:defSj})), 
\begin{center}
$\vert D\vert\geq \frac{\vert \partial B_{\cM_n}(T_y,a_n)\vert}{2}\geq n^{1/2}\log^{-7}n$
\end{center}
by (\ref{eqn:bndef}) and \ref{C5}. Hence for $n$ large enough,
\begin{align*}
p_2&\geq (d-1)\log^{\gamma -2\kappa -1}n \,\,\frac{n^{1/2}\log^{-7}n}{dn} \left(1-\frac{n^{1/2}\log^{\gamma}n}{dn}\right)^{\log^{\gamma}n}
\geq \frac{1}{2}n^{-1/2}\log^{\gamma-2\kappa -9}n.
\end{align*}
\noindent
With (\ref{eqn:p1p3}), this entails for $n$ large enough  (uniformly on $j$ and on $\cE_j$):
\begin{center}
$\dP_{ann}(B^*(z_j,a'_n)\text{ is a joining ball } \vert\, \cS_j)\geq p_1p_2p_3\geq n^{-1/2}\log^{\gamma-2\kappa-10}n$.
\end{center}
Then (\ref{eqn:p1p2p3}) follows, so that the proof of the Lemma is complete.
\end{proof}

\subsubsection{The field $\psimn$ on the joint exploration}\label{subsubsec:jointexploGFF}
\noindent
Suppose that we are on $\cS'(x,y)$. By Proposition~\ref{prop:gffann}, we can realize $\psimn$ on $T_x$ as in (\ref{eqn:recursivegffgraph}) with the $(\xi_{x,v})_{v\in \Td}$ (hence we first reveal the remaining edges of $\cM_n$). Then we can realize it in a similar way on $T_y$ with the $(\xi_{y,v})_{v\in \Td}$, letting recursively 
\begin{center}
$\psimn(y_{k,i})= \dE^{\cM_n}[\psimn(y_{k,i})\vert \sigma(A_{k,i})] + \xi_{y,\Phi(y_{k,i})}\sqrt{\text{Var}(\psimn(y_{k,i})\vert \sigma(A_{k,i}))}$
\end{center} 
where $A_{k,i}:=T_x\cup \{y_{\ell,j}\vert (\ell,j) \prec (k,i)\}$, $\prec$ being the lexicographical order on $\dN^2$, and $\Phi$ is the isomorphism between $T_y$ and $\kT_y$.

\noindent
Recall that $J=\{j\geq 1, \, B^*(z_j,a'_n)\text{ is a joining ball}\}$ and that for $j\in J$, we denote $v_j(0)$ the unique vertex of $B^*(z_j,a'_n-2a_n)\cap B_{\cM_n}(T_y,a_n)$. Since no cycle is discovered when revealing $B^*(z_j,a'_n)\setminus B^*(z_j,a'_n-2a_n)$, the intersection of $B^*(z_j,a'_n-a_n)$ and  $T_y$ is a unique vertex $v_j(a_n)$, which is in the $a_n$-offspring of $v_j(0)$ in the tree $B^*(z_j,a'_n)$ rooted at $z_j$. Then we realize $\psimn$ on $B^*(z_j, a'_n-2a_n)$ and on the shortest path $P_j$ from $v(0)$ to $v(a_n)$ as in (\ref{eqn:recursivegffgraph}), via a family of i.i.d. $\cN(0,1)$ random variables $(\xi_{j,k,i})_{k,i\geq 0}$. In the tree $T_j:=B^*(z_j, a'_n-2a_n)\cup P_j$ with root $z_j$, denoting $z_{j,k,i}$ the $i$-th vertex at generation $k$ and 
\begin{equation}\label{eqn:Ajkidef}
A_{j,k,i}:=T_x\cup T_y\cup \{\cup_{j'<j} T'_j\}\cup \{y_{j,k',i'}\, \vert \, (k',i')\prec (k,i)\}
\end{equation}
the set of vertices where $\psimn$ has already been revealed before $z_{j,k,i}$, we let
\begin{equation}\label{eqn:gffjoiningball}
\psimn(z_{j,k,i})=\dE^{\cM_n}[\psimn(z_{j,k,i})\vert \sigma(A_{j,k,i})] +\xi_{j,k,i}\sqrt{\text{Var}(\psimn(y_{j,k,i})\vert \sigma(A_{j,k,i}))}.
\end{equation}
Write $\mathcal{S}^*(x,y)\subseteq \cS'(x,y)$ the event that there exists $j_0\geq 1$ and a path from $z_{j_0}$ to $v_{j_0}(a_n)$ such that $\psimn(v)\geq h$ for every vertex $v$ of that path. In particular, on $\mathcal{S}^*(x,y)$, $x$ and $y$ are in the same connected component of $\Enh$. Recall the definitions of $\kappa$ (\ref{eqn:andef}) and $\gamma$ (Lemma~\ref{lem:joiningball}).

\begin{proposition}\label{prop:exploendjoin}
If $\kappa$ and $\gamma/\kappa$ are large enough, then 
$$\dP_{ann}(\mathcal{S}^*(x,y))\underset{n\rightarrow +\infty}{\longrightarrow} \eta(h)^2.$$
\end{proposition}

\begin{proof}[Proof of Proposition~\ref{prop:exploendjoin}.]
Let $\gamma >3\kappa +18$. By Lemma~\ref{lem:joiningball}, 
\begin{center}
$\limsup_{n\rightarrow +\infty}\dP_{ann}(\cS^*(x,y)) \leq \lim_{n\rightarrow +\infty}\dP_{ann}(\cS'(x,y))  = \eta(h)^2.$
\end{center}
Let $\cE_n:=\{\text{$\cM_n$ is not a good graph}\}\cup \{\max_{z\in V_n}\vert\psimn(z) \vert \geq \log^{2/3}n \}$. By Proposition~\ref{prop:goodgraph} and Lemma~\ref{lem:maxgff}, $\dP_{ann}(\cE_n)\rightarrow 0$.
Therefore, it is enough to show that 
\begin{center}
$\limsup_{n\rightarrow +\infty}\dP_{ann}(\cE_n^c\cap\,( \cS'(x,y)\setminus \cS^*(x,y))\,)=0$.
\end{center}
By a straightforward adaptation of the reasoning below (\ref{eqn:growthsuccessful}),
\begin{center}
$\lim_{n\rightarrow +\infty}\dP_{ann}(\cE_n^c\cap\,( \cS'(x,y)\setminus \cS''(x,y))\,)=0$,
\end{center}
where $\cS''(x,y):=\cS'(x,y)\cap \{\forall z\in T_x\cup T_y,\, \psimn(z)\geq h+(\log^{-1}n)/2 \}$. Hence, we are left with proving that
\begin{equation}\label{eqn:joinSprimenotstar}
\limsup_{n\rightarrow +\infty}\dP_{ann}(\cE_n^c\cap\,( \cS''(x,y)\setminus \cS^*(x,y))\,)=0.
\end{equation}
\noindent
We use again a binomial argument. Fix a realization of $\cM_n$ which is a good graph (recall that $x$ and $y$ have already been fixed at the beginning of Section~\ref{subsubsec:jointexplo}). For $j\in J$ in increasing order, generate the GFF on $T_j$ as in (\ref{eqn:gffjoiningball}). Denote $E_j$ the event that $z_j$ and $v_j(a_n)$ are in the same connected component of $\Enh \cap T_j$. Note that on $\cS''(x,y)$, $T_x\subseteq \cC_x^{\cM_n,h}$ and $T_y\subseteq \cC_y^{\cM_n,h}$, so that $\cS''(x,y) \cap (\cup_{j\in J}E_j)\subseteq \cS^*(x,y)$. 
\\ 
Note that $A_{j,0,1}=T_x\cup T_y\cup \{\cup_{j'<j}T_{j'}\}$ by \eqref{eqn:Ajkidef}. By Lemma~\ref{lem:jointTj} below, if $n$ is large enough, then for any good graph $\cM_n$, for any $j\in J$ and any event $\cE'_j\subseteq \cS''(x,y)\cap \{\max_{z\in A_{j,0,1}}\vert \psimn(z)\vert \leq \log^{2/3}n\}$ that is measurable w.r.t. the exploration until the revealment of $\psimn$ on $T_{j-1}$, we have 
\begin{center}
$\dP^{\cM_n}(E_j\, \vert \cE'_j )\geq \log^{\gamma(\cxiii/3 -1)}n.$
\end{center}
Then, letting $Z\sim \text{Bin}(\lfloor \log^{\gamma-3\kappa-18}n \rfloor,  \log^{\gamma(\cxiii/3 -1)}n)$, we have
\begin{center}
$\dP_{ann}(\cE_n^c\cap\,( \cS''(x,y)\setminus \cS^*(x,y))\,)\leq \dP_{ann}(\cE_n)+ \dP(Z=0).$
\end{center}
If $\kappa$ and $\gamma/\kappa$ are large enough so that 
$\gamma-3\kappa -18+\gamma(K_8/3-1)=\gamma K_8/3-3\kappa-18>0$,
we have $\lim_{n\rightarrow +\infty}\dP(Z=0)=0$. Since $\dP_{ann}(\cE_n)\rightarrow 0$, this yields (\ref{eqn:joinSprimenotstar}).
\end{proof}
\noindent 
It remains to prove the following. For a good graph $\cM_n$, denote $\cF^{\cM_n}_j$ the sigma-algebra of the exploration until the revealment of $\psimn$ on $T_{j-1}$. In particular, note that $\cF'_j$ contains the information on $\sigma(A_{j,0,1})$, and that the structure of $J$ is $\cF'_j$-measurable.
\begin{lemma}\label{lem:jointTj}
Let $\cxiii:= \log_{d-1}((1+\lambda_h)/2)$.  For $n$ large enough, we have for any good graph $\cM_n$, any $j\in J$ and any $\cF^{\cM_n}_j$-measurable event $\cE'_j\subseteq \cS''(x,y)\cap \{\max_{z\in A_{j,0,1}}\vert \psimn(z)\vert \leq \log^{2/3}n\}$:
\begin{equation}\label{eqn:joinTj}
\dP^{\cM_n}(E_j\, \vert \cE'_j)\geq \log^{\gamma(\cxiii/3 -1)}n.
\end{equation}
\end{lemma}

\begin{proof} All the inequalities in this proof hold for $n$ large enough, uniformly in the choice of a good graph $\cM_n$, $j\in J$ and $\cE'_j\in \cF^{\cM_n}$. We proceed in two steps. First, we prove that
\begin{equation}\label{eqn:joinuntilvjO}
\dP^{\cM_n}(v_j(0)\in\cC_{z_j} \, \vert \cE'_j   )\geq \log^{(\gamma-2\kappa)(K_8/2-1)}n ,
\end{equation} 
where $\cC_{z_j}$ is the connected component of $z_j$ in $\Enh\cap B^*(z_j,a'_n-2a_n)$. Second, we show that for some constant $\cxiv>0$ (uniquely depending on $d$ and $h$),
\begin{equation}\label{eqn:joinuntilvjn}
\dP^{\cM_n}(\forall v\in P_j, \psimn(v)\geq h\, \,\vert \,(\cE'_j\cap \{v_j(0)\in\cC_{z_j} \})\, )\geq \log^{-\cxiv \kappa}n.
\end{equation}
We prove that both hold for $n$ large enough, uniformly in $v\in T_j$ and on $\cE'_j$. 
\\
If $\gamma/\kappa$ is large enough, (\ref{eqn:joinuntilvjO}) and (\ref{eqn:joinuntilvjn}) imply \eqref{eqn:joinuntilvjO}, since in this case, we have
\begin{center}
$\dP^{\cM_n}(E_j\, \vert\cE'_j )\geq \log^{(\gamma-2\kappa)(K_8/2-1) -K_9\kappa}n \geq \log^{\gamma(K_8/3-1)}n $.
\end{center}

\noindent
\textbf{Part 1:} proof of (\ref{eqn:joinuntilvjO}).
\\
Since $\vert A_{j,k,i}\vert\leq n^{2/3}$ and $\tx(B_{\cM_n}(A_{j,k,i},a_n))=\tx(A_{j,k,i})$ for all $k,i\geq 0$, we can apply Proposition~\ref{prop:couplinggffsexplo} as below (\ref{eqn:growthsuccessful}) to bound the difference between $\psimn$ on $\cC_{z_j}$ and $\phid$ on an isomorphic subtree of $\Td$, with the following coupling: $\phid(\circ):=\psimn(z_j)$, and then $\phid$ is defined as in Proposition~\ref{prop:recursivegfftrees} via $(\xi_{j,k,i})_{k,i\geq 0}$. Recall that on $\cS''(x,y)$, we have that $\psimn(z_j)\geq h+(\log^{-1}n)/2$. By Proposition~\ref{prop:Chlargedevgrowthrate}, for any $\delta >0$ and for large enough $n$,
\[
\min_{a\geq h+(\log^{-1}n)/2}\dP^{\Td}_a\left(\vert \cZ_{a'_n-2a_n}^{h+(\log^{-1}n)/2,+}\vert \geq (\lambda_{h}-\delta)^{a'_n-2a_n}\right) \geq \frac{p\dP^{\Td}(\cE^+)}{2}
\]
where $\dP^{\Td}(\cE^+)>0$ (recall (\ref{eqn:thm43})) and 
\begin{center}
$p:= \min_{a\geq h+(\log^{-1}n)/2}\dP^{\Td}_a(\exists v\in B_{\Td^+}(\circ,1), \, \phid(z)\geq h+1)>0$. 
\end{center}
Note in particular that for $\delta'>\frac{\log^{-1}n}{2}$ such that $\lambda_{h+\delta'}>\lambda_h -\delta$ (such $\delta'$ exists by Proposition~\ref{prop:thm43adapted}, if $n$ is large enough), $\cZ_{a'_n-2a_n}^{h+\delta',+}\subseteq \cZ_{a'_n-2a_n}^{h+(\log^{-1}n)/2,+}$. Since $\cM_n$ is a good graph an $\cE'_j\subseteq \{\max_{z\in A_{j,0,1}}\vert \psimn(z)\vert\leq \log^{2/3}n\} $, we can apply
Proposition~\ref{prop:couplinggffsexplo} as below \eqref{eqn:growthsuccessful} to bound the difference between $\psimn$ on $B^*(z_n,a'_n-2a_n)$ and $\phid$ on $\Td$ , and we get:
\[
\dP_{ann}\left(\vert \partial\cC_{z_j} \vert \geq (\lambda_{h}-\delta)^{a'_n-2a_n}\,\,\bigg| \cE'_j\, \right) \geq \frac{p\dP(\cE^+)}{2}+o(1)\geq \frac{p\dP(\cE^+)}{3}.
\]
\noindent
By cylindrical symmetry of $B_{\Td^+}(\circ, a'_n-2a_n)$, we even have
\begin{align*}
\dP^{\cM_n}\left(v(0)\in\partial\cC_{z_j}\,\,\bigg|\cE'_j\, \right)  \geq  & \frac{p\dP(\cE^+)}{3}\,\frac{(\lambda_{h}-\delta)^{a'_n-2a_n}}{\vert \partial B_{\Td^+}(\circ, a'_n-2a_n)\vert} \geq  \frac{p\dP(\cE^+)}{3} \left(\frac{\lambda_{h}-\delta}{d-1}\right)^{a'_n-2a_n}.
\end{align*}
\noindent
Since $K_8=\log_{d-1}((1+\lambda_h)/2)$, taking $\delta$ small enough yields (\ref{eqn:joinuntilvjO}).

\noindent
\textbf{Part 2:} proof of (\ref{eqn:joinuntilvjn}).
\\
Denote $v_j(1), \ldots, v_j(a_n-1)$ the vertices from $v_j(0)$ to $v_j(a_n)$ on the path $P_j$. Remark that it suffices to prove that there exists a constant $\cxiv>0$ such that for $n$ large enough, for every $k\in \{1, \ldots, a_n\}$, 
\begin{equation}\label{eqn:joinvjvjplusone}
\dP^{\cM_n}\left(\psimn(v_j(k))\geq h\vert\,(\cE'_j\cap \{\psimn(v_j(k-1))\geq h\})\,\right)\geq (d-1)^{-K_9}.
\end{equation}
In the notation of (\ref{eqn:gffjoiningball}), $v_j(k)=y_{j,k+a'_n-2a_n,1}$ for $1\leq k\leq a_n$. Write $A_k:=A_{j,k+a'_n-2a_n,1}$. Suppose that for $n$ large enough and all $k\in \{1, \ldots, a_n\}$, on $\cE'_j\cap \{\psimn(v_j(k-1))\geq h\} $:
\begin{equation}\label{eqn:expconnexionlastline}
\dE^{\cM_n}[\psimn(v_j(k))\vert \sigma(A_k)] > -\vert h\vert -1,\,\text{ and}
\end{equation}
\begin{equation}\label{eqn:varconnexionlastline}
\text{Var}^{\cM_n}(\psimn(v_j(k))\vert \sigma(A_k))>\frac{1}{d-1}.
\end{equation}
Then (\ref{eqn:joinvjvjplusone}) holds with 
\begin{center}
$\cxiv:=-\log_{d-1} \dP(Y \geq  (\vert h\vert+\frac{\vert h\vert +1}{d-1})/\sqrt{d-1} \,)$, 
\end{center}
where $Y\sim \cN(0,1)$. Thus, it is enough to establish (\ref{eqn:expconnexionlastline}) and (\ref{eqn:varconnexionlastline}). 
\\
For $k\geq 1$, note that by construction of $B^*(z_j, a'_n)$, $v_j(k-1)$ and $v_j(a_n)$ are the only vertices of $\partial A_{k}$ at distance less than $a_n$ of $v_j(k)$. Let $(X_{s})_{s\geq 0}$ be a discrete time SRW started at $v_j(k)$, and $\tau:=\inf \{s\geq 0, \, d_{\cM_n}(v_j(k),X_s)\geq a_n\}$. Write $H$ for the hitting time of $A_k$ by $(X_s)$. Letting
\begin{center}
$a_1:=\bP_{v_j(k)}^{\cM_n}(X_{H}=v_j(k-1),H<\tau)$ and $a_2:=\bP_{v_j(k)}^{\cM_n}(X_{H}=v_j(a_n),H<\tau)$, 
\end{center}
we get as in the proof of Proposition~\ref{prop:couplinggffsexplo} that for $\psimn(A_k)$ in $\cE'_j\cap \{\psimn(v_j(k-1))\geq h\} $:
\[
\dE^{\cM_n}[\psimn(v_j(k))\vert \sigma(A_k)] > a_1\psimn(v_j(k-1))+a_2\psimn(v_j(a_n))-\log^{-1}n.
\]
Since $0\leq a_1+a_2\leq 1$ and $\min(\psimn(v_j(k-1)), \psimn(v_j(a_n)))\geq h \geq -\vert h \vert$, (\ref{eqn:expconnexionlastline}) follows.
\\
Using Proposition~\ref{prop:condgff}, we split $V:=\text{Var}^{\cM_n}(\psimn(v_j(k))\vert \sigma(A_{k}))$ in the following way: 
\begin{align*}
V=&\,\gmn\left(v_j(k),v_j(k)\right)-\bE_{v_j(k)}^{\cM_n}\left[\gmn\left(v_j(k),X_{H}\right)\mathbf{1}_{\{H<\tau \}}\right]-\bE_{v_j(k)}^{\cM_n}\left[\gmn\left(v_j(k),X_{H}\right)\mathbf{1}_{\{H\geq\tau \}}\right]
\\
&+\frac{\bE_{v_j(k)}^{\cM_n}[H]}{\bE_{\pi_n}^{\cM_n}[H]}\bE_{\pi_n}^{\cM_n}\left[\gmn\left(v_j(k),X_{H}\right)\right].
\end{align*}
By (\ref{eqn:greenfunctionGn}), (\ref{eqn:greenfunctionGnapprox1}) and (\ref{eqn:greenfunctionGnapprox2}), if $\kappa$ is large enough, for $n$ large enough,
\begin{align*}
\gmn\left(v_j(k),v_j(k)\right)\hspace{-0.5mm}-a_1\gmn\left(v_j(k),v_j(k-1)\right)\hspace{-0.5mm}-&a_2\gmn\left(v_j(k),v_j(a_n)\right)
\\
> &\frac{d-1}{d-2}-\frac{a_1+a_2}{d-2}-\log^{-1} n
\\
\geq &\frac{1}{d-2}-\log^{-1} n.
\end{align*}
As below (\ref{eqn:condvarsplit}), we get that 
$$
\left\vert \bE_{v_j(k)}^{\cM_n}\left[\gmn\left(v_j(k),X_{H}\right)\mathbf{1}_{\{H<\tau \}}\right]\hspace{-0.7mm} -a_1\gmn\left(v_j(k),v_j(k-1)\right)\hspace{-0.4mm}-a_2\gmn\left(v_j(k),v_j(a_n)\right)\right\vert\hspace{-0.7mm} \leq \hspace{-0.7mm}\log^{-1}n
$$
\noindent
and that
$$
\left\vert \bE_{v_j(k)}^{\cM_n}\left[\gmn\left(v_j(k),X_{H}\right)\mathbf{1}_{\{H\geq\tau \}}\right] -\frac{\bE_{v_j(k)}^{\cM_n}[H]}{\bE_{\pi_n}^{\cM_n}[H]}\bE_{\pi_n}^{\cM_n}\left[\gmn\left(v_j(k),X_{H}\right)\right]\right\vert \leq \log^{-1} n.
$$

\noindent
These three inequalities imply that for $\psimn(A_k)$ in $ \cE'_j\cap \{\psimn(v_j(k-1))\geq h\} $:
\begin{center}
$
\text{Var}^{\cM_n}(\psimn(v_j(k))\vert \sigma(A_{k}))\geq \frac{1}{d-2}-3\log^{-1}n>\frac{1}{d-1}.$
\end{center}
This shows (\ref{eqn:varconnexionlastline}) and the proof is complete.
\end{proof}

\subsection{Average number of connections in $\Enh$}\label{subsec:averageconnec}
Write $x\overset{h}{\leftrightarrow}y$ if  $x$ and $y$ are in the same connected component of $\Enh$, for $x,y\in V_n$.
In this section, we prove (\ref{eqn:mainthm}) of Theorem~\ref{thm:maingff} via an argument on the number of pairs of vertices such that $x\overset{h}{\leftrightarrow}y$.
Let $S_n$ be the set of pairs of distinct $x,y\in V_n$ such that $x\overset{h}{\leftrightarrow}y$. Let $A_n\subset V_n$ be the set of vertices $x$ such that $\vert \cC_x^{\cM_n, h}\vert \leq n^{1/2}$.
\\
We first suppose that the following two lemmas hold, and show (\ref{eqn:mainthm}). Then, we derive them from Propositions~\ref{prop:explo1vertexaborted} and \ref{prop:exploendjoin}, using a second moment argument.
\begin{lemma}\label{lem:connexionsupperbound}
For every $\varepsilon >0$, $\lim_{n\rightarrow +\infty}\dP_{ann}\left(\vert A_n \vert \geq (1-\eta(h)-\varepsilon)n\right)=1$.
\end{lemma}

\begin{lemma}\label{lem:connexionslowerbound}
For every $\varepsilon >0$, $\lim_{n\rightarrow +\infty}\dP_{ann}(\vert S_n\vert\geq (\eta(h)^2/2-\varepsilon)n^2)=1$.  
\end{lemma}

\begin{proof}[Proof of (\ref{eqn:mainthm})]
Fix $\varepsilon >0$. Remark that every connected component of $\Enh$ is either included in $A_n$, or does not intersect $A_n$. Let $(\widehat{\gamma}^{(n)}_i)_{i\geq 1}$ (resp. $(\gamma^{(n)}_i)_{i\geq 1}$) be the sizes of the connected components in $A_n$ (resp. not in $A_n$), listed in decreasing order of size (break ties arbitrarily). 
\\
Let $\cE_n:=\{\vert A_n \vert \geq (1-\eta(h)-\varepsilon)n\}\cap\{\vert S_n \vert \geq (\eta(h)^2/2-\varepsilon)n^2\}$.
On $\cE_n$, we have that 
\begin{center}
$\sum_{i\geq 1} \widehat{\gamma}^{(n)}_i(\widehat{\gamma}^{(n)}_i-1)\,+\sum_{i\geq 1}\gamma^{(n)}_i(\gamma^{(n)}_i-1)=2\vert S_n\vert\geq \eta(h)^2n^2 -2\varepsilon n^2$.
\end{center}
Moreover, we have by definition of $A_n$:
\begin{center}
$\sum_{i\geq 1} \widehat{\gamma}^{(n)}_i(\widehat{\gamma}^{(n)}_i-1) \leq \sum_{i\geq 1} \widehat{\gamma}^{(n)}_i\sqrt{n}\leq n^{7/4}$. 
\end{center}
Thus, for $n$ large enough, 
$$\gamma^{(n)}_1(\vert V_n \vert - \vert A_n\vert)\geq \sum_{i\geq 1}\gamma^{(n)}_i(\gamma^{(n)}_i-1)\geq (\eta(h)^2-3\varepsilon)n^2.$$
But $\vert V_n \vert -\vert A_n \vert \leq (\eta(h)+\varepsilon)n$, so that 
$$\gamma^{(n)}_1\geq \frac{\eta(h)^2-3\varepsilon}{\eta(h)+\varepsilon}n\geq ((\eta(h)-4\eta(h)^{-1}\varepsilon) n.$$
Since $\gamma^{(n)}_1$ is the cardinality of a set included in $V_n \setminus A_n$, one has $\gamma^{(n)}_1\leq (\eta(h)+\varepsilon)n$. Note that for $n$ large enough, $\gamma^{(n)}_1>\sqrt{n}\geq \widehat{\gamma}^{(n)}_1$. Therefore, $\gamma^{(n)}_1=\vert\cC^{(n)}_1\vert$, and we have on $\cE_n$:
$$((\eta(h)-4\eta(h)^{-1}\varepsilon) n \leq\vert\cC^{(n)}_1 \vert \leq  (\eta(h)+\varepsilon)n $$.
By Lemmas~\ref{lem:connexionsupperbound} and \ref{lem:connexionslowerbound}, $\lim_{n\rightarrow +\infty}\dP_{ann}(\cE_n)=1$. Since $\varepsilon$ was arbitrary, the proof is complete. 
\end{proof}

\begin{remark}\label{rem:secondcompofirstbound}
Note that we have $\vert \cC^{(n)}_2\vert =\max(\gamma^{(n)}_2,\widehat{\gamma}^{(n)}_1)$. Since on $\cE_n$,  $\widehat{\gamma}^{(n)}_1\leq \sqrt{n}$ and 
$\gamma^{(n)}_2\leq \vert V_n\vert -\vert A_n \vert -\gamma^{(n)}_1\leq (1+4\eta(h)^{-1})\varepsilon n$, 
we get that $\vert \cC^{(n)}_2\vert /n\cvpann 0$.
\end{remark}

\begin{proof}[Proof of Lemma~\ref{lem:connexionsupperbound}] Let  $A'_n$ be the set of vertices such that their lower exploration (Section~\ref{subsec:explo1vertexaborted}) is aborted. By Remark~\ref{rem:explorationsize}, for $n$ large enough, $A'_n\subseteq A_n$, and it is enough to prove the result for $A'_n$ instead of $A_n$.
\\
Let $\varepsilon \in (0,1)$. By Proposition~\ref{prop:explo1vertexaborted}, for $n$ large enough and every $x\in V_n$, we have
\begin{equation}\label{eqn:abortedprobaapprox}
\vert \dP_{ann}(x\in A'_n)- (1-\eta(h))\vert\leq\varepsilon.
\end{equation}

\noindent
We claim that for $n$ large enough, for all distinct $x,y\in V_n$,
\begin{equation}\label{eqn:variancenbaborted}
\vert \Cov_{ann}(\mathbf{1}_{x\in A'_n},\mathbf{1}_{y\in A'_n}) \vert \leq 4\varepsilon.  
\end{equation}
\noindent 
Indeed, $\Cov_{ann}(\mathbf{1}_{x\in A'_n},\mathbf{1}_{y\in A'_n}) =\dP_{ann}(x,y\in A'_n) -\dP_{ann}(x\in A'_n)\dP_{ann}(y\in A'_n)$. On one hand, by (\ref{eqn:abortedprobaapprox}), we have 
\begin{center}
$ \vert \dP_{ann}(x\in A'_n)\dP_{ann}(y\in A'_n)-(1-\eta(h))^2\vert\leq 2\varepsilon +\varepsilon^2\leq 3\varepsilon$.
\end{center}
On the other hand, perform successively the lower explorations from $x$ and then from $y$ as in Section~\ref{subsec:explo1vertexaborted} (with the additional condition \ref{C7}). We get $\dP_{ann}(\text{\ref{C7} happens})=o(1)$ as in (\ref{eqn:C1happensexplo}). Then, revealing $\psimn$ on $R_x\cup R_y$ and comparing it to $\phid$ as below (\ref{eqn:growthsuccessful}), we obtain
\begin{center}
$\vert\dP_{ann}(x,y\in A'_n) -(1-\eta(h))^2\vert \leq \varepsilon$. 
\end{center}
This shows (\ref{eqn:variancenbaborted}). We now apply Bienaymé-Chebyshev's inequality:
\begin{align*}
\dP_{ann}(\vert A'_n\vert \leq (1-\eta(h)-2\varepsilon^{1/4} )n ) &\leq \dP_{ann}(\vert\,\vert A'_n\vert -\dE_{ann}[\vert A'_n\vert]\,\vert \geq \varepsilon^{1/4} n)
\\
&\leq \frac{1}{\sqrt{\varepsilon} n^2}\left(\sum_{x,y\in V_n}\Cov_{ann}(\mathbf{1}_{x\in A'_n},\mathbf{1}_{y\in A'_n}) \right)
\\
&\leq \frac{n +n(n-1)4\varepsilon}{\sqrt{\varepsilon} n^2}
\\
&\leq 5\sqrt{\varepsilon}
\end{align*}
for $n$ large enough. Since $\varepsilon$ can be taken arbitrarily small, the proof is complete.
\end{proof}

\begin{proof}[Proof of Lemma~\ref{lem:connexionslowerbound}]
Let $\varepsilon\in (0,1)$. Denote $S_n^*$ the set of pairs $x,y\in V_n$ such that $\cS^*(x,y)$ holds. Since $S^*_n\subseteq S_n$, it is enough to prove the Lemma for $S^*_n$ instead of $S_n$. First, by Proposition~\ref{prop:exploendjoin}, 
\begin{center}
$\dE_{ann}[\vert S^*_n\vert ]\geq (\eta(h)^2-\varepsilon)\frac{n(n-1)}{2}\geq (\eta(h)^2/2-\varepsilon)n^2$ 
\end{center}
for large enough $n$. Second, we show that for $n$ large enough and all distinct $x,y,w,t \in V_n$,
\begin{equation}\label{eqn:variancenbsuccess}
\vert \Cov_{ann}(\mathbf{1}_{\cS^*(x,y)},\mathbf{1}_{\cS^*(w,t)}) \vert \leq 2\varepsilon,
\end{equation}
\noindent 
from which we conclude by a second moment computation as in Lemma~\ref{lem:connexionsupperbound}.
We have 
\begin{center}
$ \Cov_{ann}(\mathbf{1}_{\cS^*(x,y)},\mathbf{1}_{\cS^*(w,t)}) =\dP_{ann}(\cS^*(x,y) \cap \cS^*(w,t))-\dP_{ann}(\cS^*(x,y))\dP_{ann}(\cS^*(w,t))$.
\end{center}
By Proposition~\ref{prop:exploendjoin}, for $n$ large enough, 
\begin{equation}\label{eqn:covarconnexion}
\vert\dP_{ann}(\cS^*(x,y))\dP_{ann}(\cS^*(w,t)) -\eta(h)^4\vert\leq \varepsilon.
\end{equation}
Now, perform successively the exploration of Section~\ref{subsec:explo1vertexsuccess} from $x$, then from $y$, then from $z$ and finally from $t$ (via an array of i.i.d. standard normal variables $(\xi_{u,v})_{u\in \{x,y,w,t\},v\in \Td}$). We add the following condition: for any $u\in \{x,y,w,t\}$, the exploration from $v$ is stopped as soon as it meets a vertex seen in a previous exploration. The probability that this happens is $o(1)$ by Remark~\ref{rem:explorationsize} and (\ref{eqn:binomdominationdiff}), since $o(\sqrt{n})$ vertices and half-edges are revealed during these four explorations. Therefore, as for (\ref{eqn:growthsuccessful}), we get that for $n$ large enough, 
\begin{equation}\label{eqn:fourexplosuccess}
\dP_{ann}(\text{the explorations from $x,y,z,t$ are all successful})\in(\eta(h)^4-\varepsilon/2,\eta(h)^4+\varepsilon/2).
\end{equation}
\noindent
If these explorations are successful, develop balls from $\partial T_{x}$ to $\partial T_y$ as described in Section \ref{subsec:connexion}, with $Q_j:=R_x\cup R_y\cup R_w\cup R_t\cup B^*_j$ for $z_j\in \partial T_x$. Then do the same from $\partial T_w$ to $\partial T_t$, this time with $Q_j:=R_x\cup R_y\cup (\cup_{z\in \partial T_x}B^*(z,a'_n))\cup R_w\cup R_t\cup B^*_j$ for $z_j\in \partial T_w$. Finally, reveal $\psimn$ on $T_x,T_y,T_w,T_t$ and on the joining balls from $T_x$ to $T_y$ and from $T_w$ to $T_t$, in that order. 
\\
One can adapt readily the proof of Lemma~\ref{lem:joiningball} to show that with $\dP_{ann}$-probability $1-o(1)$, if the four explorations are successful then there are at least $\log^{\gamma-3\kappa -18}n$ joining balls from $\partial T_{x}$ (resp. $\partial T_w$) to $\partial T_y$ (resp. $\partial T_t$). Note in particular that the 
estimates of (\ref{eqn:explototaledges}), (\ref{eqn:saturationjoiningballs}) and (\ref{eqn:spoliationjoiningballs}) still hold. It is also straightforward to adapt the proof of Proposition~\ref{prop:exploendjoin},  
and we finally have 
\begin{center}
$\vert\dP_{ann}(\cS^*(x,y) \cap \cS^*(w,t) ) -\dP_{ann}(\text{the explorations from $x,y,w,t$ are all successful}) \vert\leq \varepsilon/2$.
\end{center} 
Together with (\ref{eqn:covarconnexion}) and (\ref{eqn:fourexplosuccess}), this yields (\ref{eqn:variancenbsuccess}).
\end{proof}

\section{Uniqueness of the giant component}\label{sec:uniqueness}

\noindent
In this Section, we prove (\ref{eqn:secondcompo}). We start by the lower bound in Section~\ref{subsec:uniquenesslowerbound}, showing the existence $\dP_{ann}$-w.h.p. of a component (different from $\cC_1^{(n)}$) having $\Theta(\log n)$ vertices. 
\\
Then, to show that $\vert\cC_2^{(n)}\vert=O(\log n)$ $\dP_{ann}$-w.h.p., we perform an exploration of a new kind, starting from some $x\in V_n$. It consists of three phases (Sections~\ref{subsec:phase1} to \ref{subsec:phase3}), during which we 
assign a \textbf{pseudo-GFF} $\psimnt$ to the vertices that we visit. $\psimnt$ is defined via a recursive construction that mimics Proposition~\ref{prop:recursivegfftrees}, as long as there are no cycles (like $\phid$ on $\kT_x$ in Section~\ref{sec:exploration}). If we meet one cycle, which can happen in the first phase (Section~\ref{subsec:phase1}, we use an \textit{ad hoc} modification of these recursive formulas. Finally, in Section~\ref{subsec:secondcomporevealpsin}, we reveal the true values of $\psimn$ one by one on the set of vertices we have explored via Proposition~\ref{prop:gffann}, and check that they are close to $\psimnt$. We show that either $\vert\cC_x^{\cM_n,h}\vert = O(\log n)$, or $\vert\cC_x^{\cM_n,h}\vert =\Theta(n)$, in which case $\cC_x^{\cM_n,h}=\cC_1^{(n)}$ by Remark~\ref{rem:secondcompofirstbound}. 
Contrary to Section~\ref{sec:exploration}, we need this alternative to hold for \emph{every} $x\in V_n$, $\dP_{ann}$-w.h.p. By a union bound, it is enough to prove that 
\begin{equation}\label{eqn:secondcompoalternativeforexplo}
\text{for $x\in V_n$, $\dP_{ann}(\{\vert\cC_x^{\cM_n,h}\vert = O(\log n)\}\cup\{\vert\cC_x^{\cM_n,h}\vert =\Theta(n)\})=o(1/n)$.} 
\end{equation}
Let us sketch this exploration in the lines below.
\\
\textbf{First phase (Section \ref{subsec:phase1}).} We explore the connected component $\cC$ of $x\in V_n$ in the set $\{y\in V_n, \psimnt(y)\geq h-n^{-a}\}$ for some constant $a>0$. More precisely, we give a mark to each vertex $y$ such that $\vert\psimnt(y)-h\vert \leq n^{-a}$, and explore each connected component $\cC'$ of $\cC\setminus \cM$, where $\cM$ is the set of marked vertices, until 
\\
- (i) $\cC'$ is fully explored and has no more than $O(\log n)$ vertices, or 
\\
- (ii) $\lfloor \cxv\log n\rfloor$ vertices of $\cC'$ have been seen but not yet explored, for some constant $\cxv$ fixed in the second phase. 
\\
We replace $a_n$ of (\ref{eqn:andef}) by a ``security radius'' $r_n=\Theta(\log n)$. Adapting Proposition~\ref{prop:couplinggffsexplo} (see Lemma~\ref{lem:couplinggffadapted}), this will allow us in Section~\ref{subsec:secondcomporevealpsin} to bound the difference between $\psimnt$ and $\psimn$ by $n^{-a}$ with probability $1-o(1/n)$, so that for every connected component $\cC'$ of $\cC\setminus \cM$, either $\cC'\subseteq\cC_x^{\cM_n,h}$ or $\cC'\cap \cC_x^{\cM_n,h}=\emptyset$.
\\
If we kept $a_n$, $\psimnt$ would approximate $\psimn$ only with precision $\log^{-\Theta(1)}n$. With probability $\Theta(1/n)$, there would be too many vertices $y$ such that $\vert \psimnt(y)-h\vert \leq \log^{-\Theta(1)}n$, hence for which we cannot know by anticipation whether they will be in $\cC_x^{\cM_n,h}$ or not.
\\
Moreover, we do not have $\dP(\text{\ref{C1} happens})=O(1/n)$ as soon as the number of vertices explored goes to infinity with $n$. We will need to accept the possible occurrence of one cycle. When this happens, we have to define $\psimnt$ in a slightly different manner. In Section~\ref{subsec:secondcomporevealpsin}, we need a variant of Lemma~\ref{lem:couplinggffadapted} to control the difference between $\psimnt$ and $\psimn$ in that case (Lemma~\ref{lem:couplinggffadapted1cycle}).
\\
\textbf{Second phase (Section~\ref{subsec:phase2}).} If (i) happens for every component $\cC'$, the exploration is over. In this case, $\vert \cC\vert=O(\log n)$ (see \ref{D2} and Proposition~\ref{prop:blockexplo}) and thus $\vert\cC_x^{\cM_n,h}\vert =O(\log n)$. For each $\cC'$ such that (ii) happens, we explore its $\lfloor \cxv\log n\rfloor$ remaining vertices, this time in a fashion similar to Section~\ref{subsec:explo1vertexsuccess}. Each of these explorations has a probability bounded away from 0 to be successful. If $\cxv$ is large enough, with probability at least $1-o(1/n)$, at least one of these explorations is successful, and has a boundary of size $\Theta( n^{1/2}b_n)$.
\\
\textbf{Third phase (Section~\ref{subsec:phase3}).} For every $\cC'$ such that (ii) happens, we show that the successful exploration of the second phase is connected to a positive proportion of the vertices of $V_n$, via an adaptation of the joint exploration in Section~\ref{subsec:connexion}. This yields (\ref{eqn:secondcompoalternativeforexplo}).

\subsection{Lower bound}\label{subsec:uniquenesslowerbound}
In this section, we prove the existence of $\cz>0$ such that
\begin{equation}\label{eqn:secondcompolowerboundd}
\dP_{ann}(\vert\cC_2^{(n)}\vert\geq \cz^{-1}\log n)\rightarrow 1.
\end{equation}
\noindent
To do so, we first show that if $\cz$ is large enough, then with probability at least $n^{-1/4}$, $\Ch$ consists of a unique ``line'' $L_n$ of length $\lfloor 2\cz^{-1}\log n\rfloor$ (each vertex of $\Ch$ having one child, except the last one which has no children, Lemma~\ref{lem:secondcompolowerboundtree}). Thus, if we take $n^{1/3}$ vertices of $V_n$ and assign to each of them an independent copy of $\Ch$, the probability that at least one of them is isomorphic to $L_n$ is $1-o(1)$. Then, we realize the component of the corresponding vertex in $E^{\geq h}$, and check that it is indeed isomorphic to $L_n$ with $\dP_{ann}$-probability $1-o(1)$, so that $\vert\cC_2^{(n)}\vert \geq \lfloor 2 \cz^{-1}\log n\rfloor\geq  \cz^{-1}\log n$.
\begin{lemma}\label{lem:secondcompolowerboundtree}
For $n\geq 1$, let $\cE_{\text{line},n}$ be the event that $\Ch$ has $m:=\lfloor 2\cz^{-1} \log n\rfloor$ vertices $x_1, \ldots, x_m$ with $x_1=\circ$ and for $2\leq i\leq m$ $x_i$ is the child of $x_{i-1}$, and that for $1\leq i\leq m$, for any other child $y$ of $x_i$ in $\Td$,   $\phid(x_i)\in [h+1,h+2]$, $\phid(y)<h-1$. If $\cz$ is large enough, then for $n$ large enough, 
\begin{center}
$\dP^{\Td}(\cE_{\text{line},n})\geq n^{-1/4}$.
\end{center}
\end{lemma}

\begin{proof}
Let $v_1, \ldots,v_d$ be the children of $\circ$ in $\Td$. 
Remark that for $n$ large enough,
\begin{center}
$\dP_{ann}(\cE_{\text{line},n})\geq\dP^{\Td}(\phid(\circ)\in [h+1,h+2]) pp'^{\lfloor2 \cz^{-1}\log n\rfloor}p''$, where
\end{center}
\begin{center}
$p:=\inf_{a\in [h+1,h+2]}\dP^{\Td}_a(\{\phid(v_1)\in [h+1,h+2]\} \cap \{\forall  i\in \{2, \ldots, d\}, \phid(v_i)<h-1\}) $,
\end{center}
\begin{center}
$p':=\inf_{a\in [h+1,h+2]}\dP^{\Td}_a(\{ \phid(v_1)\in [h+1,h+2]\}\cap \{\forall i\in \{2, \ldots, d-1\}, \, \phid(v_i)<h-1 \}) $,
\end{center}
\begin{center}
$p'':=\inf_{a\in [h+1,h+2]}\dP^{\Td}_a(\forall i\in \{1, \ldots, d-1\}, \phid(v_i)<h-1) $.
\end{center}
Using Proposition~\ref{prop:recursivegfftrees}, one checks that $p,p',p''>0$. Taking $\cz>-8\log p'$ yields the~result.
\end{proof}

\begin{proof}[Proof of \eqref{eqn:secondcompolowerboundd}.]
Fix $\cz>12\log(d-1)$ large enough such that the conclusion of Lemma~\ref{lem:secondcompolowerboundtree} holds. 
Let $x_1, \ldots, x_{\lfloor n^{1/3}\rfloor}\in V_n$. For all $1\leq i\leq \lfloor n^{1/3}\rfloor$, attach a family $(\xi^{(i)}_y)_{y\in \Td}$ of i.i.d.~$\cN(0,1)$ variables, these families being themselves independent. For each $i$, let $\Ch(i)$ be the associated realization of $\Ch$ built with the family $(\xi^{(i)}_y)$, using Proposition~\ref{prop:recursivegfftrees}. 
\\
Then, pick a realization of $\cM_n$, and let $\cE_{1,n}$ be the event that $\cM_n$ is a good graph and that the balls $B_{\cM_n}(x_i,\lfloor 2\cz^{-1}\log n\rfloor+2a_n)), 1\leq i\leq \lfloor n^{1/3}\rfloor$ are disjoint and have no cycle. 
By Proposition~\ref{prop:goodgraph} and \eqref{eqn:binomdominationdiff} with $m_0=\lfloor n^{1/3}\rfloor$, $m\leq \lfloor n^{1/3}\rfloor\times (d-1)^{\lfloor 2\cz^{-1}\log n\rfloor+2a_n+10}$ and $k=1$, $\dP_{ann}(\cE_{1,n})=1-o(1)$.
\\
On $\cE_{1,n}$, realize $\psimn$ on $B_{\cM_n}(x_i,\lfloor 2\cz^{-1}\log n\rfloor)$ for $i=1, \ldots, \lfloor n^{1/3}\rfloor$ successively, using the random variables $(\xi^{(i)}_y)_{y\in B_{\Td}(\circ,\lfloor 2\cz^{-1}\log n\rfloor )}$ and the construction of Proposition~\ref{prop:gffann}. For $1\leq i\leq   \lfloor n^{1/3}\rfloor$, let $\Phi^{(i)}$ be a rooted isomorphism between $B_{\cC_{x_i}^{\cM_n,h}}(x_i,\lfloor 2\cz^{-1}\log n\rfloor)$ and $ B_i:=B_{\Ch(i)}(x_i,\lfloor 2\cz^{-1}\log n\rfloor)$.  Let $\cE_{2,n}:=\{\sup_{1\leq i\leq   \lfloor n^{1/3}\rfloor, y\in B_i}\vert \psimn((\Phi^{(i)})^{-1}(y)) -\phid(y)\vert <1/2  \}$. A direct adaptation of the reasoning below \eqref{eqn:growthsuccessful} to show that $\dP_{ann}(\cE_{2,n})=1-o(1)$. 
\\
Let $\cE_{3,n}:=\{\exists i_0\in \{1, \ldots, \lfloor n^{1/3}\rfloor\}, B_{i_0}\text{ is isomorphic to }L_n\}$. By Lemma~\ref{lem:secondcompolowerboundtree}, we have that $\dP_{ann}(\cE_{3,n})\geq 1- (1-n^{-1/4})^{\lfloor n^{1/3}\rfloor}=1-o(1)$. 
\\
Finally, by \eqref{eqn:mainthm}, we have $\dP_{ann}(\cE_{4,n})=1-o(1)$, with $\cE_{4,n}=\{\vert \cC^{(n)}_1\vert \geq 10\log n\} $. 
\\
All in all, we have shown that $\dP_{ann}(\cE_{1,n}\cap \cE_{2,n}\cap \cE_{3,n}\cap \cE_{4,n})=1-o(1)$. And on $\cE_{1,n}\cap \cE_{2,n}\cap \cE_{3,n}\cap \cE_{4,n}$ for $n$ large enough, $\vert \cC_{x_{i_0}}^{\cM_n,h}\vert\geq K_0^{-1}\log n$ and $\cC_{x_{i_0}}^{\cM_n,h}\neq \cC^{(n)}_1$, so that $\vert \cC_{2}^{(n)}\vert\geq K_0^{-1}\log n$. 
\end{proof}

\subsection{First phase}\label{subsec:phase1}
In this section, we define the first phase of the exploration, and show that it is successful with $\dP_{ann}$-probability $1-n^{-5/4}$ (Proposition~\ref{prop:blockexplo}). 
\\
Let $a>0$. For every $n\in \dN$, define
\begin{equation}\label{eqn:rndef}
r_n:=\lfloor 0.05 \log_{d-1}n\rfloor
\end{equation}  Let $\delta\in (0,h_{\star}-h)$ and $\ell\in \dN$ be such that the conclusion of Remark~\ref{rem:uniformsupercriticsurvival} holds.
\\
\\
\textbf{The first phase of the exploration.} Let $x\in V_n$. Let $\cM$ be the set of marked vertices. Initially, $\cM=\{x\}$.
While $\cM\neq \emptyset$, pick $y\in \cM$ in an arbitrary way and proceed to its \textbf{subexploration}, as detailed below. 
\\
There are three possible scenarios, according to the number of cycles discovered during the first phase (zero, one, or more).
\\
\textbf{I -} We first assume that we do not meet any cycle throughout the first phase of the exploration of $x$. 
\\
Assume we have picked some vertex $y\in \cM$. We now define the \textbf{subexploration from} $y$. Let $T_y$ be the subexploration tree, that we will build by adding subtrees of depth $\ell$ in a breadth-first way. Initially $T_y=\{y\}$.
\\
While $1\leq \vert \partial T_{y}\vert\leq \cxv\log n$, perform a \textbf{step}: take $y_1\in \partial T_{y}$ of minimal height and if $y_1\neq x$, let $\overline{y_1}$ be its only neighbour where $\psimnt$ has already been defined. Note that if $y_1\neq y$, $\overline{y_1}$ is the parent of $y_1$ in $T_y$. Reveal all the edges of $B_{\cM_n}(y_1,\overline{y_1},r_n+\ell)$, where we recall that $B_{\cM_n}(y_1,\overline{y_1},r_n+\ell)$ is the graph obtained by taking all paths of length $r_n+\ell$ starting at $y_1$ and not going through $\overline{y_1}$. Since we suppose that no cycle arises, $B_{\cM_n}(y_1,\overline{y_1},\ell)$ is a tree, that we root at $y_1$. 
\\
If $y_1=y=x$, replace $B_{\cM_n}(y_1,\overline{y_1},r_n+\ell)$ and $B_{\cM_n}(y_1,\overline{y_1},\ell)$  by $B_{\cM_n}(x,r_n+\ell)$ and $B_{\cM_n}(x,\ell)$ respectively.  Let $\psimnt(x)\sim \cN(0, \frac{d-1}{d-2})$. If $\phid(\circ)< h-n^{-a}$, then the whole exploration (not only the first phase) is over. 
\\
We construct $T_y(y_1)$, the subtree of $B_{\cM_n}(y_1,\overline{y_1},\ell) $ in $\{z, \psimnt(z)\geq h+n^{-a} \}$. We start with $T_y(y_1)=\{y_1\}$.
\\
For $k=1,2,...\ell-1$ successively, denote $y_{k,1}, \ldots, y_{k,m}$ the children of the $(k-1)$-th generation of $T_y(y_1)$. Let $(\xi_{y_1,k,i})_{k,i\geq 0}$ be an array of i.i.d. variables of law $\cN(0,1)$, independent of everything else.  Set
\begin{equation}\label{eqn:phidellblock}
\psimnt(y_{k,i}):=\frac{1}{d-1}\psimnt(\overline{y_{k,i}})+\sqrt{\frac{d}{d-1}}\xi_{y_1,k,i}.
\end{equation}
Add $y_{k,i}$ to $T_y(y_1)$ if $\psimnt(y_{k,i})\geq h+n^{-a}$, and give a mark to $y_{k,i}$ (and thus add it to $\cM$) if $h-n^{-a}\leq \psimnt(y_{k,i})<h+n^{-a}$.
\\
Finally, include $T_y(y_1)$ in $T_y$, add the vertices of $\partial T_y(y_1)$ to $N_y$ and take $y_1$ away from $\partial T_y$. The step is then over.
\\
\\
If $\vert \partial T_y\vert\not\in [1,\cxv\log n]$, the subexploration is finished. Say that it is \textbf{fertile} if $\vert \partial T_y\vert>\cxv\log n$, and \textbf{infertile} else (hence if $\partial T_y=\emptyset$).
\\
\\

\noindent
\textbf{II -} Suppose now that a unique cycle $C$ arises in the subexploration of some vertex $y$, when revealing the pairings of $B_{\cM_n}(y_1,\overline{y_1},r_n+\ell)$ in the step from $y_1$, for some $y\in V_n$ and $y_1\in V_n\setminus\{x\}$ (we treat the special case $x=y_1$ below). Let $m:=\vert C\vert$ be the number of vertices in $C$. There are two cases.
\\
\textbf{Case 1:}\label{case1}
When $C$ is discovered, there are already $k$ consecutive vertices $y_1, \ldots, y_k$ of $C$ where $\psimnt$ has been defined, for some $1\leq k \leq m-1$. Reveal $B_{\cM_n}(C,r_n)$. Denote $z_1, \ldots, z_{m-k}$ the remaining vertices of $C$, such that $z_1\neq y_2$ is a neighbour of $y_1$, and $z_i$ is a neighbour of $z_{i-1}$ for $i\geq 2$. Give a mark to $z_1, \ldots, z_{m-k}$ and $y_1$. Take $y_1$ away from $\partial T_y$. If $y_k$ was in $\partial T_{\widetilde{y}}$ for some $\widetilde{y}$ whose subexploration was performed previously, take it away from that set. 
\vspace{10mm}

\includegraphics[scale=0.8]{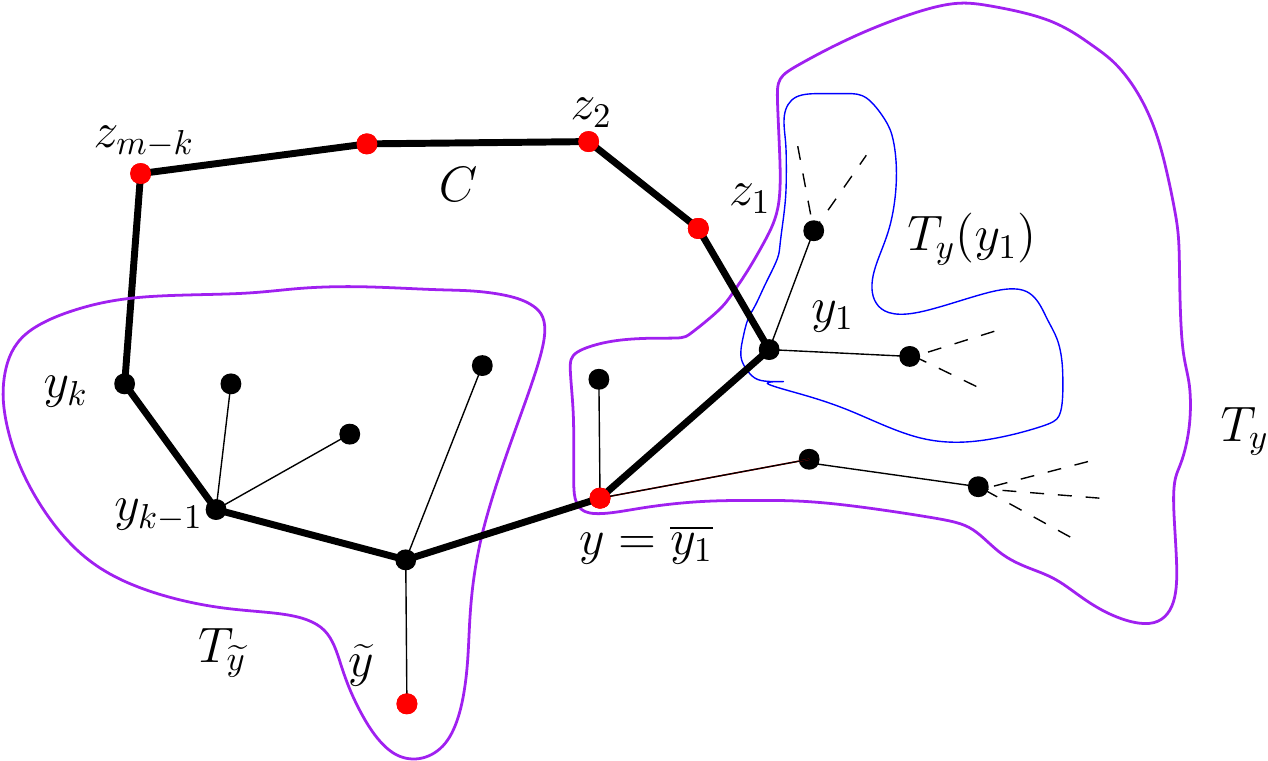}

\begin{center}
Figure 4. Case 1. Marked vertices are in red. $C$ consists of the thick edges. $T_y$ and $T_{\widetilde{y}}$ are delimited by the purple contours. Remark that we could have $y=\widetilde{y}$ (it is not the case here).
\end{center}

\noindent
We now define $\psimnt$ on the $z_i$'s. To do so, we mimic a recursive construction of the GFF on $G_m$, the infinite connected $d$-regular graph $G_m$ having a unique cycle $C_m$ of length $m$ (such a construction always exists on a transient graph, see for instance Lemma 1.2 in~\cite{RodriguezSznitman}). 
$G_m$ consists of a cycle $C_m$ of length $m$, with $d-2$ copies of $\Td^+$ attached to each vertex of $C_m$, thus it is clear that the SRW is transient and that the Green function $G_{G_m}$ and the GFF are well-defined. 
\\
Let $u_k,\ldots, u_1, v_1,\ldots, v_{m-k}$ be the vertices of $C_m$, listed consecutively. Let $U\hspace{-1mm}:=\hspace{-1mm}G_m\hspace{-0.2mm}\setminus\{u_1, \ldots, u_k\}$, $(X_j)_{j\geq 0}$ a SRW on $G_m$ and recall that $T_U$ is the exit time of $U$. Define
\begin{center}
$\alpha: =\bP_{v_1}^{G_m}(X_{T_U}\hspace{-1mm}=u_1,T_U\hspace{-1mm}<\hspace{-1mm}+\infty)$, $\beta: =\bP_{v_1}^{G_m}(X_{T_U}\hspace{-1mm}=u_k,T_U\hspace{-1mm}<\hspace{-1mm}+\infty)\mathbf{1}_{\{k>1\}}$
\vspace{5mm}

and $\gamma:=\bE_{v_1}^{G_m}[\sum_{j=0}^{T_U-1}\mathbf{1}_{\{X_j=v_1\}}],$
\end{center}
and let $(\xi_{i})_{i\geq 1}$ be a family of i.i.d standard normal variables, independent of everything else. Define
\begin{center}
$\psimnt(z_1):=\alpha\psimnt(y_1)+\beta\psimnt(y_k)+\sqrt{\gamma}\xi_{1}$. 
\end{center}
Then for $i\geq 2$, define recursively 
\begin{equation}\label{eqn:blockexplocase1}
\psimnt(z_i):=\alpha_{m-(k+i-1)}\psimnt(z_{i-1})+\beta_{m-(k+i-1)}\psimnt(y_k)+\sqrt{\gamma_{m-(k+i-1)}}\xi_{i}
\end{equation}
\noindent
where we set 
\begin{center}
$U_i:=G_m\setminus\{u_1, \ldots, u_k,v_1, \ldots,v_{i-1}\}$, $\alpha_{m-(k+i-1)}:=\bP_{v_i}^{G_m}(X_{T_{U_i}}=v_{i-1},\,T_{U_i}<+\infty)$, 
\end{center}
\begin{center}
$\beta_{m-(k+i-1)}:=\bP_{v_i}^{G_m}(X_{T_{U_i}}=u_k,\,T_{U_i}<+\infty)$ and $\gamma_{m-(k+i-1)}:=\bE_{v_i}^{G_m}[\sum_{j=0}^{T_{U_i}-1}\mathbf{1}_{\{X_j=v_i\}}]$. 
\end{center}

\noindent
\textbf{Case 2:}
$\psimnt$ has not been defined on any vertex of $C$. There exists a unique path of consecutive vertices $y_1, \ldots, y_j$ for some $j\geq 2$ such that $y_j\in C$, and for $2\leq i \leq j-1$, $\psimnt(y_i)$ has not been defined and $y_i\not \in C$. 
Reveal $B_{\cM_n}(\{y_1 \ldots, y_{j-1}\}\cup C,r_n)$.  Give a mark to the vertices of $\{y_1, \ldots, y_{j-1}\}\cup C$. 
Take $y_1$ away from $\partial T_y$. 

\includegraphics[scale=0.8]{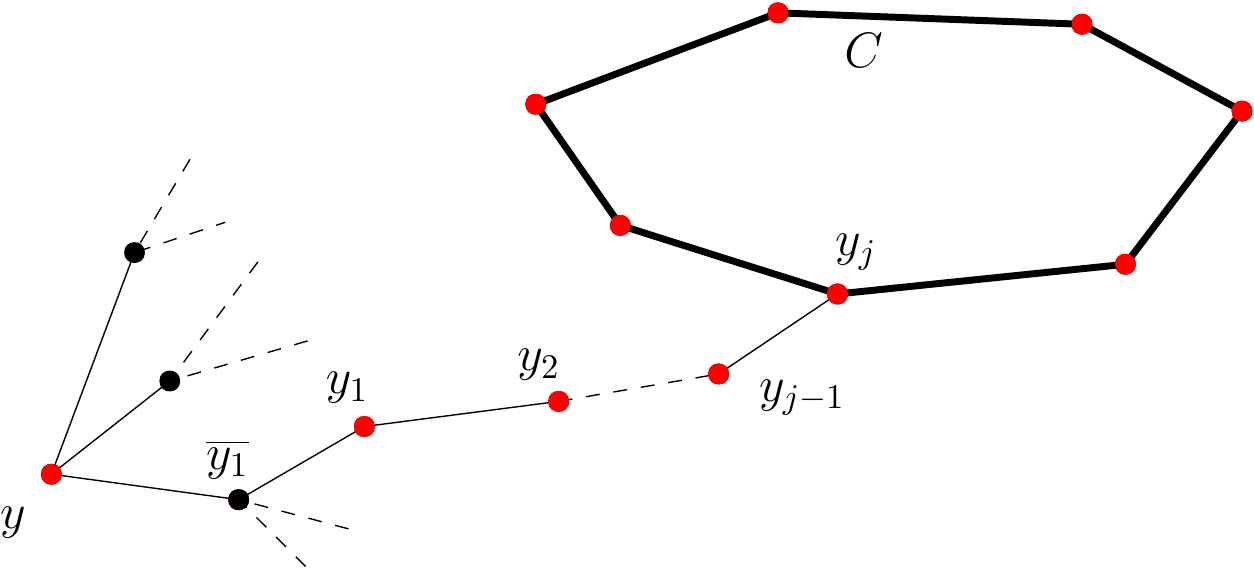}\label{figure5}

\begin{center}
Figure 5. Case 2. Marked vertices are in red.
\end{center}

\noindent
We define $\psimnt$ on $\{y_1, \ldots, y_{j-1}\}\cup C$ in a way similar to Case 1. Let $(\xi_i)_{i\geq 1}$ be a sequence of i.i.d. standard normal variables, independent of everything else. For $i=1, 2, \ldots, j$, set
\begin{equation}\label{eqn:blockexplocase2}
\psimnt(y_i):=\alpha'_{j-i}\psimnt(y_{i-1})+\sqrt{\gamma'_{j-i}}\xi_{i},
\end{equation}
\noindent
where $\alpha'_{j-i}:=\bP_{z}^{G_m}(H_{\{z'\}}<+\infty)$ and $\gamma'_{j-i}:=\bE_z^{G_m}[\sum_{l=0}^{H_{\{z'\}}-1}\mathbf{1}_{\{X_l=z\}}]$, $z,z'$ being two neighbours in $G_m$ such that $z$ (resp. $z'$) is at distance $j-i$ (resp. $j-i+1$) of $C_m$. Then, define $\psimnt$ on $C$ as in Case 1 with $k=1$. 
\\
\\
If $C$ is discovered while revealing $B_{\cM_n}(x,r_n+\ell)$, give a mark to all vertices of $C$, and, if $x\not\in C$, all vertices on the unique shortest path from $x$ to $C$. Let $k:=d_{\cM_n}(x,C)$. Let $u_k\in G_m$ be at distance $k$ of the cycle $C_m$. Let $\psimnt(x)\sim \cN(0, G_{G_m}(u_k,u_k))$. If $\phid(\circ)< h-n^{-a}$, then the whole exploration (not only the first phase) is over. 
\\
\\
\textbf{After the discovery of $C$.} Assume that no other cycle will be discovered during the first phase of the exploration from $x$. 
Resume the subexploration from $y$, by picking a new vertex $y'\in \partial T_y$ (recall that we have taken $y_1$ away from $\partial T_y$) and proceeding to the step from $y'$, and so on until the subexploration from $y$ is over. After that, as long as $\cM\neq \emptyset$, pick a vertex $y''$ in $\cM$ and proceed to its subexploration as described in \textbf{I} (recall that no new cycle is discovered), with the following amendment, if $y''$ has several neighbours $y''_1, \ldots, y''_k$ (for some $k\geq 2$) on which $\psimnt$ has already been defined (see for instance $y_2$ or $y_j$ in Figure~\ref{figure5}). In this case, in the step from $y''$ (the first step of the subexploration from $y''$), reveal only $\cap_{1\leq i\leq k}B_{\cM_n}(y'',y''_i,r_n+\ell) $ instead of $B_{\cM_n}(y'',\overline{y''},r_n+\ell) $ (not by the way that $\overline{y''}$ is not properly defined). In words, perform the subexploration only in the direction of non-marked vertices. 
\\
\\
\textbf{III -} If a second cycle arises, the whole exploration from $x$ is over, and is not successful. 
\\
\\
We have now fully described how the first phase can unfold. Say that the first phase is successful if at some point, the following holds:
\begin{enumerate}[label=D\arabic*]
\item\label{D1} $\cM=\emptyset$,
\item\label{D2} $\psimnt$ has been defined on at most $\lfloor \cxvi\log n\rfloor$ vertices, and
\item\label{D3}  at most one cycle has been discovered.
\end{enumerate}
\noindent
Denote $\cS_{1}(x)$ this event. If all subexploration trees were infertile, then the whole exploration from $x$ is over, and said to be successful. Denote $\cS_{1,\text{stop}}(x)$ this event.

\begin{proposition}\label{prop:blockexplo}
Fix $a>0$. For any fixed value of $\cxv>0$, if $\cxvi$ is large enough, then for large enough $n$ and every $x\in V_n$, $\dP_{ann}(\cS_{1}(x))\geq 1-n^{-5/4}$.
\end{proposition}

\begin{proof}
In a nutshell, the argument is as follows. We first show that with probability $1-o(n^{-5/4})$, we see at most one cycle after having explored $n^{3/10}$ vertices (which will turn out to be more than the number of vertices seen in the first phase). Then, we show that for some large constant $K>0$, with probability $1-o(n^{-5/4})$, the first $K\log n$ subexplorations encompass less than $K^2\log n$ vertices. To this end, we use the fact that during any subexploration, say of some vertex $y$, the increment of $\vert\partial T_y\vert$ during a step of the subexploration from $y$ has a positive expectation if $\ell$ is large enough, since $\Ch$ has a positive probability to have an exponential growth (see Lemma~\ref{lem:uniformsupercritical} and Remark~\ref{rem:uniformsupercriticsurvival}, this is why we explore subtrees of depth $\ell$ instead of individual vertices). Finally, we show that  with probability $1-o(n^{-5/4})$, the first $K\log n$ subexplorations generate less than $K\log n$ marks. In particular, note that except for cycles, each vertex where $\psimnt$ is defined has a chance $O(n^{-a})$ to get a mark, and a cycle brings less than $2d^{\ell}r_n=\Theta( \log n)$ marked vertices. This in turn ensures that there are indeed no more than $K\log n$ subexplorations in the first phase. 
%
%
\\
\\
We now proceed to the proof itself. Let $\cE_{1,n}$ be the event that two cycles are discovered before $n^{3/10}$ vertices have been seen during the exploration from $x$. By (\ref{eqn:binomdominationdiff}) with $k=2$, $m_0=1$, $m_E=0$ and $m\leq n^{3/10}$, 
\begin{equation}\label{eqn:blockexplosuccess1}
\dP_{ann}(\cE_{1,n})\leq n^{-4/3}.
\end{equation}

\noindent
Next, let $K>0$ be a constant; we bound the number $N$ of steps during the first $\lfloor K\log n\rfloor$ subexplorations (or on all subexplorations if there are less than $\lfloor K\log n\rfloor$ of them). Let $\cE_{2,n}:=\{N\geq K^2\log n\}$, we now show that if $K$ is large enough, then we have for $n$ large enough:
\begin{equation}\label{eqn:blockexplosuccess2}
\dP_{ann}(\cE_{2,n})\leq n^{-2}.
\end{equation}
Suppose that we perform the subexploration from some vertex $y$. Let $y_1\in \partial T_y$, with $y_1\neq y$. If no cycle arises when revealing the $(r_n+\ell)$-offspring of $y_1$, then  $\vert\partial T_y(y_1)\vert-1$, the increment of $\vert \partial T_y\vert$ during that step, dominates stochastically $\rho_{\ell,h,\delta}-1$, by definition of $\rho_{\ell,h,\delta}$ at Remark~\ref{rem:uniformsupercriticsurvival}, as soon as $n^{-a}<\delta$. If a cycle arises (which happens at most once on $\cE_{1,n}^c$), at most two vertices are marked and taken away from $\partial T_y$. 
If the subexploration is not over after $j$ steps, then $\vert \partial T_y\vert \stdomi S_{j-1}-2$, where $S_{j-1}$ is the sum of $j-1$ i.i.d. random variables of law $\rho_{\ell,h, \delta}-1$, hence taking values in the bounded interval $[0,d^{\ell}]$, and with a positive expectation by Remark~\ref{rem:uniformsupercriticsurvival}. The $-1$ in $j-1$ comes from the fact that we do not include the step from $y$ (as it might be that we only explore a reduced number of vertices once we have discovered a cycle, if $y$ is a marked vertex with at least two marked neighbours). The $+2$ comes from the possibility that up to two vertices can be taken away from $\partial T_y$ if there is a cycle (see Case 1 in particular). Hence
\begin{center}
$\dP_{ann}(\text{the subexploration from $y$ lasts more than $j$ steps})\leq \dP(S_{j-1}\leq \cxv\log n+2)$ 
\end{center}
By the exponential Markov inequality, there exist constants $c,c'>0$ such that
\begin{center}
for every $j\geq 2\cxv(\dE[\rho_{\ell,h,\delta}]-1)^{-1}\log n$, $\,\,\,\,\,\dP(S_{j-1}\leq \cxv\log n +2)\leq c e^{-c'j}$.
\end{center} 
Therefore, $N$ is stochastically dominated by a sum $S$ of $\lfloor K\log n\rfloor$ i.i.d. variables of some law $\mu$ (independent of $n$) such that 
\begin{center}
for every $j\geq 2\cxv(\dE[\rho_{\ell,h,\delta}]-1)^{-1}\log n$, $\,\,\,\,\mu([j,+\infty))\leq ce^{-c'j}$. 
\end{center}
Hence, letting $\cE_{2,n}:=\{N\geq K^2\log n\}$, by the exponential Markov inequality,
\begin{center}
$\dP_{ann}(\cE_{2,n}\cap \cE_{1,n}^c)\leq\dP_{ann}(S\geq K\lfloor K\log n\rfloor)\leq  (\dE[e^{c'Y/2}]e^{-c'K\mu/2})^{\lfloor K\log n\rfloor} $
\end{center} 
for large enough $n$ and for $Y\sim \mu $. Thus \eqref{eqn:blockexplosuccess2} follows.

\noindent
Finally, let $\cE_{3,n}$ be the event that more than $3r_n+ \frac{K}{2}\log n$ marks are given during the first $\lfloor K\log n\rfloor$ subexplorations. Suppose that for $n$ large enough, 
\begin{equation}\label{eqn:blockexplosuccess3}
\dP_{ann}(\cE_{1,n}^c\cap \cE_{2,n}^c \cap \cE_{3,n})\leq n^{-2}.
\end{equation}
\noindent
Taking $K>1$, on $\cE_{1,n}^c\cap \cE_{2,n}^c\cap \cE_{3,n}^c$ (which holds with $\dP_{ann}$-probability $1-o(n^{-5/4})$ by (\ref{eqn:blockexplosuccess1}), (\ref{eqn:blockexplosuccess2}) and (\ref{eqn:blockexplosuccess3})), less than $1+3r_n+ \frac{K}{2}\log n\leq \lfloor K\log n\rfloor$ vertices receive a mark during the first $\lfloor K\log n\rfloor $ subexplorations, so that $\cM=\emptyset$ after at most $\lfloor K\log n\rfloor$ subexplorations. Moreover, in a step of a subexploration, less than $ d^{\ell}$ vertices are added to the subexploration tree. Hence if $\cxvi>K^2 d^{\ell}$, on $\cE_{2,n}^c$, \ref{D2} holds. In addition, less than $(d-1)^{2r_n}\leq n^{1/10}$ new vertices are seen in a step of a subexploration. When a cycle $C$ is revealed, there are at most $2r_n$ new vertices in $C$ (and on the path leading to $C$, in Case 2), so that less than $3r_n (d-1)^{r_n}\leq n^{1/10}$ new vertices are seen. Therefore, on $\cE_{1,n}^c\cap\cE_{2,n}^c$, less than $n^{3/10}$ vertices are seen during the first $\lfloor K\log n\rfloor$ subexplorations. Thus, on $\cE_{1,n}^c\cap \cE_{2,n}^c\cap \cE_{3,n}^c$, conditions \ref{D1}, \ref{D2} and \ref{D3} are satisfied and this yields the conclusion.
\\
Therefore, it only remains to show \eqref{eqn:blockexplosuccess3}, which we now do. When a cycle $C$ appears, our construction implies that less than $3r_n$ vertices receive a mark. When performing $m$ steps in a subexploration, the number of marked vertices obtained is stochastically dominated by a binomial random variable $\text{Bin}(m, d^{\ell}n^{-a})$. Indeed, at each step, we reveal $\psimnt$ (and thus $\phid$) on less than $d^{\ell}$ vertices. And for any vertex $y\in \Td\setminus\{\circ\}$, by Proposition~\ref{prop:recursivegfftrees},
$$
\max_{a'\geq h}\dP^{\Td}(\phid(y)\in [h-n^{-a},h+n^{-a}]\, \vert\, \phid(\overline{y})=a')\leq \frac{2n^{-a}}{\sqrt{2\pi d/(d-1)}}\leq n^{-a}.
$$
\noindent
Hence, if $Z\sim \text{Bin}(\lceil K^2\log n\rceil, C_2d^{\ell}n^{-a})$,
$$
\dP_{ann}(\cE_{1,n}^c\cap \cE_{2,n}^c \cap \cE_{3,n})\leq \dP\left(Z\geq  \frac{K}{2}\log n\right)\leq {\lceil K^2\log n \rceil\choose \frac{K}{2}\log n} (C_2d^{\ell}n^{-a})^{\frac{K}{2}\log n}.
$$
But by (\ref{eqn:binomiallittlethings}), for large enough $n$, ${\lceil K^2\log n \rceil\choose \frac{K}{2}\log n}\leq\lceil K^2\log n \rceil^{\frac{K}{2}\log n} $ . This yields \eqref{eqn:blockexplosuccess3}, and the conclusion follows.
\end{proof}

\subsection{Second phase}\label{subsec:phase2}
\noindent
If $\cS_{1,\text{stop}}(x)$ holds, i.e. all subexploration trees are infertile, the exploration is over. In this Section, we suppose that $\cS_{1}(x)\setminus\cS_{1,\text{stop}}(x)$ holds. For every fertile tree, we perform an exploration similar to that of Section~\ref{subsec:explo1vertexsuccess} from each vertex of its boundary, and show that with probability at least $1-n^{-6/5}$, at least one exploration per fertile tree is successful, and hence has a boundary of size $\Theta(n^{1/2}b_n)$ (Proposition~\ref{prop:secondphase}). We illustrate this second phase in Figure 6 below. 
\\
Let $T_1, \ldots T_m$ be the fertile subexploration trees for some positive integer $m$. For every $q\in \{1, \ldots,  m\}$, denote $y_{q,1}, y_{q,2},\ldots$ the vertices of $\partial T_q$. For $q=1,2,\ldots$ successively, we perform the explorations from $y_{q,i}$, $i\geq 1$ as defined in Section~\ref{subsec:explo1vertexsuccess} (using an array of independent standard normal variables $(\xi_{y_{q,i},k,j})_{k,j\geq 0}$). If $y$ is the $j$-th vertex of the $k$-th generation in the exploration tree of $y_{q,i}$, let recursively 
$$\psimnt(y):=\frac{\psimnt(\overline{y})}{d-1}+\sqrt{\frac{d}{d-1}}\xi_{y_{q,i},k,j},
$$
so that $\psimnt$ plays the role of $\phid$ on $\Td$ in Section~\ref{subsec:explo1vertexsuccess}. We implement three modifications:
\begin{itemize}
\item we do not explore towards $\overline{y_{q,i}}$, the parent of $y_{q,i}$ in $T_q$ (hence we generate $\psimnt$ as the GFF on a subtree of $\Td^+$ instead of $\Td$),
\item we do not stop the exploration if $\psimnt(y_{q,i})<h+\log^{-1}n$ (we only know \textit{a priori} from the first phase that $\psimnt(y_{q,i})\geq h+n^{-a}$), and
\item we stop the exploration if it meets a vertex already discovered in the first phase or during the previous exploration of some $y_{q',j}$ (thus with $q'<q$, or $q'=q$ and $j<i$).
\end{itemize}
A vertex $y_{q,i}$ whose exploration is successful is \textbf{back-spoiled} if one vertex of its exploration is seen later during the exploration of $y_{q',j}$. Let 
\begin{center}
$\cE_{4,n}:=\{\exists q\leq m, \,\text{all successful explorations of $y_{q,1}, y_{q,2},\ldots,$ are back-spoiled}\}$ 
\end{center}
On $\cS_2(x):= ( \cS_{1}(x)\setminus  \cS_{1,\text{stop}}(x))\cap \cE_{4,n}^c$, say that the second phase is successful.

\vspace{-20mm}
\includegraphics[scale=0.5]{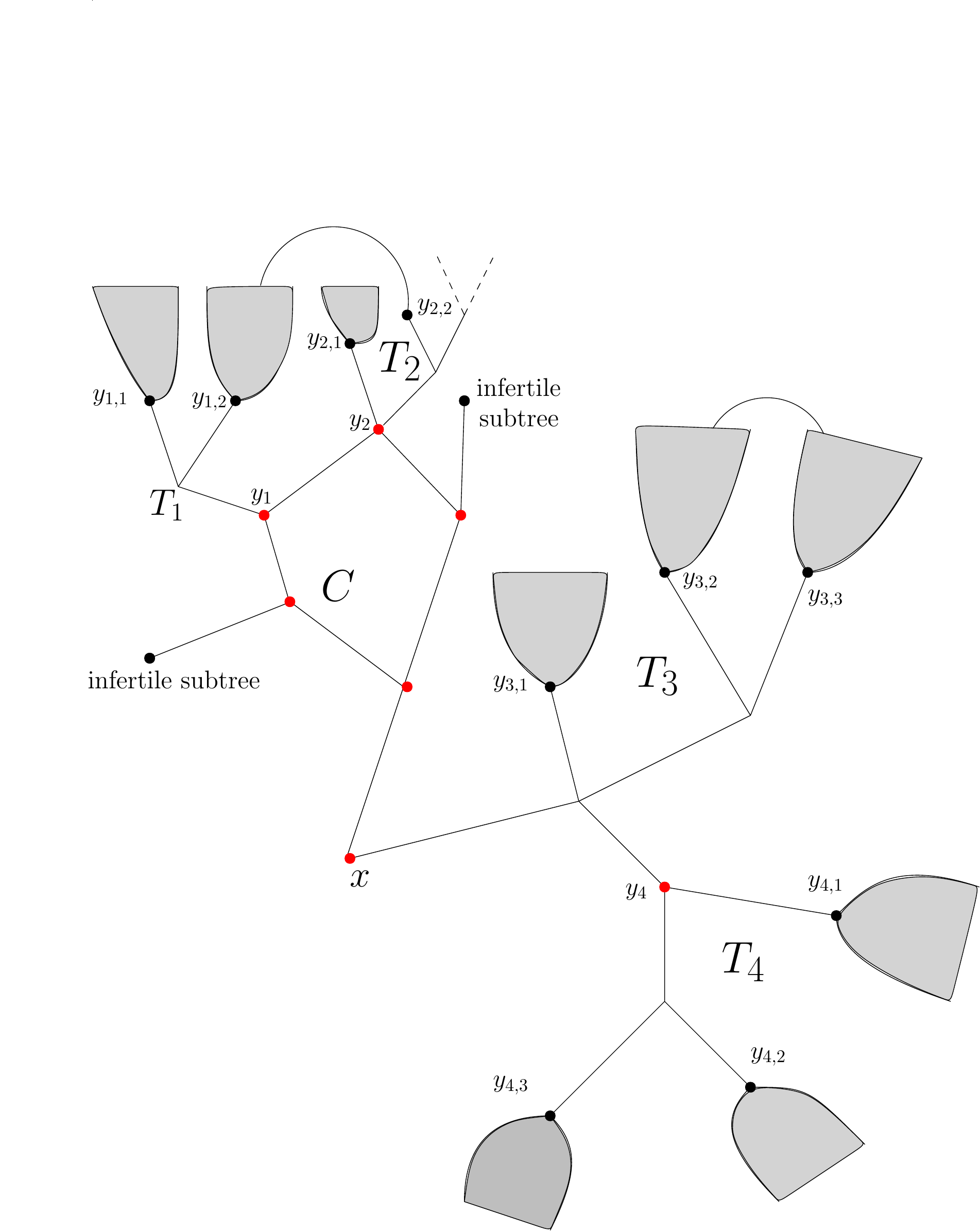}

\begin{center}
Figure 6. Marked vertices are in red. There are four fertile subexploration trees ($T_1$ rooted at $y_1$, $T_2$ rooted at $y_2$, $T_3$ rooted at $x$ and $T_4$ rooted at $y_4$), and two infertile ones. Lightgray areas correspond to the explorations of the $y_{q,i}$'s in the second phase. $y_{2,2}$ is spoiled by $y_{1,2}$. $y_{3,2}$ is back-spoiled by $y_{3,3}$. Each of the subxploration trees will be either included in $\cC_x^{\cM_n,h}$ or have no common vertex with $\cC_x^{\cM_n,h}$, depending on the value of $\psimn$ on the marked vertices.
\end{center}

\begin{proposition}\label{prop:secondphase}
If $\cxv$ is large enough (depending only on $d$ and $h$), then for $n$ large enough and every $x\in V_n$,
\begin{equation}\label{eqn:secondphase}
\dP_{ann}(\cS_2(x)\cup \cS_{1,\emph{stop}}(x) )\geq 1- n^{-6/5}
\end{equation}
\end{proposition}

\begin{proof}
Say that $y_{q,i}$ is \textbf{spoiled} if it is met during the previous exploration of some $y_{q',j}$. Define $\cE_{5,n}:=\{\text{at least $\lfloor 1000\log n\rfloor$ vertices are spoiled or back-spoiled}\}$. We claim that for $n$ large enough,
\begin{equation}\label{eqn:secondphasespoil}
\dP_{ann}(\cE_{5,n})\leq n^{-2},
\end{equation}
and for some constant $\cviii>0$ (that only depends on $d$ and $h$), for all $q\leq m$ and $i\leq \vert \partial T_q\vert$, for every event $\cE_{q,i}\subseteq \{y_{q,i}\text{ is not spoiled}\}$ that is measurable w.r.t. the whole exploration until $y_{q,i-1}$ (or $y_{q-1,\vert \partial T_{q-1}\vert}$ if $i=1$) has been explored:
\begin{equation}\label{eqn:secondphaseexplosuccess}
\dP_{ann}(\text{the exploration from $y_{q,i}$ is successful } \vert\, \cE_{q,i})\geq \cviii.
\end{equation}
On $\cE_{5,n}^c$,  there are at least $(\cxv-1000)\log n$ non-spoiled vertices on each $\partial T_q$ (and at most $(\cxv+1)\log n$ since each step of a subexploration brings less than $\log n$ vertices to $\partial T_q$). And, if more than $\lfloor 1000\log n\rfloor $ vertices of each $\partial T_q$ are successful, one of them will be successful and not back-spoiled, thus fulfilling the requirement of $\cS_2(x)$. 
\\
By (\ref{eqn:secondphaseexplosuccess}), if $\cxv$ is large enough, for $n$ large enough, the probability that no more than $\lfloor 1000\log n\rfloor$ explorations from $\partial T_q$ are successful is at most 
\begin{center}
${\lfloor(\cxv+1)\log n\rfloor \choose \lfloor 1000\log n\rfloor}(1-\cviii)^{(\cxv-1000)\log n}\leq \left(\frac{(\cxv+1)^{\cxv+1}}{1000^{1000}(\cxv-999)^{\cxv-999}}(1-\cviii)^{\cxv-1000}\right)^{\log n}\leq n^{-4}$
\end{center}
by Stirling's formula. Hence by a union bound on $1\leq q \leq m$, noticing that $m \leq 2\cxvi/\cxv$ for $n$ large enough by \ref{D2}, we have:
\begin{center}
$\dP_{ann}((\cS_2(x)\cup \cS_{1, \text{stop}}(x))^c)\leq \dP_{ann}(\cS_1(x)^c)+\dP_{ann}(\cE_{5,n}^c)+mn^{-4}\leq 2n^{-5/4}$
\end{center}
by Proposition~\ref{prop:blockexplo} and (\ref{eqn:secondphasespoil}), and this concludes the proof. Hence, it remains to establish (\ref{eqn:secondphasespoil}) and (\ref{eqn:secondphaseexplosuccess}).

\noindent
\textbf{Proof of (\ref{eqn:secondphasespoil}).} Note that by \ref{D2}, by Remark~\ref{rem:explorationsize} and by (\ref{eqn:rndef}),
\begin{equation}\label{eqn:firsttwophasestotalvertices}
\text{less than $n^{1/2}\log^{-1}n$ vertices and half-edges have been seen in the first two phases.}
\end{equation}
In particular, less than $n^{1/2}\log^{-1}n$ edges are built during second phase. Since by \ref{D2}, we have that $\vert\{y_{q,i}, q\leq m, i\leq \vert \partial T_q\vert\}\leq \cxvi\log n$, each new edge has a probability at most 
\begin{center}
$\cxvi\log n /(n-n^{1/2}\log^{-1} n)\leq 2\cxvi n^{-1}\log n$ 
\end{center}
to spoil a vertex. Thus, the number of spoiled vertices is stochastically dominated by a random variable $Z\sim \text{Bin}(n^{1/2}\log^{-1}n,2\cxvi n^{-1}\log n)$. For $n$ large enough,
\begin{center}
$\dP(Z\geq 10)\leq {n^{1/2}\log^{-1}n \choose 10 } \left( \frac{2\cxvi \log n}{n}\right)^{10}\leq n^5 \frac{\log^{20}n}{n^{10}}\leq n^{-3}$.
\end{center}
\noindent
Moreover, by (\ref{eqn:binomdominationatinfinitydiff}) with $k=\lfloor 999\log n\rfloor$ and $m_0,m_1,m_E,m\leq n^{1/2}\log^{-1}n$ due to (\ref{eqn:firsttwophasestotalvertices}), 
\begin{center}
$\dP_{ann}(\text{more than $\lfloor 999\log n\rfloor$ vertices are back-spoiled})\leq n^{-3}$.
\end{center}
(\ref{eqn:secondphasespoil}) follows.

\noindent
\textbf{Proof of (\ref{eqn:secondphaseexplosuccess}).} 
By (\ref{eqn:firsttwophasestotalvertices}) and (\ref{eqn:binomdominationdiff}) with $k=1$ and $m_0,m_1,m_E,m\leq n^{1/2}\log^{-1}n$, for $n$ large enough,
$$
\dP_{ann}(\text{a cycle is created during the exploration from $y_{q,i}$})\leq \log^{-1}n.
$$
\noindent
The law of $\psimnt$ on the exploration tree from $y_{q,i}$ is that of $\phid$ on an isomorphical subtree of $\Td^+$ (and not $\Td$, since we do not explore towards $\overline{y_{q,i}}$), with $\phid(\circ)=\psimnt(y_{q,i})$. Denote $\cC_{\circ}^{n}$ the connected component of $\circ$ in $E_{\phid}^{\geq h+\log^{-1}n,+}\cup \{\circ\}$, and $\cZ_k$ its $k$-th generation for every $k\geq 0$. Then for any event $\cE_{q,i}\subseteq\{y_{q,i}\text{ is not spoiled} \}$ that only depends on the exploration before we inspect $y_{q,i}$,
\begin{center}
$\dP_{ann}(\text{the exploration from $y_{q,i}$ is successful}\vert \cE_{q,i})\geq p_n-\log^{-1}n$,
\end{center}
where
$p_n:=\min_{b\geq h+n^{-a}}\dP^{\Td}_b(\exists k\leq \log_{\lambda_{h}}n, \vert \cZ_k\vert \geq n^{1/2}b_n)$ (recall that $\psimnt(y_{q,i})\geq h+n^{-a}$). 
\\
Let $\delta \in (0,h_{\star}-h)$. Clearly, there exists $p'>0$ such that for $n$ large enough, 
\begin{center}
$\min_{b\geq h+n^{-a}}\dP^{\Td}_b(\exists v\in \cZ_1,\, \phid(v)\geq h+\delta)>p'$. 
\end{center}
For $\varepsilon >0$ small enough so that $\log_{d-1}(\lambda_{h+\delta}-\varepsilon) \geq (3\log_{d-1}\lambda_h)/4$ (such $\varepsilon$ exists by continuity of $h'\mapsto \lambda_{h'}$, Proposition~\ref{prop:thm43adapted}), for $n\in \dN$,
\begin{align*}
p''_n&:=\min_{b\geq h+\delta}\dP^{\Td}_b(\exists k\leq \log_{\lambda_h}\hspace{-1mm}n\,-\hspace{-1mm}1, \,\vert \cZ_k^{h+\log^{-1}n,+}\vert \geq n^{1/2}b_n) \geq \min_{b\geq h+\delta}\dP^{\Td}_b(\vert \cZ_{\lfloor\log_{\lambda_h}\hspace{-1mm}n\rfloor\,-1}^{h+\delta,+}\vert \geq n^{1/2}b_n)
\\
&\geq \min_{b\geq h+\delta}\dP^{\Td}_b(\vert \cZ_{\lfloor\log_{\lambda_h}\hspace{-1mm}n\rfloor\,-1}^{h+\delta,+}\vert \geq (\lambda_{h+\delta}-\varepsilon)^{\lfloor\log_{\lambda_h}\hspace{-1mm}n\rfloor\,-1}). 
\end{align*}
By Proposition~\ref{prop:Chlargedevgrowthrate}, $\liminf_{n\rightarrow +\infty}p''_n=:p''>0$. Since $p_n\geq p'p''_n$ for all $n\geq 1$, 
\begin{center}
$\dP_{ann}(\text{the exploration from $y_{q,i}$ is successful}\vert \text{$y_{q,i}$ is not spoiled})\geq p_n-\log^{-1}n\geq \frac{p'p''}{2}$ 
\end{center}
for $n$ large enough, and we can take $\cviii=\frac{p'p''}{2}$. This shows (\ref{eqn:secondphaseexplosuccess}). 
\end{proof}

\subsection{Third phase}\label{subsec:phase3}

\noindent
Suppose now that we are on $\cS_{2}(x)$. For $1\leq q\leq m$, denote $y_q$ one vertex of $\partial T_q$ whose exploration was successful and not back-spoiled in the second phase, $T_{y_q}$ its exploration tree, and $B_q$ the boundary of $T_{y_q}$. In this section, we connect each $B_q$ to $\Theta(n)$ vertices in a fashion similar to Section~\ref{subsec:connexion}. However, revealing the GFF on a positive proportion of the vertices would prevent us to use an approximation of the GFF as in Proposition~\ref{prop:couplinggffsexplo}. 
\\
To circumvent this difficulty, denote $R_{1,2}$ the set of vertices seen in the first two phases, and partition $V_n\setminus R_{1,2}$ arbitrarily in sets $D_1, D_2, \ldots, D_r$ for some $r\in \dN$ such that 
\begin{equation}\label{eqn:cxdef}
\vert D_1\vert =\vert D_2\vert = \ldots =\vert D_{r-1}\vert=\lfloor \cx\log n\rfloor,
\end{equation}
for some constant $\cx>0$ to be determined in Proposition~\ref{prop:thirdphaseconnexion}. We will connect each $B_q$ to a positive proportion of the vertices of $D_1$ only, with $\dP_{ann}$-probability $1-n^{-3}$ (Proposition~\ref{prop:thirdphaseconnexion}), before revealing the GFF on the vertices of the first two phases, and on the connection from $D_1$ to $B_q$ in Section~\ref{subsec:secondcomporevealpsin}. The result will follow by symmetry of the $D_i$'s and a union bound on $1\leq i \leq r-1$. 
\\

\noindent
\textbf{The exploration.} 
\\
\textbf{1) The $w$-explorations.}
Let $w_1, \ldots, w_{\lfloor \cx\log n\rfloor}$ be the vertices of $D_1$. We proceed successively to the $w$\textbf{-explorations} of $w_1, w_2, \ldots$, i.e. for $i\geq 1$, we perform the exploration from $w_i$ as in Section \ref{subsec:explo1vertexsuccess}, but stop it if we reach a vertex seen in the first two phases or in the $w$-exploration of some $w_j,\,j\leq i-1$. In particular, if $w_i$ was discovered during the exploration of some $w_j, j\leq i-1$, say then that $w_i$ is $w$-\textbf{spoiled} and do not proceed to its $w$-exploration. Denote $R_w$ the set of vertices seen during all the $w$-explorations.
\\
For $i\geq 1$, if we explore $w_i$ and \ref{C5} happens, say that the $w$-exploration from $w_i$ is $w$\textbf{-successful}. Let $s_0$ be the number of $w$-successful vertices. Let $w_{i_1}, \ldots,w_{i_{s_0}}$ be the $w$-successful vertices with $i_1<\ldots<i_{s_0}$. Let $T_{w_{i_j}}$ be the exploration tree of $w_{i_j}$, for $j\in \{1, \ldots, s_0\}$. Take away
\begin{itemize}
\item from each $\partial T_{w_{i_j}}$: the vertices that are seen in the $w$-exploration of some $w_i$, for $i>i_j$, 
\item from each $B_q$: the vertices $z$ such that $B_{\cM_n}(z,\overline{z},a_n)$  intersects $R_w$.
\end{itemize}
\noindent 
Say that those vertices are \textbf{$w$-back-spoiled}. 

\noindent
\textbf{2) The joining balls.} For $q=1, \ldots, m$ successively,
we develop balls from $B_q$ to the $\partial T_{w_{i_j}}$'s, $1\leq j\leq s_0$, with a few modifications w.r.t the construction of Section \ref{subsubsec:jointexplo}: let $z_{1,q}, z_{2,q}, \ldots$ be the vertices of $B_q$. For $z_{i,q}\in B_q$, let 
\begin{center}
$B^*_{i,q}:=\cup_{(i',q'): q'<q\text{ or } q'=q,i'<i}B^*(z_{i',q'},a'_n)$, 
\end{center}
and let $R_{i,q}:= R_{1,2}\cup R_w\cup B^*_{i,q}$ be the vertices seen before building $B^*(z_i,a'_n)$. 
\\
Replace $B^*_j$, $Q_j$ and $B_{\cM_n}(T_y,a_n)$ of Section \ref{subsec:connexion} by $B^*_{i,q}$, $R_{i,q}$ and $\cup_{j=1}^{s_0} B_{\cM_n}(T_{w_{i_j}},a_n)$ respectively.
\\
Say that $B^*(z_{i,q},a'_n)$ is a $J$\textbf{-joining ball} if it hits $B_{\cM_n}(T_{w_{i_J}},a_n)$ at one vertex after $a'_n-2a_n$ steps, and no other intersection with vertices seen previously is created.

\begin{proposition}\label{prop:thirdphaseconnexion}
For $n$ large enough, we have
\begin{equation}\label{eqn:thirdphasejoiningballs}
\dP_{ann}(\cS_2(x)\cap \{\exists (J,q), \text{ there are less than $\log^{\gamma-3\kappa -18}n$ $J$-joining balls from $B_q$} \})\leq n^{-4}
\end{equation}
and if $\cx$ is large enough, then for large enough $n$:
\begin{equation}\label{eqn:wsuccessful}
\dP_{ann}(\cS_2(x)\cap \{s_0\leq \log n\})\leq n^{-4}.
\end{equation}
\end{proposition}

\begin{proof} 
\textbf{Proof of (\ref{eqn:thirdphasejoiningballs}).}
We adapt the proof of Lemma~\ref{lem:joiningball}. Since there are less than $\log^2n$ $w_i$'s (and $\vert R_{1,2}\vert$ is controlled by (\ref{eqn:firsttwophasestotalvertices})), we can replace (\ref{eqn:explototaledges}) by
\begin{equation}\label{eqn:thirphasetotaledges}
\text{less than $n^{1/2}\log^{\gamma +2}n$ vertices are seen during the three phases}.
\end{equation}
\noindent 
Let $B^*:=\cup_{q=1}^{m}\cup_{z\in B_q}B^*(z,a'_n)$. (\ref{eqn:saturationjoiningballs}) becomes
\begin{equation}\label{eqn:thirdphasesaturationjoiningballs}
\dP_{ann}(\cS_2(x)\cap \{\,\vert B^* \cap(\cup_{j=1}^{s_0} B_{\cM_n}(T_{z_{i_j}},a_n)) \vert\geq\log^{3\gamma}n\})\leq n^{-5}
\end{equation}

\noindent
Let $N$ be the number of vertices of $\cup_{q=1}^{m}B_q$ that are spoiled, i.e. the vertices $z$ such that $B^*(z,a_n)=B_{\cM_n}(z,\overline{z},a_n)$ is hit by a previously constructed $B^*(z',a'_n)$. (\ref{eqn:spoliationjoiningballs}) becomes
\begin{equation}\label{eqn:thirdphasespoiledinthejunction}
\dP_{ann}(\cS_2(x)\cap \{N\geq \log^{3\gamma }n\})\leq n^{-5}.
\end{equation}

\noindent
In addition, by (\ref{eqn:binomdominationatinfinitydiff}) with $k=\log^2n$, $m_0=1$, $m_1,m_E,m\leq n^{1/2}\log^{-1/2}n$, 
\begin{equation}\label{eqn:backspoiledzis}
\dP_{ann}\left(\cE_{6,n} \right)\leq n^{-5},
\end{equation}

\noindent
where $\cE_{6,n}:=\cS_2(x)\cap\left\lbrace\text{more than $\log^2 n$ vertices are $w$-back-spoiled} \right\rbrace $. 
Then, define 
\begin{center}
$\cS_{i,q}:=\cS_2(x)\cap\cE_{6,n}^c\cap \{ \vert (\cup_{j=1}^{s_0}B_{\cM_n}(T_{w_{i_j}},a_n)) \cap B^*_{i,q} \vert \leq \log^{3\gamma}n\}\cap \{z_{i,q}\text{ is not spoiled}\}$. 
\end{center}
It is straightforward to adapt the proof of (\ref{eqn:p1p2p3}) to get that for every $1\leq J\leq s_0$, the probability that $B^*(z_{i,q},a'_n)$ is a $J$-joining ball, conditionally on $\cS_{i,q}$, is at least $ n^{-1/2}\log^{\gamma-2\kappa-10}n$. On $\cS_2(x)\cap \{N\leq \log^{3\gamma}n\}\cap \cE_{6,n}^c $, at least 
\begin{center}
$n^{1/2}b_n-\log^2n- \log^{3\gamma}n\geq n^{1/2}\lfloor\log^{-\kappa-7}n\rfloor$ 
\end{center}
$z_{i,q}$'s are neither spoiled nor $w$-back-spoiled by (\ref{eqn:bndef}). As in the end of the proof of Lemma~\ref{lem:joiningball}, if $\gamma >3\kappa +18$, we get that for $n$ large enough: for every $(J,q)\in (\{1, \ldots, s_0\}\cap \{1, \ldots, m\}) $ and $Z\sim \text{Bin}(n^{1/2}\lfloor\log^{-\kappa-7}n\rfloor,n^{-1/2}\log^{\gamma-2\kappa-10}n )$:
\begin{align*}
&\dP_{ann}(\cS_2(x)\cap \{\text{there are less than $\log^{\gamma-3\kappa -18}n$  $J$-joining balls from $B_q$}\} )
\\
&\leq \dP_{ann}\left(\cS_2(x)\cap ( \{\,\vert B^* \cap(\cup_{j=1}^{s_0} B_{\cM_n}(T_{z_{i_j}},a_n)) \vert\geq\log^{3\gamma}n\}  \cup \{N\geq \log^{3\gamma }n\}\cup \cE_{6,n})\right)
\\
&+ \dP(Z\leq  \log^{\gamma-3\kappa -18}n)
\\
& \leq 3n^{-5}+\log^{\gamma}n\max_{k \leq  \log^{\gamma-3\kappa -18}n}\dP(Z=k)\,\,\,\text{ by (\ref{eqn:thirdphasesaturationjoiningballs}), (\ref{eqn:thirdphasespoiledinthejunction}) and (\ref{eqn:backspoiledzis})}
&\leq  4n^{-5}
\end{align*}
as below \eqref{eqn:p1p2p3}.
Since $s_0\leq \vert D_1\vert \leq \log^2n$ by (\ref{eqn:cxdef}) and $m\leq \cxvi\log n$ by \ref{D2}, a union bound on $(J,q)$ yields (\ref{eqn:thirdphasejoiningballs}).
\\

\noindent
\textbf{Proof of (\ref{eqn:wsuccessful}).} We now estimate the probability that at most $\log n$ $w$-explorations are $w$-successful. Note that by Remark~\ref{rem:explorationsize} and (\ref{eqn:firsttwophasestotalvertices}),
\begin{equation}\label{twophasesandwexplototaledges}
\vert R_{1,2} \vert +\vert R_w\vert \leq n^{1/2}\log^{-1/2}n.
\end{equation}

\noindent 
Hence if $C>0$ is large enough, by (\ref{eqn:binomdominationatinfinitydiff}) with $k=\lfloor C\log n\rfloor$ and $m_0,m_1,m_E,m\leq n^{1/2}\log^{-1/2}n$,
\begin{equation}\label{eqn:spoiledzis}
\dP_{ann}\left(\cS_2(x)\cap \left\lbrace\text{more than $C\log n$ $w_i$'s are $w$-spoiled}\right\rbrace\right)\leq n^{-5}.
\end{equation}
\noindent
Moreover, for every $i\geq 1$, conditionally on the fact that $w_i$ is not $w$-spoiled,  the probability that the $w$-exploration from $w_i$ is stopped because it reaches a vertex of $R_{1,2}$ or a vertex seen in the exploration of some $w_j$, $j<i$ is $o(1)$ by (\ref{eqn:binomdominationdiff}) with $k=1$, $m_0=1$, $m_1,m_E,m\leq  n^{1/2}\log^{-1/2}n$. Hence, a straightforward adaptation of the proof of (\ref{eqn:growthsuccessful}) yields
\begin{equation}\label{eqn:zexplorationfirst}
\dP_{ann}(\text{the exploration from $w_i$ is $w$-successful }\vert \,\cS_{2}(x)\cap\{w_i\text{ is not $w$-spoiled}\})\geq \eta(h)/2.
\end{equation}

\noindent
Take $\cx >3C$. By (\ref{eqn:spoiledzis}), (\ref{eqn:zexplorationfirst}) and (\ref{eqn:cxdef}), if $Z\sim \text{Bin}( \lfloor (\cx\log n)/2\rfloor, \eta(h)/2)$ and $n$ is large~enough,
\begin{center}
$\dP_{ann}(\cS_2(x)\cap\{s_0\leq \log n\})\hspace{-1mm}\leq n^{-5}+\dP(Z\leq \log n).
$
\end{center}
\noindent
One checks via Stirling's formula that if $\cx$ is large enough, then for large enough $n$, 
\begin{align*}
\dP(Z\leq \log n)&\leq 2\log n \max_{ k\leq \log n}\dP(Z=k)\leq 2\log n{\cx \log n\choose \log n}(1-\eta(h)/2)^{(\cx-1)\log n } \leq n^{-5},
\end{align*}
and (\ref{eqn:wsuccessful}) follows.

\end{proof}

\subsection{Revealing $\psimn$ on the three phases}\label{subsec:secondcomporevealpsin}

\noindent
Let $R_1^{\psi}$ (resp. $R_2^{\psi}$) be the set of vertices where $\psimnt$ has been defined during the first (resp. second) phase, and $R_3^{\psi}$ be the set of vertices in the $w$-successful $w$-explorations and on the $J$-joining balls, for all $1\leq J\leq m$, on which we will realize $\psimn$ on the third phase. 
\\
By Proposition~\ref{prop:gffann}, we can realize $\psimn$ jointly with $\cM_n$ by\begin{itemize}
\item proceeding to the three phases of the exploration from $x$,
\item revealing the remaining pairings of half-edges of the $\cM_n$,
\item defining $\psimn$ on $R_1^{\psi}\cup R_2^{\psi}$, in the same order as $\psimnt$ has been defined, using the same standard normal variables: we let 
\begin{center}
$\psimn(x):=\psimnt(x)\sqrt{\frac{d-2}{d-1}}\sqrt{\gmn(x,x)}$ 
\end{center}
if $B_{\cM_n}(x,r_n+\ell)$ has no cycle, and $\psimn(x)=\psimnt(x)\sqrt{\gmn(x,x)/G_{G_m}(u_k,u_k)}$ if it has one cycle $C$, and $d_{\cM_n}(x,C)=k$. For every $y$, we define the event 
\\
$A_y:=\{z\in V_n,\, \psimnt(z)\text{ was defined before }\psimnt(y)\}$ and
\begin{center}
$\psimn(y):=\dE^{\cM_n}[\psimn(y)\vert \sigma(A_y)]+\xi_y\sqrt{\Var(\psimn(y)\vert \sigma(A_y))},$
\end{center}
for every $y\in R_1^{\psi}$, where $\xi_y$ is the normal variable used when defining $\psimnt(y)$,
\item revealing $\psimn$ on $R_3^{\psi}$, and finally on $V_n\setminus (R_1^{\psi}\cup R_2^{\psi}\cup R_3^{\psi})$.
\end{itemize} 

\noindent
Let $\cE_{7,n}:=\{\cM_n\text{ is a good graph}\}\cap \{\max_{z\in V_n}\vert \psimn(z)\vert \leq \log^{2/3}n\} $,
\\
$\cS_1^{\psi}(x):= \cS_1(x)\cap  \{\sup_{y\in R_1^{\psi}}\vert \psimn(y)-\psimnt(y)\vert \leq n^{-a}/2\}$,
\\
$\cS_2^{\psi}(x):= \cS_2(x)\cap \cS_1^{\psi}(x)\cap \{\sup_{y\in R_2^{\psi}}\vert \psimn(y)-\psimnt(y)\vert \leq (\log^{-1}n)/2\} $, 
and
\\
$\cS_{3,i}^{\psi}(x):= \cS_2^{\psi}(x)\cap  \{\forall q\in \{1, \ldots, m\}\text{ at least $\log n$ vertices of $D_i$ are connected to $T_q$ in $\Enh$}\} $ for every $i\geq 1$. 

\noindent
Suppose that for $a>0$ (defined in the beginning of Section~\ref{subsec:phase1}) small enough (depending only on $d$ and $h$), and for $n$ large enough:
\begin{equation}\label{eqn:firstphasepsinapprox}
\dP_{ann}((\cS_1(x)\setminus \cS_1^{\psi}(x))\cap \cE_{7,n})\leq n^{-3}, 
\end{equation}

\begin{equation}\label{eqn:secondphasepsinapprox}
\dP_{ann}((\cS_2(x)\setminus \cS_2^{\psi}(x))\cap \cE_{7,n})\leq n^{-3}, 
\end{equation}
and for every $ 1\leq i\leq r-1$,
\begin{equation}\label{eqn:thirdphasepsinapprox}
\dP_{ann}((\cS_2(x)\setminus \cS_{3,i}^{\psi}(x))\cap \cE_{7,n})\leq n^{-3}.
\end{equation}
Letting $\cS_{1,\text{stop}}^{\psi}(x):= \cS_{1,\text{stop}}(x)\cap  \cS_1^{\psi}(x)$, (\ref{eqn:firstphasepsinapprox}), (\ref{eqn:secondphasepsinapprox}), (\ref{eqn:thirdphasepsinapprox}) and (\ref{eqn:secondphase}) imply that 
\begin{equation}\label{eqn:threephasespsin}
\dP_{ann}((\cS^{\psi}_{1,\text{stop}}(x) \cup (\cap_{i= 1}^{r-1}\cS_{3,i}^{\psi}(x))\,)\cap \cE_{7,n})\geq 1-n^{-7/6}.
\end{equation}
On $\cS^{\psi}_{1,\text{stop}}(x) \cup (\cap_{i=1}^{r-1}\cS_{3,i}^{\psi}(x))$, we have the following alternative:
\begin{itemize}
\item either $\cC_x^{\cM_n,h}$ contains a subexploration tree $T_q$ whose exploration was fertile, the exploration from $y_q$ is successful and connected to at least $\log n$ vertices of every $D_i$, $1\leq i\leq r-1$ in $\Enh$;
\item or $\cC_x^{\cM_n,h}$ contains no such tree, and $\cC_x^{\cM_n,h}\subseteq R_1^{\psi}$, so that $\vert \cC_x^{\cM_n,h}\vert \leq \cz\log n$, where we take $\cz \geq \cxvi$.
\end{itemize}
Note that the second case comprises $\cS_{1,\text{stop}}^{\psi}(x)$ but is a priori not included in it: there could exist fertile subexploration trees not connected to $x$ in  $\Enh$ if $\psimn$ is below $h$ on the appropriate marked vertices. 
\\
In the first case, $\cC_x^{\cM_n,h}$ contains at least $\log n$ vertices of each $D_i$, $1\leq i\leq r-1$, so that by (\ref{eqn:firsttwophasestotalvertices}) and (\ref{eqn:cxdef}) for $n$ large enough:
$$\vert \cC_x^{\cM_n,h}\vert \geq (r-1) \log n\geq \log n \,\,\frac{n-\vert R_{1,2}\vert-\vert D_r\vert}{\cx\log n} \geq \frac{n}{2\cx}.$$
\noindent
Letting thus $\cxi:=(2\cx)^{-1}$, for every $n$ large enough, we have on $\cS^{\psi}_{1,\text{stop}}(x) \cup (\cap_{i= 1}^{r-1}\cS_{3,i}^{\psi}(x))$:
$$\vert \cC_x^{\cM_n,h}\vert\leq \cz\log n\text{ or } \vert \cC_x^{\cM_n,h}\vert\geq \cxi n.$$

\noindent
By (\ref{eqn:threephasespsin}) and a union bound on all $x\in V_n$, 
\begin{center}
$ \dP_{ann}(K_0\log n\leq\vert \cC_2^{(n)}\vert \leq \cxi n)\leq n^{-1/6}+\dP_{ann}(\cE_{7,n}^c)$. 
\end{center}
By Remark~\ref{rem:secondcompofirstbound}, $\vert \cC_2^{(n)}\vert \leq \cxi n$ $\dP_{ann}$-w.h.p. By Proposition~\ref{prop:goodgraph} and Lemma~\ref{lem:maxgff}, we have $\dP_{ann}(\cE_{7,n}^c)\rightarrow 0$ so that
\begin{center}
$\dP_{ann}(\vert \cC_2^{(n)}\vert \leq K_0\log n )\rightarrow 1$,
\end{center}
yielding (\ref{eqn:secondcompo}). Hence, it remains to show (\ref{eqn:firstphasepsinapprox}), (\ref{eqn:secondphasepsinapprox}) and (\ref{eqn:thirdphasepsinapprox}).

\subsubsection{The field $\psimn$ on the first phase: proof of (\ref{eqn:firstphasepsinapprox})}

\begin{proposition}\label{prop:blockexplogoodapprox}
Let $a,\cz, \ell$ be such that Proposition~\ref{prop:blockexplo} holds with $\cxvi=\cz$ and $\cxv$ as in Proposition~\ref{prop:secondphase}, and such that the conclusion of Lemmas~\ref{lem:couplinggffadapted} and~\ref{lem:couplinggffadapted1cycle} hold. Then for $n$ large enough, (\ref{eqn:firstphasepsinapprox}) holds. 
%
%

\end{proposition}

%
%

\noindent
To prove Proposition~\ref{prop:blockexplogoodapprox}, we need two variants of Proposition~\ref{prop:couplinggffsexplo}, whose proofs are postponed to the~\hyperref[appn]{Appendix}, Section~\ref{subsec:AppendixSection7}. The first consists in replacing the ``security radius'' $a_n=\Theta(\log\log n)$ by $r_n=\Theta(\log n)$.

\begin{lemma}\label{lem:couplinggffadapted}
If $a>0$ is small enough (depending only on $d$ and $h$), then the following holds for $n$ large enough. Assume that $\cM_n$ is a good graph, that $A\subset V_n$ satisfies 
\begin{itemize}
\item $\vert A\vert \leq n^{2/3}$,
\item $\text{\emph{tx}}(B_{\cM_n}(A,r_n))=\text{\emph{tx}}(A)$, and
\item $ \max_{z\in A}\vert\psimn(z)\vert\leq \log^{2/3}n$.
\end{itemize}
For every $y\in \partial B_{\cM_n}(A,1)$, writing $\overline{y}$ for the unique neighbour of $y$ in $A$, we have:
\begin{equation}\label{eqn:condexpapproxadapted}
\left\vert\dE^{\cM_n}[\psimn(y)\vert \sigma(A_y)]-\frac{1}{d-1}\psimn(\overline{y})\right\vert \leq n^{-2a}
\end{equation}

and

\begin{equation}\label{eqn:condvarapproxadapted}
\left\vert \emph{\text{Var}}^{\cM_n}(\psimn(y)\vert \sigma(A_y))-\frac{d}{d-1}\,\,\right\vert\leq n^{-2a}.
\end{equation}

\end{lemma}

\noindent
The second variant of Proposition~\ref{prop:couplinggffsexplo} corresponds to Case 1 and Case 2 of the first phase of the exploration.
Let us introduce some notations before stating the Lemma. 
Recall that for $m\geq 3$, $G_m$ is the connected (and infinite) $d$-regular graph with a unique cycle $C_m$ of length $m$. Recall the definitions of $\alpha_k,\beta_k,\gamma_k, \alpha'_k$ and $\gamma'_k$ from (\ref{eqn:blockexplocase1}) and (\ref{eqn:blockexplocase2}).
\\
For $k\geq 0$, let $z_{k}$ be a vertex at distance $k$ of the cycle $C_m$ in $G_m$, and $\overline{z_{k}}$ be a neighbour of $z_{k}$ at distance $k+1$ of $C_m$. Note that $B_{G_m}(z_k,\overline{z_{k}},r_n)$ (the subgraph of $G_m$ obtained by taking all paths of length $r_n$ starting at $z_{k}$ and not going through $\overline{z_{k}}$) contains $C_m$ if and only if $k\leq r_n-\lceil m/2\rceil$.

\begin{lemma}\label{lem:couplinggffadapted1cycle}
If $a>0$ is small enough (depending only on $d$ and $h$), then the following holds for $n$ large enough (uniformly in $m\geq 3$).
\\
Assume that $\cM_n$ is a good graph, that $A\subseteq V_n$ is such that 
\begin{itemize}
\item $\vert A\vert \leq n^{2/3}$,
\item $A$ is a tree, and
\item $ \max_{z\in A}\vert\psimn(z)\vert\leq \log^{2/3}n$.
\end{itemize}

\noindent
\textbf{Case 1.} Let $y\in V_n$, and suppose that
\begin{itemize}
\item[--] $y$ has a neighbour $\overline{y}$ in $A$, 
\vspace{-2mm}
\item[--] for some $1\leq k<m$, there exists $\widehat{y}$ in $A$, a path $P$ of length $m-k$ from $y$ to $\widehat{y}$ whose only vertex in $A$ is $\widehat{y}$, and a path $P'$ in $A$ of length $k-1$ from $\widehat{y}$ to $\overline{y}$, so that $C:=P\cup P'\cup (\overline{y},y)$ is a cycle of length $m$ (and $\widehat{y}=\overline{y}$ if $k=1$), and 
\vspace{-2mm}
\item[--] $\emph{\tx}(B_{\cM_n}(A\cup C,r_n))=1$. 
\end{itemize}
Then
\begin{equation}\label{eqn:condexpapproxadaptedcase1}
\left\vert\dE^{\cM_n}[\psimn(y)\vert \sigma(A)]- \alpha_{k}\psimn(\overline{y})-\beta_k\psimn(\widehat{y})\right\vert \leq n^{-2a}
\end{equation}
and
\begin{equation}\label{eqn:condvarapproxadaptedcase1}
\left\vert \emph{\text{Var}}^{\cM_n}(\psimn(y)\vert \sigma(A))-\gamma_k\,\,\right\vert\leq n^{-2a}.
\end{equation}

\noindent
\textbf{Case 2.} Let $y\in V_n$, and suppose that
\begin{itemize}
\item[--] $y$ has a unique neighbour $\overline{y}$ in $A$, 
\vspace{-2mm}
\item[--] for some $1\leq k\leq r_n-\lceil m/2\rceil$, $B_{\cM_n}(y,\overline{y},r_n)$ is isomorphic to $B_{G_m}(z_k,\overline{z_{k}},r_n)$, and
\vspace{-2mm}
\item[--] $\emph{\tx}(B_{\cM_n}(A\cup P\cup C,r_n))=1$, where $C$ is the cycle in $B_{\cM_n}(y,\overline{y},r_n)$ and $P$ the path from $y$ to $C$.
\end{itemize}
Then
\begin{equation}\label{eqn:condexpapproxadaptedcase2}
\left\vert\dE^{\cM_n}[\psimn(y)\vert \sigma(A)]- \alpha'_{k}\psimn(\overline{y})\right\vert \leq n^{-2a}
\end{equation}
and
\begin{equation}\label{eqn:condvarapproxadaptedcase2}
\left\vert \emph{\text{Var}}^{\cM_n}(\psimn(y)\vert \sigma(A))-\gamma'_k\,\,\right\vert\leq n^{-2a}.
\end{equation}

\end{lemma}

\begin{proof}[Proof of Proposition~\ref{prop:blockexplogoodapprox}] We proceed as below (\ref{eqn:growthsuccessful}) in the proof of Proposition~\ref{prop:explo1vertex}. Denote 
\begin{center}
$\cE_n:=\{\exists y\in R_1^{\psi},\vert \psimn(y) -\psimnt(y)\vert \geq n^{-a} \}\cap \cE_{7,n}\cap \cS_1(x)$.
\end{center}
%
On $\cE_n$, $\psimnt$ is defined on at most $\cz\log n$ vertices by~\ref{D2} (and our choice of $\cxvi=\cz$), so that by the triangle inequality, either $\vert \psimnt(x)-\psimn(x)\vert \geq n^{-a}\log^{-2}n$, or there exists $y$ such that 
\begin{center}
$\vert \psimnt(y) - \psimn(y)\vert \geq n^{-a}\log^{-2}n+\sup_{\overline{y}\in R_y^{\psi}}\vert \psimnt(\overline{y}) - \psimn(\overline{y})\vert$, 
\end{center}
where $R_y^{\psi}$ is the set of vertices where $\psimnt$ has been defined before $\psimnt(y)$. Let
\begin{center}
$\cE(y):=\{\vert \psimnt(y) - \psimn(y)\vert \geq n^{-a}\log^{-2}n+\sup_{\overline{y}\in R_y^{\psi}}\vert \psimnt(\overline{y}) - \psimn(\overline{y})\vert \}\cap \cE_n$
\end{center}
For $y\neq x$, we can apply Lemma~\ref{lem:couplinggffadapted1cycle} (if $\psimnt(y)$ is defined in Case 1 or Case 2) or Lemma~\ref{lem:couplinggffadapted} (in the other cases), so that
$$
\dP_{ann}(\cE(y))\leq \dP_{ann}(n^{-2a}\vert \xi_y\vert \geq n^{-a}\log^{-2}n-n^{-2a})\leq \dP_{ann}(\vert \xi_y\vert \geq n^{a/2})\leq n^{-4}
$$
by the exponential Markov inequality, and $\dP_{ann}(\cE(x))\leq n^{-4}$ by the same argument, where we set $\cE(x):=\{\vert \psimnt(x)-\psimn(x)\vert \geq n^{-a}\log^{-2}n\}\cap \cE_n$. For the particular case that there is a cycle $C$ in $B_{\cM_n}(x,r_n+\ell)$, at distance $k\geq 0$ from $x$, a computation similar to that of the variance in Lemma~\ref{lem:couplinggffadapted1cycle} shows that for $a$ small enough (depending on $d$ and $h$ only), $\max_{k\geq 0}\vert G_{G_m}(u_k,u_k) -\gmn(x,x)\vert \leq n^{-2a}$.
Hence, we get by a union bound on $y\in R_1^{\psi}$ that for large enough $n$, 
$$
\dP_{ann}(\cE_n)\leq n^{-4} \cz\log n\leq n^{-3}.
$$
The conclusion follows.
\end{proof}

\subsubsection{The field $\psimn$ on the second phase: proof of (\ref{eqn:secondphasepsinapprox}) }
\noindent
It is enough to show that for each $y_{q,i}$ whose exploration in the second phase is successful, 
\begin{center}
$\dP_{ann}\left(\lbrace \sup_{ z\in T_{y_{q,i}}\setminus\{y_{q,i}\} }\vert \psimnt(z)-\psimn(z)\vert \geq \frac{\log^{-1}n}{2} \rbrace\cap \cE_{7,n}\right)\leq n^{-4} $,
\end{center}
where $T_{y_{q,i}}$ is its exploration tree, and to conclude by a union bound on $y_{q,i}$. 
This follows from a straightforward adaptation of the reasoning below (\ref{eqn:growthsuccessful}). Note that the $n^{-3}$ in the RHS of (\ref{eqn:diffgffgraphtree}) can be replaced by any polynomial in $n$. 

\subsubsection{The field $\psimn$ on the third phase: proof of (\ref{eqn:thirdphasepsinapprox})}
\noindent
By symmetry, it is enough to consider the case $i=1$. Following readily the argument of the proof of (\ref{eqn:joinTj}), we get that the probability that $\Enh$ percolates through a given $J$-joining ball is at least $\log^{\gamma(\cxiii/3 -1)}n$, for any $J$. By (\ref{eqn:thirdphasejoiningballs}) and (\ref{eqn:wsuccessful}), and a union bound on every couple $(J,q)\in \{1, \ldots, s_0\}\times \{1, \ldots m\}$, 
\begin{center}
$\dP_{ann}((\cS_2(x)\setminus \cS^{\psi}_{3,1}(x))\cap \cE_{7,n}) \leq 2n^{-4}+s_0m\dP(Z=0)$,
\end{center}
where $Z\sim \text{Bin}(\lfloor\log^{\gamma - 3\kappa-18}n\rfloor, \log^{\gamma(\cxiii/3 -1)}n )$. If $\kappa$ and $\gamma/\kappa$ are large enough, then
$$\dP(Z=0)=(1-\log^{\gamma(\cxiii/3 -1)}n )^{\lfloor\log^{\gamma - 3\kappa-18}n\rfloor}\leq n^{-4}$$
for $n$ large enough,  and (\ref{eqn:thirdphasepsinapprox}) follows.

\section{Properties of $\cC_1^{(n)}$}\label{sec:properties}
\subsection{The local limit: proof of Theorem~\ref{thm:locallimit}}
\begin{proof}[Proof of Theorem~\ref{thm:locallimit}]
The proof mimics the reasoning of Lemma~\ref{lem:connexionsupperbound}. Let $k\geq 0$ and let $T$ be a rooted tree of height $k$, with no vertex of degree more than $d$. Let $x\in V_n$. We perform an exploration as in Section~\ref{subsec:explo1vertexsuccess}. Denote $\cS_T(x)$ the event that the exploration is successful, that $\cS_T(x)\subseteq\cC_x^{\cM_n,h}$ and that $B_{\cM_n}(x,k)$ is isomorphic to $T$. We claim that
\begin{equation}\label{eqn:locallimexpect}
\dP_{ann}(\cS_T(x)) \underset{n\rightarrow +\infty}{\longrightarrow} \dP^{\Td}(B_{\Td}(\circ,k)=T, \vert \Ch\vert =+\infty).
\end{equation}
The proof goes as those of Lemma~\ref{lem:explo1vertex} and Proposition~\ref{prop:explo1vertex}. In the proof of Lemma~\ref{lem:explo1vertex}, replace $\cF^{(n)}_j$ by $\cF^{(n)}_{T,j}:=\cF^{(n)}_j\cap \{B_{\Ch}(\circ,k)=T\}$ and $\mathcal{F'}_j^{(n)} $ by $\mathcal{F'}_{T,j}^{(n)}:=\mathcal{F'}_j^{(n)}   \cup \{ B_{\Ch}(\circ,k)\neq T\}$, for every $j\geq 1$. We also check easily that 
\begin{center}
$\dP^{\Td}(\Chlnplus\cap B_{\Td}(\circ,k)=\Chlnminus\cap B_{\Td}(\circ,k))\rightarrow 1$
\end{center}
in order to determine $B_{\cM_n}(x,k)$ $\dP_{ann}$-w.h.p., as Proposition~\ref{prop:explo1vertex} only ensures thatwe have w.h.p.~the inclusion $\Chlnplus\cap B_{\Td}(\circ,k)\subseteq B_{\cM_n}(x,k)$.
Moreover, we get as for (\ref{eqn:variancenbaborted}):
\begin{equation}\label{eqn:locallimcovar}
\sup_{x,y\in V_n}\vert \Cov_{ann}(\cS_T(x),\cS_T(y))\vert \underset{n\rightarrow +\infty}{\longrightarrow}0.
\end{equation}
\noindent
Let $\varepsilon>0$. Applying Bienaym\'e-Chebyshev's inequality as in Lemma~\ref{lem:connexionsupperbound}, we get
\begin{equation}\label{eqn:locallimstx}
\dP_{ann}( \vert \,\vert\cS_T\vert - \dP^{\Td}(B_{\Td}(\circ,k)=T, \vert \Ch\vert =+\infty)n\vert \leq \varepsilon n)\underset{n\rightarrow +\infty}{\longrightarrow}1,
\end{equation}
where $\cS_T$ is the set of vertices $x\in V_n$ such that $\cS_T(x)$ holds. Let $\cS\subset V_n$ be the set of vertices such that their exploration is successful. By Theorem~\ref{thm:maingff} and a reasoning as in the proof of Lemma~\ref{lem:connexionsupperbound}, with $\dP_{ann}$-probability $1-o(1)$,
\vspace{-3mm}

\begin{enumerate}[label=(\roman*)]
\item\label{i} $\vert \cC_2^{(n)}\vert \leq n^{1/3}$, so that $\cS_T\subseteq \cS\cap  \cC_1^{(n)}$,
\item\label{ii} $\vert \,\vert \cC_1^{(n)}\vert -\eta(h)n\vert\leq \varepsilon n$,
\item\label{iii} $\vert \,\vert \cS\vert -\eta(h)n\vert\leq \varepsilon n$.
\end{enumerate}
By \ref{i}, $\cS_T =\cS\cap V_n^{(T)}$. Thus,
$\vert \cS_T\vert \leq \vert V_n^{(T)}\vert \leq \vert \cS_T\vert + \vert \cC_1^{(n)}\cap\cS ^c\vert$, so that by \ref{ii}, \ref{iii} and (\ref{eqn:locallimstx}), 
$$
(\dP^{\Td}(B_{\Td}(\circ,k)=T, \vert \Ch\vert =+\infty)-\varepsilon)n \leq  \vert V_n^{(T)}\vert\leq (\dP^{\Td}(B_{\Td}(\circ,k)=T, \vert \Ch\vert =+\infty)+3\varepsilon)n.
$$
Moreover, $\vert \,\vert \cC_1^{(n)}\vert -\eta(h)n\vert\leq \varepsilon n $ by \ref{ii}, so that for $\varepsilon $ small enough,
$$
\frac{\dP^{\Td}(B_{\Td}(\circ,k)=T, \vert \Ch\vert =+\infty)}{\eta(h)}-\sqrt{\varepsilon}\leq \frac{\vert V_n^{(T)}\vert}{\vert \cC_1^{(n)}\vert} \leq \frac{\dP^{\Td}(B_{\Td}(\circ,k)=T, \vert \Ch\vert =+\infty)}{\eta(h)}+\sqrt{\varepsilon}.
$$
Since $\eta(h)=\dP^{\Td}(\vert\Ch\vert =+\infty)$, we have
\begin{align*}
\frac{\dP^{\Td}(B_{\Td}(\circ,k)=T, \vert \Ch\vert =+\infty)}{\eta(h)} &=\dP^{\Td}(B_{\Td}(\circ,k)=T \,\vert\,\, \vert \Ch\vert =+\infty).
\end{align*}
And since we can take $\varepsilon$ arbitrarily small, the conclusion follows.

\end{proof}

\subsection{The core and the kernel: proof of (\ref{eqn:core}) and (\ref{eqn:kernel})}
We now prove (\ref{eqn:core}) and (\ref{eqn:kernel}) of Theorem~\ref{thm:globalstruct}, starting with (\ref{eqn:core}). Let $\ci$ (resp. $\cii$) be the probability that $\circ$ has at least 2 (resp. 3) children with an infinite offspring in $\Ch$, under $\dP^{\Td}$. Then \eqref{eqn:core} follows by a second moment argument as in Lemma~\ref{lem:connexionsupperbound} once we show that for $x,y\in V_n$,
\begin{equation}\label{eqn:coreexpectation}
\dP_{ann}(x\in \bC^{(n)})\underset{n\rightarrow +\infty}{\longrightarrow} \ci, \,\text{ and}
\end{equation}
\begin{equation}\label{eqn:corevariance}
\text{Cov}_{ann}(\mathbf{1}_{x\in \bC^{(n)}}, \mathbf{1}_{y\in \bC^{(n)}}) \underset{n\rightarrow +\infty}{\longrightarrow} 0.
\end{equation}

\begin{proof}[Proof of (\ref{eqn:coreexpectation})]
We proceed in three steps: first, we show that with probability $\simeq K_1$, the exploration of Section~\ref{subsec:explo1vertexsuccess} is successful for two children of $x$. Second, we prove that some vertex $w$ is connected to these two explorations, forming a cycle with $x$. Third, we adapt the lower exploration of Section~\ref{subsec:explo1vertexaborted} to show reciprocally that $x$ is not in $\bC^{(n)}$ with probability $\gtrsim 1-K_1$.
\\
\textbf{First step.}
For $x\in V_n$, we perform the exploration in Section~\ref{subsec:explo1vertexsuccess} from $x$, replacing \ref{C5} by the following condition: for every neighbour $v$ of $x$, stop exploring the subtree from $v$ at step $k+1$ if the $k$-offspring of $v$ has at least $n^{1/2}b_n$ vertices. Stop the exploration if this happens for at least two neighbours of $x$. In this case, say that the exploration is successful. 
\\
We adapt easily the proofs of Lemma~\ref{lem:explo1vertex} and Proposition~\ref{prop:explo1vertex} to show that for $x\in V_n$,
\begin{equation}\label{eqn:coreexplosuccess}
\dP_{ann}(\text{the exploration from $x$ is successful})\underset{n\rightarrow +\infty}{\longrightarrow}\ci. 
\end{equation}
\noindent
Indeed, $\ci$ is the probability that the realization of $\phid$ to which we couple $\psimn$ is such that $\circ$ has at least two children with an infinite offspring. Then, as in Lemma~\ref{lem:explo1vertex}, there is a probability $1-o(1)$ that the offspring of these children grows at an exponential rate close to $\lambda_h$ (Proposition~\ref{prop:Chlargedevgrowthrate}). Thus, letting $\cF^{(n)}_{v,k}:=\cup_{1\leq j\leq k}\{\text{the $j$-offspring of $v$ has at least $n^{1/2}b_n$ vertices}\}$ for every child $v$ of $\circ$, and $\cF^{(n)}_{\text{core},k}:=\cup_{v_1,v_2\text{ children of $\circ$}}(\cF^{(n)}_{v_1,k-1}\cap \cF^{(n)}_{v_2,k-1})$, we get that 
\begin{center}
$\dP_{ann}(\cup_{1\leq k \leq \lfloor \log_{\lambda_h}n\rfloor}\cF^{(n)}_{\text{core},k})\underset{n\rightarrow +\infty}{\longrightarrow}\ci$, and
\end{center}
\begin{center}
$\dP_{ann}(\cup_{1\leq k \leq \lfloor \log_{\lambda_h} n\rfloor}\{\text{at most one child of $\circ$ has a non-empty $k$-offspring}\})\underset{n\rightarrow +\infty}{\longrightarrow}1-\ci$.
\end{center}
As for (\ref{eqn:C1happensexplo}), we do not meet any cycle with $\dP_{ann}$-probability $1-o(1)$. This yields (\ref{eqn:coreexplosuccess}). 
\\
\textbf{Second step.}
If the exploration from $x$ is successful, let $x_1,x_2$ be two children of $x$ such that the exploration subtrees $T_{x_1}$ and $T_{x_2}$ from $x_1$ and $x_2$ satisfy $\min(\vert \partial T_{x_1}\vert, \, \vert\partial T_{x_2}\vert)\geq n^{1/2}b_n$. Then, let $K>0$ and let $w_1, \ldots, w_{\lfloor K\log n\rfloor}\in V_n$ be vertices that have not been met in the exploration from $x$. Proceed to their $w$-exploration as described in the first part of the construction in Section~\ref{subsec:phase3}. By Remark~\ref{rem:explorationsize}, $o(\sqrt{n})$ vertices are seen during the exploration from $x$ and the $w$-explorations. By (\ref{eqn:binomdominationdiff}) with $k=1$, $m_0,m_1,m_E,m=o(\sqrt{n})$, with $\dP_{ann}$-probability $1-o(1)$, no $w$-exploration intersects the exploration from $x$, and no $w_i$ is spoiled or back-spoiled.
\\
As in (\ref{eqn:zexplorationfirst}), we get that for each $w_i$, its $w$-exploration has probability at least $\eta(h)/2$ to be $w$-successful. Hence with $\dP_{ann}$-probability at least $1-(1-\eta(h)/2)^{\lfloor K\log n\rfloor}=1-o(1)$, there exists $i_0$ such that the $w$-exploration from $w_{i_0}$ is successful. Denote $\partial T_{w_{i_0}}$ the boundary of its exploration tree. Write 
\begin{align*}
\cS_{\text{core}}(x):=&\{\text{the exploration from $x$ is successful}\}\cap
\\
&\{\exists i_0~\geq ~1, \,\text{ the $w$-exploration from $w_{i_0}$ is successful}\}.
\end{align*} 
We have just shown that
\begin{equation}\label{eqn:coresuccess}
\liminf_{n\rightarrow +\infty}\dP_{ann}(\cS_{\text{core}}(x))\geq \liminf_{n\rightarrow +\infty}\dP_{ann}(\text{the exploration from $x$ is successful})\geq \ci.
\end{equation}

\noindent
Next, we grow joining balls from $\partial T_{x_1}$ to $\partial T_{w_{i_0}}$, and then from $\partial T_{x_2}$ to $\partial T_{w_{i_0}}$. We proceed as in the second part of the construction in Section~\ref{subsec:phase3} with $m=2$ and $s_0=1$. Similarly to (\ref{eqn:thirdphasejoiningballs}), we get that 
\begin{center}
$\dP_{ann}(\cS_{\text{core}}(x)\cap (\cS_{\text{core,joining,1}}\cup \cS_{\text{core,joining,2}} ))=o(1)$,
\end{center}
with $\cS_{\text{core,joining,i}}:=\{\text{there are less than $\log^{\gamma-3\kappa-18}n$ joining}\text{ balls from $\partial T_{x_i}$ to $\partial T_{w_{i_0}} $} \}$ for $i=1,2$.
\\
Finally, we reveal $\psimn$ on the exploration from $x$, on the $w$-exploration from $w_{i_0}$, on the joining balls from $\partial T_{x_1}$ to $\partial T_{w_{i_0}}$, and on the joining balls from $\partial T_{x_2}$ to $\partial T_{w_{i_0}}$, in that order. Denote $\cS_{\text{core,connect}}$ the event that there exists a joining ball $B_1$ from $\partial T_{x_1}$ to $\partial T_{w_{i_0}}$ and another $B_2$ from $\partial T_{x_2}$ to $\partial T_{w_{i_0}}$ such that $\min_{y\in T_{x_1}\cup T_{x_2}\cup  T_{w_{i_0}}\cup B_1\cup B_2}\psimn(y)\geq h$. As in the proof of Proposition~\ref{prop:exploendjoin}, we get that 
$\dP_{ann}(\cS_{\text{core}}(x)\cap \cS_{\text{core,connect}}^c)=o(1)$. Hence
\begin{center}
$\liminf_{n\rightarrow +\infty}\dP_{ann}(\text{$x,x_1,x_2,z$ are in a cycle of $\cC_x^{\cM_n}$ and $\vert \cC_x^{\cM_n}\vert \geq n^{1/3}$})$\\$\geq \liminf_{n\rightarrow +\infty}\dP_{ann}(\cS_{\text{core}}(x))$. 
\end{center}
If $\vert \cC_2^{(n)}\vert <n^{1/3}$, this cycle is in $\cC_1^{(n)}$, and thus in $\bC^{(n)}$ so that by (\ref{eqn:secondcompo}), for large enough $n$ and any $x\in V_n$, one has by (\ref{eqn:coresuccess}):
\begin{center}
$\liminf_{n\rightarrow +\infty}\dP_{ann}(x\in \bC^{(n)})\geq \liminf_{n\rightarrow +\infty}\dP_{ann}(\cS_{\text{core}}(x))\geq \ci$ . 
\end{center}
\noindent
\textbf{Third step.} For $x\in V_n$, turn the exploration into a lower exploration,  replacing $h+\log^{-1}n$ by $h-\log^{-1}n$ (as in Section \ref{subsec:explo1vertexaborted}). Say that the lower exploration from $x$ is aborted if for some $k\leq \log\log n$, at most one child of $x$ has a non-empty $(k-1)$-offspring. Let $\cA_{\text{core}}(x):=\{\text{the lower exploration from $x$ is aborted}\}$. We get as in the proof of (\ref{eqn:coreexplosuccess}):
\begin{equation}\label{eqn:coreaborted}
\dP_{ann}(\cA_{\text{core}}(x)) \underset{n\rightarrow +\infty}{\longrightarrow} 1-\ci. 
\end{equation}

\noindent
Moreover, revealing $\psimn$ on $T_x$, we can apply Proposition~\ref{prop:couplinggffsexplo} as below (\ref{eqn:growthsuccessful}) to get that

\begin{equation}\label{eqn:coreabortedpsimn}
\dP_{ann}(\cA_{\text{core}}(x)\cap \{\cC_{x}^{\cM_n,h}\cap B_{\cM_n}(x,\lfloor \log \log n\rfloor)\subseteq T_x\})\underset{n\rightarrow +\infty}{\longrightarrow} 1-\ci. 
\end{equation}
\noindent 
For each neighbour $y$ of $x$, denote $\cC_y$ its connected component in $\cC_x^{\cM_n}\setminus \{x\}$. If the exploration~is aborted and $\cC_{x}^{\cM_n,h}\cap B_{\cM_n}(x,\lfloor \log \log n\rfloor)\hspace{-0.8mm} \subseteq\hspace{-0.8mm}  T_x$, then $x$ has at most one neighbour $y$ such that $\cC_y\cup \{x\}$ is not a tree and $x\not \in \bC^{(n)}$. Thus 
$\liminf_{n\rightarrow +\infty}\dP_{ann}(x\not \in \bC^{(n)})\hspace{-1mm}\geq\hspace{-1mm} 1-\ci$ 
and (\ref{eqn:coreexpectation})~follows.
\end{proof}

\begin{proof}[Proof of (\ref{eqn:corevariance})]
By (\ref{eqn:coreexpectation}), for $x,y \in V_n$,
\begin{center}
$\dP_{ann}(x\in \bC^{(n)}) \dP_{ann}(y\in \bC^{(n)}) \underset{n\rightarrow +\infty}{\longrightarrow}\ci^2$. 
\end{center}
It remains to show that $\dP_{ann}(x,y\in \bC^{(n)})\underset{n\rightarrow +\infty}{\longrightarrow}\ci^2$.
\\
Perform the exploration from $x$ as in the beginning of the proof of (\ref{eqn:coreexpectation}), then do the same from $y$ (and stop the latter if it reaches a vertex of the exploration from $x$). Since $o(\sqrt{n})$ vertices are revealed during these explorations (see Remark~\ref{rem:explorationsize}), then by (\ref{eqn:binomdominationdiff}), the probability that the exploration from $y$ meets that of $x$ is $o(1)$. Thus by (\ref{eqn:coreexplosuccess}),
\begin{center}
$\dP_{ann}(\text{the explorations from $x$ and $y$ are both successful})\underset{n\rightarrow +\infty}{\longrightarrow}\ci^2.$
\end{center}
\noindent
Then, let $x_1,x_2$ (resp. $y_1,y_2$) be the children of $x$ (resp. $y$) whose exploration is successful. We complete the exploration in a fashion similar to that above (\ref{eqn:coresuccess}). Let $K>0$ and let $w_1, \ldots, w_{\lfloor K\log n\rfloor}\in V_n$ be vertices that have not been met in the explorations from $x$ and $y$, and proceed to their $w$-exploration. If there exists $i_0\geq 1$ such that the $w$-exploration from $w_{i_0}$ is successful, build joining balls from $\partial T_{x_1},\partial T_{x_2},\partial T_{y_1}$ and $\partial T_{y_2}$ to $T_{w_{i_0}}$. Finally, reveal $\psimn$ on $T_x$, on $T_y$, on $T_{w_{i_0}}$ and on the joining balls from $\partial T_{x_1},\partial T_{x_2},\partial T_{y_1}$ and $\partial T_{y_2}$, in that order. As in the proof of (\ref{eqn:coresuccess}) and below, we get that 
\begin{equation}\label{eqn:corevariancelowbound}
\liminf_{n\rightarrow +\infty}\dP_{ann}(x,y\in \bC^{(n)})\geq\ci^2.
\end{equation}

\noindent
Conversely, if we perform the lower explorations from $x$ and $y$ as defined in the end of the proof of (\ref{eqn:coreexpectation}), we easily get that
\begin{center}
$\dP_{ann}(\exists z\in \{x,y\},\text{ the lower exploration from $z$ is aborted})\underset{n\rightarrow +\infty}{\longrightarrow}1-\ci^2$.
\end{center} 

\noindent
Then, we reveal $\psimn$ on $T_z$. Following the reasoning below (\ref{eqn:growthsuccessful}), we get that $\dP_{ann}$-w.h.p., $B_{\cC_z^{\cM_n,h}}(z,\lfloor\log\log n\rfloor)\subseteq T_z$, and thus
\begin{center}
$\dP_{ann}(\{\text{the lower exploration from $z$ is aborted}\}\cap \{z\in \bC^{(n)}\})=o(1)$.
\end{center}
This yields 
\begin{center}
$\liminf_{n\rightarrow +\infty}\dP_{ann}(\exists z\in \{x,y\},z\not\in \bC^{(n)})\geq 1-\ci^2$.
\end{center}
\noindent
Together with (\ref{eqn:corevariancelowbound}), this concludes the proof.
\end{proof}


\noindent
This reasoning can be readily adapted to prove (\ref{eqn:kernel}), with a modification of the exploration (requiring that at least three children of $x$ have a successful exploration).

\subsection{The typical distance: proof of (\ref{eqn:typicaldistance})}
The proof of (\ref{eqn:typicaldistance}) goes as that of (\ref{eqn:mainthm}), with a slight modification of the explorations of Section~\ref{sec:exploration}. Those explorations were indeed stopped after at most $\log_{\lambda_h}n$ steps. But since around $\sqrt{n}$ vertices were explored, and since the growth rate of $\Ch$ is close to $\lambda_h$, we can expect that a successful exploration lasts in fact $(1/2+o(1))\log_{\lambda_h}n$ steps. Then, connecting two such explorations as in Proposition~\ref{prop:exploendjoin} (adding an additional distance $\Theta(\log\log n)=o(\log n)$) yields the typical distance (reciprocally, explorations lasting less steps will be too small to be connected).

\begin{proof}[Proof of (\ref{eqn:typicaldistance})]
Fix $\varepsilon \in (0,1)$. 
\\
\textbf{Upper bound.} In the exploration of Section~\ref{subsec:explo1vertexsuccess}, replace \ref{C6} by the following condition: stop the exploration if $k\geq (1/2+\varepsilon/3)\log_{\lambda_h}n$. By Proposition~\ref{prop:Chlargedevgrowthrate},
\begin{center}
$\lim_{n\rightarrow +\infty}\dP^{\Td}\left(\vert\cZ_{\lfloor   (1/2+\varepsilon/3)\log_{\lambda_h}n \rfloor}^{h+\log^{-1}n}\vert > n^{1/2}b_n \,\bigg| \,\vert \Chlnplus\vert =+\infty\right)=1$.
\end{center}
Thus, Propositions~\ref{prop:explo1vertex}, \ref{prop:exploendjoin} and Lemma~\ref{lem:connexionslowerbound} remain unchanged. And if $x,y$ are connected in $\Enh$ via successful explorations from $x$ and $y$ and joining balls from $\partial T_x$ to $\partial T_y$, then
\begin{center}
$d_{\Enh}(x,y)\leq 2(1/2+\varepsilon/3)\log_{\lambda_h}n +a'_n\leq (1+\varepsilon)\log_{\lambda_h}n$  
\end{center}
 for $n$ large enough. Then, Lemma~\ref{lem:connexionslowerbound} implies that
\begin{center}
$\dP_{ann}(\cE_{1,n})\underset{n\rightarrow +\infty}{\longrightarrow} 1$,
\end{center}
\noindent
where $\cE_{1,n}:= \{ \vert \{(x,y)\in V_n^2, d_{\Enh}(x,y)\leq(1+\varepsilon)\log_{\lambda_h}n\}\vert \geq (\eta(h)^2-\varepsilon)n^2\}$.
\\
We have to check that only $o(n^2)$ of the couples $(x,y)$ described in $\cE_{1,n}$ are not in $\cC_1^{(n)}$, and that $\vert \cC_1^{(n)}\vert/n$ is indeed close to $\eta(h)$. Note that by (\ref{eqn:mainthm}) and (\ref{eqn:secondcompo}),
\begin{center}
$\dP_{ann}(\cE_{1,n}\cap \cE_{2,n}\cap \cE_{3,n})\underset{n\rightarrow +\infty}{\longrightarrow} 1$,
\end{center}
\begin{center}
where $\cE_{2,n}:=\{ \forall i\geq 2, \,\vert \cC_i^{(n)}\vert \leq \cz \log n\}$ and $\cE_{3,n}:=\{ (\eta(h)-\varepsilon)n\leq\vert \cC_1^{(n)} \vert \leq (\eta(h)+\varepsilon)n \}$.
\end{center}
On $\cE_{2,n}$, we have $\vert \{(x,y)\in V_n^2\setminus (\cC_1^{(n)})^2,\, d_{\Enh}(x,y)\leq(1+\varepsilon)\log_{\lambda_h}n\}\vert \leq n^{3/2}$, so that on $\cE_{1,n}\cap \cE_{2,n}$, 
\begin{center}
$\vert \{(x,y)\in (\cC_1^{(n)})^2, d_{\cC_1^{(n)}}(x,y)\leq(1+\varepsilon)\log_{\lambda_h}n\}\vert \geq (\eta(h)^2-2\varepsilon)n^2$. 
\end{center}
Thus, on $\cE_{1,n}\cap \cE_{2,n}\cap \cE_{3,n}$:
\begin{align*}
\pi_{2,n}\left(\{(x,y)\in (\cC_1^{(n)})^2, d_{\cC_1^{(n)}}(x,y)\leq(1+\varepsilon)\log_{\lambda_h}n\}\right)&\geq \frac{\eta(h)^2-2\varepsilon}{(\eta(h)+\varepsilon)^2}\geq 1- \frac{2\varepsilon +2\eta(h)\varepsilon+\varepsilon^2}{(\eta(h)+\varepsilon)^2}
\\
&\geq 1 - \frac{(3+2\eta(h))}{\eta(h)^2}\varepsilon.
\end{align*}

\noindent
\textbf{Lower bound.} It remains to show that the typical distance in $\cC_1^{(n)}$ is at least $(1-\varepsilon)\log_{\lambda_h}n$. Modify the lower exploration of Section~\ref{subsec:explo1vertexaborted} by saying that it is aborted if
\begin{itemize}
\item \ref{C1} did not happen, and
\item it is stopped at some step $k\leq (1/2-\varepsilon/2)\log_{\lambda
_h}n$, or less than $n^{1/2-\varepsilon/10}$ vertices and half-edges have been seen at step $\lfloor(1/2-\varepsilon/2)\log_{\lambda_h}n \rfloor$.

\end{itemize}

\noindent 
By Lemma~\ref{lem:typicaldistaborted} below, 
\begin{equation}\label{eqn:exploaborteddiameter}
\dP_{ann}(\text{the lower exploration from $x$ is aborted})=1-o(1).
\end{equation}
\noindent
For $x,y\in V_n$, perform the lower exploration from $x$, then that from $y$, and stop it if it meets a vertex of the exploration from $x$. This happens with $\dP_{ann}$-probability $o(1)$ by (\ref{eqn:binomdominationdiff}) with $k=1$, $m_0,m_1,m,m_E=o(\sqrt{n})$ (recall Remark~\ref{rem:explorationsize}).Hence by (\ref{eqn:exploaborteddiameter}):
\begin{center}
$\hspace{-1mm}\hspace{-1mm}\hspace{-1mm}\hspace{-1mm}\dP_{ann}(\cE_{\text{ab}}(x,y))\hspace{-1mm}=\hspace{-1mm}1-o(1)$, with $\cE_{\text{ab}}(x,y)\hspace{-1mm}:=\hspace{-1mm}\{\text{the lower explorations from $x$ and $y$ are aborted}\}$, 
\end{center}
Then, reveal $\psimn$ on the exploration trees $T_x$ and $T_y$. Applying Proposition~\ref{prop:couplinggffsexplo} as below (\ref{eqn:growthsuccessful}), we get that: 
\begin{center}
$\dP_{ann}(\cE_{\text{ab}}(x,y)\cap\cE_x\cap\cE_y)=1-o(1)$, 
\end{center}
with $\cE_x\hspace{-1mm}:=\hspace{-1mm}\{B_{\cC_x^{\cM_n,h}}(x,\lfloor (1/2-\varepsilon/2)\log_{\lambda_h}n\rfloor)\hspace{-1mm}\subseteq T_x \}$ and $\cE_y\hspace{-1mm}:=\hspace{-1mm} \{B_{\cC_y^{\cM_n,h}}(y,\lfloor (1/2-\varepsilon/2)\log_{\lambda_h}n\rfloor)\hspace{-1mm}\subseteq T_y \}$. 
\\
\\
On $\cE_{\text{ab}}(x,y)\cap \cE_x\cap \cE_y$, if $x,y\in \cC_1^{(n)}$, then $d_{\cC_1^{(n)}}(x,y)\geq (1-\varepsilon)\log_{\lambda_h}n$. 
Therefore, for every $x,y\in V_n$, $\dP_{ann}(\cE_{x,y})=o(1)$, with 
\begin{center}
$\cE_{x,y}:=\{x,y\in \cC_1^{(n)}, d_{\cC_1^{(n)}}(x,y)<(1-\varepsilon)\log_{\lambda_h}n\}$. 
\end{center}
Similarly, for all distinct $x,y,z,t\in V_n$, we get that $\dP_{ann}(\cE_{x,y}\cap \cE_{z,t})=o(1)$, so that
\begin{center}
$\text{Cov}_{ann}(\mathbf{1}_{\cE_{x,y}},\mathbf{1}_{\cE_{z,t}})=o(1)$.
\end{center}
Thus by Bienaymé-Chebyshev's inequality, 
\begin{center}
$\dP_{ann}(\cE_{3,n}\cap\{\vert \{(x,y)\in V_n^2, \, d_{\cC_1^{(n)}}(x,y)\leq (1-\varepsilon)\log_{\lambda_h}n\} \vert \geq \varepsilon n^2\})\underset{n\rightarrow +\infty}{\longrightarrow} 1.$
\end{center}
\noindent
For $\varepsilon>0$ small enough and $n$ large enough, on 
\begin{center}
$\cE_{3,n}\cap\{\vert \{(x,y)\in V_n^2, \, d_{\cC_1^{(n)}}(x,y)\leq (1-\varepsilon)\log_{\lambda_h}n\} \vert \geq \varepsilon n^2\} $,
\end{center}
$\pi_{2,n}\left(\{(x,y)\in (\cC_1^{(n)})^2, d_{\cC_1^{(n)}}(x,y)\leq(1-\varepsilon)\log_{\lambda_h}n\}\right)\leq \frac{2\varepsilon}{\eta(h)^2}$. This concludes the proof of (\ref{eqn:typicaldistance}).
\end{proof}

It remains to establish the lemma below. 
\begin{lemma}\label{lem:typicaldistaborted}
We have 
\[
\lim_{n\rightarrow \infty}\dP_{ann}(\text{the lower exploration from $x$ is aborted})=1.
\]
\end{lemma}
\begin{proof}
Note first that  $\dP_{ann}(\text{\ref{C1} happens})=o(1)$ by (\ref{eqn:binomdominationdiff}) with $k=1$, $m_0=1$, $m_E=0$ and $m=o(\sqrt{n})$ by Remark~\ref{rem:explorationsize}. Therefore,  it is enough to prove that
\begin{center}
$\dP^{\Td}(\vert B_{\Chlnminus}(\circ,\lfloor (1/2-\varepsilon/2)\log_{\lambda_h}n\rfloor+a_n)\vert < n^{1/2-\varepsilon/8})\rightarrow 1,$
\end{center}
since if this event happens, less than $n^{1/2-\varepsilon/10}$ vertices and half-edges have been seen at step $\lfloor(1/2-\varepsilon/2)\log_{\lambda_h}n \rfloor$.
By (\ref{eqn:andef}), it suffices to show that 
\begin{equation}\label{eqn:exploaborteddiameterTd}
\dP^{\Td}(\vert B_{\Chlnminus}(\circ,\lfloor (1/2-\varepsilon/2)\log_{\lambda_h}n\rfloor)\vert < n^{1/2-\varepsilon/7})\rightarrow 1,
\end{equation}
To do so, we first prove that for $n$ large enough and every $\log\log n\leq k\leq (1/2-\varepsilon/2)\log_{\lambda
_h}n$,
\begin{equation}\label{eqn:exploaborteddiameterTdZh}
\dP^{\Td}\left(\vert \cZ_k^{h-\log^{-1}n}\vert \geq n^{1/2-\varepsilon/6}\right) \leq e^{-Ck}+ n^{-3}.
\end{equation}
Let $\delta>0$ such that $\lambda_h\leq \lambda_{h-\delta}\leq \lambda_h+\varepsilon/10$ ( $h'\mapsto \lambda_{h'}$ being continuous, Proposition~\ref{prop:thm43adapted}). Since $\cZ_k^{h-\log^{-1}n} \subseteq  \cZ_k^{h-\delta}$ for $n$ large enough and all $k$, a direct computation shows that 
\begin{align*}
\dP^{\Td}_a\left(\vert \cZ_k^{h-\log^{-1}n}\vert \geq n^{1/2-\varepsilon/5}\chi_{h-\delta}(a)\right) &\leq \dP^{\Td}_a\left(\log \vert \cZ_k^{h-\delta}\vert \geq \lambda_{h-\delta}(1+\varepsilon/5)k+\log\chi_{h-\delta}(a)\right).
\end{align*}
Then by Proposition~\ref{prop:Chlargedevgrowthrate}, there exists $C>0$ (depending on $\varepsilon$) such that for $n$ large enough, for every $\log\log n\leq k\leq (1/2-\varepsilon/2)\log_{\lambda
_h}n$ and $a\geq h$, 
\begin{align*}
\dP^{\Td}_a\left(\vert \cZ_k^{h-\log^{-1}n}\vert \geq n^{1/2-\varepsilon/5}\chi_{h-\delta}(a)\right) &
\leq e^{-Ck}.
\end{align*}
\noindent
By Proposition 2.1 of \cite{ACregulgraphs}, there exists $c>0$ such that for all $h'\leq h_{\star}$ and $a\geq d-1$, one has $\chi_{h'}(a)\leq ca^{1-\log_{d-1}\lambda_{h'}}\leq ca$. Since $\chi_{h'}$ is continuous on $[h,+\infty)$ (Lemma~\ref{lem:SZ16}), we have for $n$ large enough $\max_{h\leq a\leq \log^2n}\chi_{h-\delta}(a)<n^{\varepsilon/30}$, so that 
\begin{center}
$\dP^{\Td}\left(\vert \cZ_k^{h-\log^{-1}n}\vert \geq n^{1/2-\varepsilon/6}\right)  \leq e^{-Ck}+\dP^{\Td}(\phid(\circ)\geq \log^2 n)$
\end{center}
Using the exponential Markov inequality as in Lemma~\ref{lem:maxgff}, we get $\dP^{\Td}(\phid(\circ)\geq \log^2 n)\leq n^{-3}$. This yields (\ref{eqn:exploaborteddiameterTdZh}). Then, for $n$ large enough, this implies
\begin{align*}
 \dP^{\Td}(\vert B_{\Ch}(\circ,\lfloor (1/2-\varepsilon/2)&\log_{\lambda_h}n\rfloor)\vert \geq n^{1/2-\varepsilon/7})
\\
&\leq \dP^{\Td}(\exists k\in [\log\log n, (1/2-\varepsilon/2)\log_{\lambda_h}n], \vert \cZ_k^{h-\log^{-1}n}\vert \geq n^{1/2-\varepsilon/6})
\\
&\leq \sum_{k=\lfloor\log\log n \rfloor}^{\lfloor (1/2-\varepsilon/2)\log_{\lambda_h}n\rfloor}(e^{-Ck}+n^{-3})
\\
&\leq 1/\log\log n.
\end{align*}
(\ref{eqn:exploaborteddiameterTd}) and the conclusion follow.
\end{proof}

\subsection{The diameter: proof of (\ref{eqn:diameter})}
Recall that $D_1^{(n)}$ is the diameter of $\cC_1^{(n)}$. In Section~\ref{sec:uniqueness}, we have in fact proven that there exists a constant $\cxii>0$ such that for $\cE_n:=\{ \forall x\in \cC_1^{(n)},\,\vert B_{\cC_1^{(n)}}(x,\lfloor \cxii\log n\rfloor )\vert \geq \cxi n\}$, we have
\begin{center}
$\dP_{ann}(\cE_n)\underset{n\rightarrow +\infty}{\longrightarrow}1.$
\end{center}
Namely, one can take $\cxii =\cz+3{\log^{-1}\lambda_h}$. 
\\
Hence it is enough to show that on $\cE_n$, $D_1^{(n)}\leq 6K_{12}^{-1}K_{13}\log n$, which will imply (\ref{eqn:diameter}). We do this by a short deterministic argument.
\\
Let $x_1\in \cC_1^{(n)}$. If $\partial B_{\cC_1^{(n)}}(x_1,2\lfloor \cxii\log n\rfloor +1)=\emptyset$, then 
$D_1^{(n)}\leq 4\cxii\log n +2$. 
\\
Else, let $x_2\in \partial B_{\cC_1^{(n)}}(x_1,2\lfloor \cxii\log n\rfloor +1)$. For $i=1,2$, we have
\begin{center}
$\vert B_{\cC_1^{(n)}}(x_i, \lfloor \cxii\log n\rfloor) \vert \geq \cxi n$ and $ B_{\cC_1^{(n)}}(x_i, \lfloor \cxii\log n\rfloor) \subseteq B_{\cC_1^{(n)}}(x_1,4\lfloor \cxii\log n\rfloor) $. 
\end{center}
Moreover, $ B_{\cC_1^{(n)}}(x_1,\lfloor \cxii\log n\rfloor)\cap  B_{\cC_1^{(n)}}(x_2, \lfloor \cxii\log n\rfloor)=\emptyset$. Thus, we have 
\begin{center}
$\vert B_{\cC_1^{(n)}}(x_1,4\lfloor \cxii\log n\rfloor )\vert \geq 2\cxi n$. 
\end{center}
For $i\geq 2$, if $\partial B_{\cC_1^{(n)}}(x_1,(3i-4)\lfloor \cxii\log n\rfloor +1)=\emptyset$, then $D_1^{(n)}\leq 2(3i-4)\cxii\log n +2$. Else, let $x_{i}\in \partial B_{\cC_1^{(n)}}(x_1,(3i-4)\lfloor \cxii\log n\rfloor +1)$. As for $i=2$, we get that
$$
\vert B_{\cC_1^{(n)}}(x_1,(3i-2)\lfloor \cxii\log n\rfloor )\vert -  \vert B_{\cC_1^{(n)}}(x_1,(3i-5)\lfloor \cxii\log n\rfloor)\vert\geq \vert B_{\cC_1^{(n)}}(x_i, \cxii\log n)\vert\geq \cxi n.
$$
Since $\vert V_n\vert=n$, $\partial B_{\cC_1^{(n)}}(x_1,(3i_0-4)\lfloor \cxii \log n\rfloor +1)=\emptyset$ for some $i_0\leq \cxi^{-1}$, and
$$
D_1^{(n)}\leq 6\cxi^{-1}\cxii\log n.
$$
This shows (\ref{eqn:diameter}).

\begin{appendix}
\section{Appendix} \label{appn}
\subsection{Proof of Proposition~\ref{prop:goodgraph}}\label{subsec:AppendixSection2}

\begin{proof}[Proof of Proposition~\ref{prop:goodgraph}]
By Theorem 1 of~\cite{BollobasExpansion}, there exists $c=(d)>0$ such that $\cM_n$ has an expansion ratio at least $c(d)$ w.h.p., and thus satisfies~\ref{expander1} with $\ciii = (c(d)/d)^2/2 $ by the Cheeger bound (see for instance Theorem 2.4 in~\cite{Hoory}, the expansion ratio being defined in Definition 2.1 of~\cite{Hoory}). 
\\
As for \ref{expander2}, fix $\ciii >0$. For all $x\in V_n$, one obtains $B_{\cM_n}(x,\lfloor \ciii\log n\rfloor)$ by proceeding to at most $d(d-1)^{\lfloor\ciii\log n\rfloor}$ pairings of half-edges. If $\ciii$ is small enough, $d(d-1)^{\lfloor\ciii\log n\rfloor}<n^{1/5}-1$ so that by (\ref{eqn:binomdominationdiff}) with $m_0=1$, $m_E=0$, $m=d(d-1)^{\lfloor\ciii\log n\rfloor}$ and $k=2$, for $n$ large enough:
\begin{center}
$\dP(\tx(B_{\cM_n}(x,\lfloor \ciii\log n\rfloor))\geq 2)\leq C(2)\left(\frac{n^{2/5}}{n} \right)^2 \leq n^{-11/10}.$
\end{center}
By a union bound on $x\in V_n$, w.h.p. $\cM_n$ is such that for all $x\in V_n$, $\tx(B_{\cM_n}(x,\lfloor \ciii \log n\rfloor))\leq 1$. 
\\
In the remainder of the proof, we fix a realization of $\cM_n$ such that \ref{expander1} and \ref{expander2} hold. In particular, $\cM_n$ is connected and $\gmn$ is well-defined.
\\
Let us establish (\ref{eqn:greenfunctionGnapprox1}). Let $x\in V_n$ be such that $\tx (B_{\cM_n}(x,K_4\lfloor\log\log n\rfloor))=0$. Note that $U:=B_{\cM_n}(x,\lfloor \cvi\log\log n\rfloor)$ and $W:=B_{\Td}(\circ, \lfloor \cvi\log\log n\rfloor)$ are isomorphic. Then $\gmn^U(x,x)=\gtd^W(\circ,\circ)$, where we let 
\begin{center}
$\gmn^A(y,z):=\bE_y^{\cM_n}[\sum_{k=0}^{T_A}\mathbf{1}_{\{X_k =z\}}]$ for every $y,z\in V_n$ and $A\subsetneq V_n$. 
\end{center}
Recall that $T_A$ is the exit time of $A$ by the SRW $(X_k)_{k\geq 0}$. Similarly for every $B\subsetneq \Td$ and $y,z\in \Td$, we define
\begin{equation}\label{eqn:killedgreenfnTd}
\gtd^B(y,z):=\bE_y^{\Td}\left[\sum_{k=0}^{T_B}\mathbf{1}_{\{X_k =z\}}\right].
\end{equation}
 On one hand, by the strong Markov property applied to the exit time $T_W$, 
\begin{center}
$\gtd^W(\circ,\circ)=\gtd(\circ,\circ)-\bE_{\circ}^{\Td}[\gtd(\circ,X_{T_W})]=\frac{d-1}{d-2}-\bE_{\circ}^{\Td}[\gtd(\circ,X_{T_W})]$.
\end{center}
On the other hand, by Lemma 1.4 of~\cite{Abacherli}, for all $y,z\in V_n$ and $A\subsetneq V_n$:
\begin{equation}\label{eqn:lem14Aba}
\gmn^A(y,z)=\gmn(y,z) -\bE_y^{\cM_n}[\gmn(z,X_{T_A})] +\frac{\bE_y^{\cM_n}[T_A]}{n},
\end{equation} 
so that $\gmn^U(x,x)=\gmn(x,x) -\bE_x^{\cM_n}[\gmn(x,X_{T_U})] +\frac{\bE_{x}^{\cM_n}[T_U]}{n}$. Therefore, 
$$
\left\vert\gmn(x,x)-\frac{d-1}{d-2}\right\vert \leq \vert \bE_{\circ}^{\Td}[\gtd(\circ,X_{T_W})]\vert + \vert \bE_x^{\cM_n}[\gmn(x,X_{T_U})]\vert +\frac{\bE_x^{\cM_n}[T_U]}{n}.
$$

\noindent
By (\ref{eqn:greenonTd}) and (\ref{eqn:greenfunctionGn}), if $\cvi$ is large enough, then for large enough $n$:
\begin{center}
$ \vert \bE_{\circ}^{\Td}[\gtd(\circ,X_{T_W})]\vert + \vert \bE_x^{\cM_n}[\gmn(x,X_{T_U})]\vert\leq \log^{-7}n$. 
\end{center}
Note that $T_U$ is stochastically dominated by the hitting time $H$ of $\lfloor \cvi\log\log n\rfloor$ by a SRW $(Z_k)_{k\geq 0}$ on $\dZ$ starting at $0$, whose transition probabilities from any vertex are $\frac{d-1}{d}$ towards the right and $1/d$ towards the left. By Markov's exponential inequality, there exists a constant $c>0$ such that for $n$ large enough and every $k>n^{1/10}$, 
\begin{center}
$\dP(H\geq k)\leq \dP(Z_k\leq \lfloor \cvi\log\log n\rfloor)\leq \dP(Z_k\leq (\frac{d-2}{d}-1/100)k)\leq e^{-ck}$. 
\end{center}
Hence for $n$ large enough, $\bE_x^{\cM_n}[T_U]\leq\dE[H]\leq n^{1/10} +\sum_{k\geq n^{1/10}}ke^{-ck}\leq n^{1/2}$. Thus,
$$
\left\vert\gmn(x,x)-\frac{d-1}{d-2}\right\vert \leq \log^{-7}n+n^{-1/2}\leq \log^{-6}n
$$
for large enough $n$, and this yields (\ref{eqn:greenfunctionGnapprox1}). One proves (\ref{eqn:greenfunctionGnapprox2}) by the same reasoning.
\end{proof}

\subsection{Proof of Propositions~\ref{prop:thm43adapted},~\ref{prop:expomomentsCh} and~\ref{prop:Chlargedevgrowthrate}}\label{subsec:AppendixSection3}
\noindent
We start with the proof of Proposition~\ref{prop:thm43adapted}. We will need the following Lemma (whose proof is immediate from Propositions 3.1 and 3.3 of \cite{SZ2016} and Proposition 2.1 (ii) of \cite{ACregultrees}), from which $\lambda_h$ originates. 

\begin{lemma}\label{lem:SZ16}
Fix $h <h_{\star}$. There exist $\lambda_h>1$ and a function $\chi_h$ that is continuous with a positive minimum $\chi_{h,\min}$ on $[h,+\infty)$, that vanishes on $(-\infty,h)$ and such that
\[
M_k^h:=\lambda_h^{-k}\sum_{x\in  \cZ_k^{h,+}}\chi_h(\phid(x))
\]
is a martingale w.r.t. the filtration $\cF_k:=\sigma \left(\phid(x), x\in B_{\Td^+}(\circ,k)\right)$, $k\geq 0$, and has an a.s. limit $M_{\infty}^h$. 
\end{lemma}

\begin{proof}[Proof of Proposition~\ref{prop:thm43adapted}]
We first establish that $\lim_{k\rightarrow +\infty}\dP^{\Td}(\vert\cZ_k^h\vert>\lambda_{h}^k/k^2)= \eta(h) $. Clearly, 
\begin{center}
$\limsup_{k\rightarrow +\infty}\dP^{\Td}(\vert\cZ_k^h\vert>\lambda_{h}^k/k^2)\leq \dP^{\Td}(\vert\Ch\vert =+\infty)=\eta(h).$
\end{center}
\noindent
Conversely, denote $\mathcal{E}^+=\{\vert\Chplus\vert =+\infty\}$ and $\mathcal{E}^+_k=\{ \vert\cZ_k^{h,+}\vert\geq\lambda_{h}^k/k^2\}$.
By Theorem 4.3 of \cite{ACregultrees}, 
\begin{equation}\label{eqn:thm43}
\lim_{k\rightarrow +\infty}\dP^{\Td}(\cE^+_k)=\dP^{\Td}(\cE^+)>0.
\end{equation} 
Hence, for any $\varepsilon >0$, for $k$ large enough, 
\begin{equation}\label{eqn:escapeoneend}
\dP^{\Td}(\cE^+_k)\geq \dP^{\Td}(\cE^+) -\varepsilon.
\end{equation}
Let $\Td^-$ be the cone from $\oc$ out of $\circ$, $\cC_{\oc}^{h,-}:=\cC_{\oc}^h\cap\Td^-$, and $\cZ_k^{h,-}:=\Chminus\cap \partial B_{\Td^-}(\oc,k)$ for $k\geq 1$.
Let $\mathcal{E}^-~=~\{\vert \cC_{\oc}^{h,-}\vert = +\infty\}$ and $\mathcal{E}^-_k=\{ \vert\cZ_{k-1}^{h,-}\vert\geq\lambda_{h}^k/k^2\}$. Define $\mathcal{E}:=\{\phid(\circ)\geq h\} \cap\{\text{ $\Chplus$ is finite}\}$ and $\mathcal{E}_k:=\{\phid(\circ)\geq h\}\cap\{\text{ $\cZ_k^{h,+}=\emptyset$}\}$. We have
\[
\dP^{\Td}(\cE^-_k\cap \cE_k)\geq \dP^{\Td}(\cE^- \cap \cE)-\dP^{\Td}(\cE^-\cap \cE\cap (\cE_k)^c) -\dP^{\Td}(\cE^- \cap \cE\cap (\cE^-_k)^c).
\]
Define $M_{\infty}^{h,-}$ on $\Chminus$ as $M_{\infty}^{h}$ on $\Chplus$.
From the proof of Theorem 4.3 in \cite{ACregultrees}, we get that $\dP^{\Td}(\{M_{\infty}^{h,-}>0 \} \cap (\cE_k^-)^c ) \rightarrow 0$. And, by Proposition~4.2 in~\cite{ACregultrees}, $\dP^{\Td}(\cE^-\cap \{M_{\infty}^{h,-}=0\})=0$, so that 
$\dP^{\Td}(\cE^- \cap (\cE_k^-)^c)\rightarrow 0$. Moreover, $(\mathcal{E}_k)_{k\geq 0}$ is an increasing sequence of events and $\mathcal{E}=\cup_{k\geq 1}\cE_k$, so that $ \dP^{\Td}(\cE \cap (\cE_k)^c)\rightarrow 0$. Hence, for $k$ large enough, 
\begin{equation}\label{eqn:escapeotherend}
\dP^{\Td}(\cE^-_k\cap \cE_k) \geq\dP^{\Td}(\cE^-\cap \cE) -2\varepsilon.
\end{equation}
\noindent
Note that $\{\vert\Ch\vert =+\infty\}=\cE^+\sqcup (\cE^-\cap \cE)$, so that $\dP(\cE^+)+ \dP(\cE^-\cap\cE)=\eta(h).$ And for all $k\geq 2$, $ \cE^+_k\sqcup (\cE^-_k \cap \cE_k)\subseteq \{\vert\cZ_k^h\vert>\lambda_{h}^k/k^2\}$, therefore, if $k$ is large enough, by (\ref{eqn:escapeoneend}) and (\ref{eqn:escapeotherend}), one has
\[
\dP^{\Td}(\vert\cZ_k^h \vert>\lambda_{h}^k/k^2)\geq \eta(h)-3\varepsilon.
\]
Since $\varepsilon>0$ was arbitrary, $\lim_{k\rightarrow +\infty}\dP^{\Td}(\vert\cZ_k^h\vert>\lambda_{h}^k/k^2)= \eta(h) $. 

\noindent
Now, we show that $\dP^{\Td}(\vert\cZ_k^h\vert<k\lambda_{h}^k )\rightarrow 1$.  For all $k\geq 1$, by definition of $M_k^h$,
\begin{center}
$M_k^h \geq \chi_{h,\min}\lambda_{h}^{-k} \lv \cZ_k^{h,+}\rv$.
\end{center}
From the proof of Proposition 3.3 in \cite{SZ2016}, $M_{\infty}^h$ is a.s. finite. Therefore, $k^{-1}\lambda_{h}^{-k} \lv \cZ_k^{h,+}\rv\rightarrow 0$ a.s., and
\[
\dP^{\Td}(\vert \cZ_k^{h,+} \vert \geq k\lambda_{h}^k/2 )\rightarrow 0.
\]
In the same way, $\dP^{\Td}(\vert \cZ_{k-1}^{h,-} \vert\geq k\lambda_{h}^k/2 )\rightarrow 0$. Since $\cZ_k^h \subseteq \cZ_k^{h,+}\cup \cZ_{k-1}^{h,-}$, we are done.
\\
The last part of the Proposition comes directly from Propositions 3.1 and 3.3 in~\cite{SZ2016}.
\end{proof}

\noindent
A crucial idea to prove Propositions~\ref{prop:expomomentsCh} and~\ref{prop:Chlargedevgrowthrate} is to make a finite scaling, in order to get a branching process that is uniformly supercritical w.r.t. to the value of $\phid(\circ)$. Indeed, the fact $\lambda_h>1$ does not ensure that the expected number of children of $\circ$ in $\Td^+$ (or even in $\Td$) conditionally on $\phid(\circ)=a$ is more than one for every $a\geq h$, in particular if $a$ is small. However, due to the exponential growth of $\Ch$ (and $\Chplus$) of Proposition~\ref{prop:thm43adapted}, it turns out that for $\ell \in \dN$ large enough, even conditionally on $\phid(\circ)=h$, the expected number of vertices in the $\ell$-offspring of $\circ$ is more than one, as stated in the Lemma below. 
\begin{lemma}\label{lem:uniformsupercritical}
There exists $\ell \in \dN$ such that for every $a\geq h$, 
$$\dE^{\Td}_a[\vert \cZ_{\ell}^h\vert ] \geq\dE^{\Td}_a[\vert \cZ_{\ell}^{h,+}\vert ] \geq \dE^{\Td}_h[\vert \cZ_{\ell}^{h,+}\vert ]>1.$$
\end{lemma}
\noindent
We will use it in the proofs of Propositions~\ref{prop:expomomentsCh} and~\ref{prop:Chlargedevgrowthrate}, looking at the branching process whose vertices are those of $\Ch$ at height $0$, $\ell$, $2\ell$, etc.

\begin{proof}[Proof of Lemma~\ref{lem:uniformsupercritical}] Write $\cE_k:=\{\vert \cZ_k^{h,+}\vert \geq \lambda_h^k/k^2\}$. By (\ref{eqn:thm43}), there exists $\varepsilon>0$ small enough such that for every $k\geq \varepsilon^{-1}$, $\dP^{\Td}(\cE_k)\geq \varepsilon$. For $a_1$ large enough, $\dP^{\Td}(\phid(\circ) \geq a_1)<\varepsilon/2$. Note that $\cE_k$ is an increasing event, so that by Lemma~\ref{lem:monotonicityphid}, the map $a'\mapsto \dP^{\Td}_{a'}(\cE_k)$ is non-decreasing. Therefore, for every $a'\geq a_1$ and $k\geq M$, if $\nu\sim \cN(0,\frac{d-1}{d-2})$ denotes the law of $\phid(\circ)$, 
\begin{equation}\label{eqn:uniformsupercritical1}
\dP^{\Td}_{a'}(\cE_k)\geq\dP^{\Td}_{a_1}(\cE_k)\geq \int_{-\infty}^{a_1}\dP^{\Td}_{b}(\cE_k)\nu(db) \geq \dP^{\Td}(\cE_k)-\dP^{\Td}(\phid(\circ) \geq a_1)\geq \varepsilon /2.
\end{equation}
From the construction of Proposition~\ref{prop:recursivegfftrees}, it is straightforward that 
\begin{equation}\label{eqn:uniformsupercritical2}
p:=\dP^{\Td}_h(\text{$\circ$ has one child $z$ in $\Chplus$, and }\phid(z)\geq a_1)>0.
\end{equation}
Hence for $\ell \in \dN$ large enough, 
$$
\dE^{\Td}_h[\vert \cZ_{\ell}^{h,+}\vert ] \geq \frac{p\varepsilon}{2}\times \frac{\lambda_h^{\ell-1}}{(\ell-1)^2} >1.
$$
In addition, for every $M\in \dR$, $\{\vert\{\cZ_{\ell}^{h,+}\vert\geq M\}$ is an increasing event. By Lemma~\ref{lem:monotonicityphid}, for every $a\geq h$, $\dE^{\Td}_a[\vert \cZ_{\ell}^{h,+}\vert ]\geq \dE^{\Td}_h[\vert \cZ_{\ell}^{h,+}\vert ]$. Since $ \cZ_{\ell}^{h,+}\subseteq \cZ_{\ell}^h$ a.s., the conclusion follows.
\end{proof}

\begin{remark}\label{rem:uniformsupercriticsurvival}
By a direct adaptation of this proof, in particular \eqref{eqn:uniformsupercritical1} and \eqref{eqn:uniformsupercritical2}, one has $\dP^{\Td}_a[\vert \Ch\vert=\infty ]\geq \dP^{\Td}_h[\vert \Ch\vert=\infty ]>0$ for all $a\geq h$. Also, the conclusion of the Lemma still holds if for a fixed $\delta\in [0,h_{\star}-h)$,  $\cZ_{\ell}^{h,+}$ now denotes the $\ell$-th generation of the connected component of $\circ$ in $(\{\circ\}\cup E_{\phid}^{\geq h+\delta})\cap \Td^+$: denoting $\rho_{\ell,h,\delta}$ the law of $\vert \cZ_{\ell}^{h,+}\vert$ conditionally on $\phid(\circ)=h$, we have $\dE^{\Td}_h[\cZ_{\ell}^{h,+}]=\dE[\rho_{\ell, h, \delta}]>1$ if $\ell$ is large enough (depending on $h$ and $\delta$).
\end{remark}

%

\begin{proof}[Proof of Proposition~\ref{prop:expomomentsCh}]
Fix $a\geq h$, and let $\ell\in \dN$ be such that the conclusion of Lemma~\ref{lem:uniformsupercritical} holds. Let $F_1:=\partial B_{\Ch}(\circ, \ell)$. For $j\geq 1$, if $F_j\neq\emptyset$, choose an arbitrary vertex $z_j\in F_j$. Let $O_j$ be the $\ell$-offspring of $z_j$ in $\Ch$ and let $F_{j+1}:=O_j\cup F_j\setminus\{z_j\}$. Thus, we explore $\Ch$ by revealing subtrees of height $\leq \ell$, so that at each step, we see at most $(d-1)+\ldots +(d-1)^{\ell}\leq d^{\ell+1}$ new vertices. Hence, if $\vert \Ch\vert\geq k$ for some $k\in \dN$, there will be at least $\lfloor k/d^{\ell+1}\rfloor$ steps before $\Ch$ is fully explored. 
\\
By Lemma~\ref{lem:monotonicityphid} (applied to $\{\vert \cZ_{\ell}^{h,+} \vert\geq k\}$ for every $k\geq 1$), for every $j\geq 1$, $\vert F_j\vert$ dominates stochastically a sum $S_{j}$ of $j$ i.i.d. random variables of law $\rho_{\ell,h} -1$, where 
\begin{equation}\label{eqn:rhoellhdef}
\text{$\rho_{\ell,h}$ is the law of $\vert \cZ_{\ell}^{h,+}\vert$ conditionally on $ \phid(\circ)=h$.}
\end{equation}
Therefore, 
\[
\dP^{\Td}_a(k\leq \vert \Ch \vert <+\infty)\leq \sum_{j=\lfloor k/d^{\ell+1}\rfloor}^{+\infty} \dP(S_{j}\leq 0).
\]
But a variable of law $\rho_{\ell,h}-1$ is bounded and has a positive expectation by Lemma~\ref{lem:uniformsupercritical}, therefore there exist $c,c'>0$ such that $\dP(S_j\leq 0)\leq c e^{-c'j}$ for all $j\geq 1$. Hence, 
\[
\dP^{\Td}_a(k\leq \vert \Ch \vert <+\infty)\leq c \sum_{j=\lfloor k/d^{\ell+1}\rfloor }^{+\infty}e^{-c'j}\leq \frac{c\exp(-c'\lfloor k/d^{\ell+1}\rfloor)}{1-e^{-c'}}
\]
and (\ref{eqn:expomoments}) follows.
\end{proof}

\begin{proof}[Proof of Proposition~\ref{prop:Chlargedevgrowthrate}] 
Let $\varepsilon >0$. By definition of $M_k^h$ and Lemma~\ref{lem:SZ16}, 
\begin{center}
$\{\vert \cZ_k^{h,+}\vert \geq \chi_h(a)(\lambda_h+\varepsilon)^k\}\subset \{M_k^h\geq \lambda_h^{-k}\chi_{h,\min}\chi_h(a)(\lambda_h+\varepsilon)^k \}$ 
\end{center}
so that by Markov's inequality, for any $a\geq h$,
\begin{align*}
\dP_a^{\Td}(\vert \cZ_k^{h,+}\vert \geq\chi_h(a) (\lambda_h+\varepsilon)^k)&\leq \dP_a^{\Td}\left(M_k^h\geq \chi_{h,\min}\chi_h(a)\left(\frac{\lambda_h+\varepsilon}{\lambda_h}\right)^k\right)
\\
&\leq \chi_{h,\min}^{-1}\left(\frac{\lambda_h}{\lambda_h+\varepsilon}\right)^k\chi_h(a)^{-1}\dE_a^{\Td}[M_k^h]
\\
&\leq \chi_{h,\min}^{-1}\left(\frac{\lambda_h}{\lambda_h+\varepsilon}\right)^k.
\end{align*}

\noindent
This yields the upper large deviation for $\cZ_k^{h,+}$ (and for $\cZ_k^{h}$, using the facts that $\cZ_k^{h} \subseteq\cZ_k^{h,+}\cup \cZ_{k-1}^{h,-}$ and that $\vert\cZ_{k-1}^{h,-}\vert$ and $\vert\cZ_{k-1}^{h,+}\vert$ have the same law). 
\\
It remains to prove that for some $C>0$ and $k$ large enough,
\begin{equation}\label{eqn:largedevlowerrate}
\max_{a\geq h}\dP^{\Td}_a(k^{-1}\log\vert \cZ_k^{h}\vert \leq\log(\lambda_h-\varepsilon)\,\vert \,\cZ_k^{h} \neq \emptyset)\leq \exp(-Ck).
\end{equation}
We proceed in two steps. We first initiate the growth of $\Ch$ by showing that if $\cZ_{\ell n}\neq \emptyset$, the probability that $\vert \cZ_{\ell n}\vert=o(n)$ decays exponentially with $n$, where $\ell$ is such that Lemma~\ref{lem:uniformsupercritical} holds. Then, if $\cZ_{\ell n}$ has at least $\Theta(n)$ vertices, each of them has a positive probability to have a $Kn$-offspring of size at least $\lambda_h^{Kn}/n^3\geq (\lambda_h-\varepsilon)^k$ with $k:=(K+\ell)n$ and for a large enough constant $K$, independently of the others vertices. Hence the probability that $\vert\cZ_{k}\vert \leq (\lambda_h-\varepsilon)^k$ decays exponentially with $n$, and thus with $k$. 
\\
\textbf{First step.} Recall the exploration of $\Ch$ of the proof of Proposition~\ref{prop:expomomentsCh}, but perform it in a breadth-first way: reveal first the $\ell$-offspring of $\circ$, then the $\ell$-offspring of each vertex of $\cZ_{\ell}^h$, then the $\ell$-offspring of each vertex of $\cZ_{2\ell}^h$, and so on. For $n\geq 1$, if $\cZ_{\ell n}^h\neq \emptyset$, let $j+1$ be the first step at which we explore the offspring of a vertex of $\cZ_{\ell n}$. Note that $j\geq n$. As in the proof of Proposition~\ref{prop:expomomentsCh}, there exist $\epsilon,c,c'>0$ such that $\dP(S_i\leq \epsilon i)\leq c e^{-c'i}$ for every $i\geq 1$. Hence, for every $n\geq 1$, 
\begin{equation}\label{eqn:largedevstep1}
\min_{a\geq h}\dP_a^{\Td}(\vert \cZ_{\ell n}^h\vert \geq \epsilon n\, \vert \cZ_{\ell n}^h\neq \emptyset)\geq 1-\sum_{i\geq n}ce^{-c'i} \geq 1-\frac{c}{1-e^{-c'}}e^{-c'n}. 
\end{equation}

\noindent
\textbf{Second step.} Let $K$ be a positive integer constant, and let $F$ be the set of vertices $z\in\cZ_{\ell n}^h$ such that the $Kn$-offspring of $z$ has at least $\lambda_h^{K n}/n^3$ vertices. This step mainly comes down to showing that the probability that $F$ is empty decays exponentially with $n$.
\\
Define the events
\begin{center}
$\cE_n:=\{\vert \cZ_{K n}^{h,+} \vert \geq \lambda_h^{K n}/n^3\}$ and $\cE'_n:=\{\vert \cZ_{K n-1}^{h,+} \vert \geq \lambda_h^{K n}/n^3\}$.
\end{center}
We first show that 
\begin{equation}\label{eqn:pminexpogrowth}
p:=\min_{a\geq h}\dP_a^{\Td}(\cE_n) >0.
\end{equation}
By (\ref{eqn:thm43}), $\liminf_{n\rightarrow +\infty} \dP^{\Td}(\cE'_n)=:p'\hspace{-1mm}>0$.
Let $a_1$ be such that $\dP(\phid(\circ)\hspace{-1mm}\geq\hspace{-1mm}a_1\hspace{-0.5mm})\hspace{-1mm}<\hspace{-1mm}p'/4$. For $n$ large enough,
\begin{center}
$\int_{a\geq h}\dP^{\Td}_a(\cE'_n)\nu(da)>p'/2$, hence $\int_{h}^{a_1}\dP^{\Td}_a(\cE'_n)\nu(da)>p'/4,$
\end{center}
where we recall that $\nu$ is the density of $\phid(\circ)$. Since $\cE'_n$ is an increasing event, by Lemma~\ref{lem:monotonicityphid}: 
\begin{center}
$\min_{a\geq a_1}\dP^{\Td}_{a}(\cE'_n)\geq p'/4$. 
\end{center}
By Lemma~\ref{lem:monotonicityphid} again, 
\begin{center}
$\min_{a\geq h} \dP^{\Td}_{a}(\exists z\in \cZ_1^{h,+},\, \phid(z)\geq a_1)= \dP^{\Td}_{h}(\exists z\in \cZ_1^{h,+},\, \phid(z)\geq a_1)=:p''>0$. 
\end{center}
Hence $p\geq p'p''/4$ and (\ref{eqn:pminexpogrowth}) is proved. Note that $\vert F\vert \stdomi\text{Bin}(\vert \cZ_{\ell n}^h\vert, p'')$. Thus by (\ref{eqn:largedevstep1}), 
$$\min_{a\geq h}\dP^{\Td}_a(\vert F\vert\geq 1 \vert \cZ_{\ell n}^h\neq \emptyset) \geq 1 - \dP_a^{\Td}(\vert \cZ_{\ell n}^h\vert \leq \epsilon n \vert \cZ_{\ell n}^h\neq \emptyset) -(1-p'')^{\epsilon j}\geq 1-ce^{-c'n}$$
for $n$ large enough, up to changing the values of the constants $c$ and $c'$. Therefore,
\[
\max_{a\geq h}\dP^{\Td}_a\left( \vert \cZ_{(K+\ell) n}^h \vert < \lambda_h^{K n}/n^3 \,\vert \, \cZ_{\ell n}^h \neq \emptyset\right) \leq ce^{-c'j}\leq ce^{-c'n}.
\]
If $K$ is large enough, then for $n$ large enough, $ \lambda_h^{Kn}/n^3 > (\lambda_h-\varepsilon)^{(K+\ell)(n+1)} $, so that
\[
\max_{a\geq h}\dP^{\Td}_a\left(\vert \cZ_{(K+\ell) n}^h \vert < (\lambda_h-\varepsilon)^{(K+\ell) (n+1)} \,\vert\, \cZ_{\ell n}^h \neq \emptyset  \right)\leq ce^{-c'n}.
\]
\noindent
We adjust the conditionning: $\{\cZ_{(K+\ell) n}^h\neq \emptyset \} \subset\{\cZ_{\ell n}^h\neq \emptyset  \}$, and $\vert \cZ_{(K+\ell) n}^h \vert < (\lambda_h-\varepsilon)^{(K+\ell) (n+1)} $ on $\{\cZ_{\ell n}^h\neq \emptyset  \}\setminus \{\cZ_{(K+\ell) n}^h\neq \emptyset \}$.
Therefore, there exists $n_0\geq 1$ such that for all $n\geq n_0$,
\begin{equation}\label{eqn:lowerlargedevendalmost}
\max_{a\geq h}\dP^{\Td}_a\left(\vert \cZ_{(K+\ell) n}^h \vert < (\lambda_h-\varepsilon)^{(K+\ell) (n+1)} \,\vert\, \cZ_{(K+\ell) n}^h \neq \emptyset  \right)\leq c\exp(-c'(K+\ell)(n+1)),
\end{equation}
where the new value of $c'$ depends on the constants $K$ and $\ell$. This yields (\ref{eqn:largedevlowerrate}) for large enough multiples of $K+\ell$. One can readily replace each $\cZ_m^h$ by $\cZ_m^{h,+}$ in this reasoning (for any $m\geq 1$), to get the same result for $\cZ_{(K+\ell) n}^{h,+}$ instead of $\cZ_{(K+\ell) n}^h$.
\\
It remains to show the result for non multiples of $K+\ell$. Let $k\geq (K+\ell) n_0$. Write $k=(K+\ell) n + m$, with $0\leq m\leq (K+\ell) -1$. Note that on  $\{\cZ_{k}^h \neq \emptyset\}$, $\cZ_{m}^h$ has at least one vertex whose $(k-m)$-offspring is not empty.
Hence
\begin{align*}
\max_{a\geq h}\dP^{\Td}_a\left(\vert \cZ_{k}^h \vert < (\lambda_h-\varepsilon)^{k} \,\vert\, \cZ_{k}^h \neq \emptyset  \right)&\leq \max_{a\geq h}\dP^{\Td}_a\left(\vert \cZ_{k-m}^{h,+} \vert < (\lambda_h-\varepsilon)^{k} \,\vert\, \cZ_{k-m}^{h,+} \neq \emptyset  \right)
\\
&\leq \max_{a\geq h}\dP^{\Td}_a\left(\vert \cZ_{k-m}^{h,+} \vert < (\lambda_h-\varepsilon)^{(k-m)+(K+\ell)} \,\vert\, \cZ_{k-m}^{h,+} \neq \emptyset  \right)
\\
&\leq c\exp(-c'(K+\ell )(n+1))
\\
&\leq ce^{-c'k},
\end{align*}
\noindent
where the third inequality comes from (\ref{eqn:lowerlargedevendalmost}). Adapting this last computation for $\cZ_k^{h,+}$ is immediate. This concludes the proof of (\ref{eqn:expomomentssize}).
\end{proof}

\subsection{Proof of Proposition~\ref{prop:couplinggffsexplo}}\label{subsec:AppendixSection4}

\noindent
In the proof, we will use the following two facts on tree excesses, that we use implicitely in other parts of the paper. First, for any subgraph $A$ of any finite graph $G$ and $R\in \dN$, $\tx(B_G(A,R))\geq \tx(A)$, with equality if and only if $B_G(A,R)$ has the same number of cycles and the same number of connected components as $A$. Second, if $G$ is connected, $\tx(G)=0$ if and only if $G$ is a tree, i.e. has no cycle.
\\
Let us give a short proof. For the first fact, it is enough to show that for any choice of $A$ and $G$, the sequence $(\tx(B_G(A,R)))_{R\geq 0}$ is non-decreasing. Since $B_G(A,R+1)=B_G(B_G(A,R),1)$, it is enough to check that $\tx(B_G(A,1))\geq \tx(A)$, without loss of generality. Let $a_1, \ldots, a_n$ be the vertices of $A$ for some $n\geq 1$. For $i=1, \ldots, n$, add to $A$ the neighbours of $a_i$ and the corresponding edges. This yields $B_G(A,1)$. Each time we add a neighbour $x$ of $a_i$, if $x$ was not already in $A$, then the edge $\{a_i,x\}$ was not either. Hence, the quantity $(\text{number of edges }-\text{number of vertices})$ is not decreasing during this process, so that $\tx(B_G(A,1))\geq \tx(A) $. As for the second fact, it is classical that if $G$ is a tree, then $\tx(G)=0$ (see e.g. Theorem 2.2 in~\cite{tree}). Reciprocally, if $G$ is connected, is not a tree and $\tx(G)=0$, then there is a spanning tree $T$ of $G$ such that $\tx(T)<0$ (since $T$ is obtained from $G$ by deleting no vertex and at least one edge on a cycle of $G$), but this contradicts the fact that if $T$ is a tree, then $\tx(T)=0$.
\\
\\
We will also need the lemma below, which is a consequence of Lemma~3.3 in \cite{CTW} and of the following observation. If $(X_j)_{j\geq 0}$ is a SRW on $\Td$, then the trajectory of its height $(\mathfrak{h}_{\Td}(X_j))_{j\geq 0}$ is distributed as a random walk on $\dN\cup \{0\}$ with transition probabilities $1/d$ towards the left neighbour and $(d-1)/d$ towards the right neighbour, and reflected at $0$.

\begin{lemma}[\textbf{Geometric repulsion}]\label{lem:repulsion}
Let $s\in \dN$, and  $A\subset V_n$ such that $\emph{\tx}(B_{\cM_n}(A,s))=\emph{\tx}(A)$. Let $x\in V_n\setminus B_{\cM_n}(A,s)$, let $(X_j)_{j\geq 0}$ be a SRW started at $x$ and $\tau$ its first hitting time of $A$. Then $\tau$ dominates stochastically a geometric random variable of parameter $(d-1)^{-s}$.
\end{lemma}

We give a shortened proof of Proposition~\ref{prop:couplinggffsexplo}, due to the many similarities with the proof of Proposition 2.7 in~\cite{ACregulgraphs}, and refer the reader to~\cite{ACregulgraphs} for details.
\begin{proof}[Proof of Proposition~\ref{prop:couplinggffsexplo}]
\noindent 
Let us first prove (\ref{eqn:condexpapprox}). By our assumptions on $A$ and $y$ $T_y:=B_{\cM_n}(y,\overline{y},a_n-1)$ is a tree rooted at $y$. Let $\partial T_y$ be the $(a_n-1)$-offspring of $y$ in $T_y$.  Let $\tau$ (resp.~$\tau'$) be the hitting time of $\partial T_y$ (resp.~$\partial T_y\cup \{y\}$ by a SRW $(X_j)_{j\geq 0}$.   By Proposition~\ref{prop:condgff},
\begin{equation}\label{eqn:condexpsplit}
\begin{split}
\dE^{\cM_n}[\psimn(y)\vert \sigma(A)]  = \bE_{y}^{\cM_n}\left[\psimn\left(X_{H_{A}}\right)\mathbf{1}_{H_A\leq \tau}\right] &+ \bE_{y}^{\cM_n}\left[\psimn\left(X_{H_{A}}\right)\mathbf{1}_{H_A>\tau}\right]
\\
&-\frac{\bE_{y}^{\cM_n}[H_{A}]}{\bE_{\pi_n}^{\cM_n}[H_{A}]}\bE_{\pi_n}\left[\psimn\left(X_{H_{A}}\right)\right].
\end{split}
\end{equation}
\noindent
By the reasoning leading to the domination of $U_{A,x}^{\cG_n}$ below (2.43) in~\cite{ACregulgraphs}, we get 
\[
\lv \bE_{y}^{\cM_n}\left[\psimn\left(X_{H_{A}}\right)\mathbf{1}_{H_A\leq \tau}\right] -\frac{1}{d-1}\psimn(\overline{y}) \rv \leq \log^{-4}n.
\]
In particular, it is worth noting that in our context, (2.33) in~\cite{ACregulgraphs} becomes 
\begin{equation}\label{eqn:leavingtree}
\left\vert\bP_y^{\cM_n}(H_A\leq \tau)-\frac{1}{d-1}\right\vert\leq \log^{-6}n \text{ and } 
\bE_y^{\cM_n}\left[\tau'\right]\leq 2\log n .
\end{equation} 
%
%
\noindent
Therefore, to establish (\ref{eqn:condexpapprox}), it is enough to show that
\begin{equation}\label{eqn:condexpsplit1}
\lv \bE_{y}^{\cM_n}\left[\psimn\left(X_{H_{A}}\right)\mathbf{1}_{H_A>\tau}\right] - \frac{d-2}{d-1} \bE_{\pi_n}^{\cM_n}\left[\psimn\left(X_{H_{A}}\right)\right]\rv \leq 3\log^{-5}n
\end{equation}
and 
\begin{equation}\label{eqn:condexpsplit2}
\lv\frac{d-2}{d-1}  - \frac{\bE_{y}^{\cM_n}[H_{A}]}{\bE_{\pi_n}^{\cM_n}[H_{A}]}\rv \bE_{\pi_n}^{\cM_n}\left[\psimn\left(X_{H_{A}}\right)\right] \leq \log^{-5}n.
\end{equation}

\noindent
We start with \eqref{eqn:condexpsplit1}. Following the proofs of (2.35) (using Lemma~\ref{lem:repulsion} with $s=a_n$, fixing $\kappa$ large enough, instead of Lemma 3.4 of~\cite{CTW}) and (2.36) in~\cite{ACregulgraphs}, and using the fact that $\max_{z\in A}\vert\psimn(z)\vert\leq \log^{2/3}n$, we get that 
\[
\sup_{z\in \partial T_y} \vert \bE_{z}^{\cM_n}\left[\psimn\left(X_{H_A}\right)\right] - \bE_{\pi_n}^{\cM_n}\left[\psimn\left(X_{H_A}\right)\right]\vert \leq  \log^{-6}n.
\]
Also, by \eqref{eqn:leavingtree}, $ \lv\bP_y^{\cM_n}(H_A> \tau) -\frac{d-2}{d-1}\rv \leq \log^{-6}n$ . Combining these two facts yields \eqref{eqn:condexpsplit1}.

\noindent
As for \eqref{eqn:condexpsplit2}, we have $\bE_y^{\cM_n}\left[H_A\right]=\bE_y^{\cM_n}\left[\tau'\right]+\sum_{z\in \partial T_y}\bP_y^{\cM_n}(X_{\tau'}=z)\bE_z^{\cM_n}\left[H_A \right].$ 
By (3.20) of \cite{CTW}, $\bE_{\pi_n}^{\cM_n}[H_{A}]\geq \frac{n}{4\vert A \vert}\geq \log^{8}n\,/4$. Combining this with \eqref{eqn:leavingtree}, we get
\[
\lv\frac{\bE_{y}^{\cM_n}[H_{A}]}{\bE_{\pi_n}^{\cM_n}[H_{A}]}-\sum_{z\in \partial T_y}p'_z\frac{\bE_z^{\cM_n}\left[H_A \right]}{\bE_{\pi_n}^{\cM_n}[H_A]}\rv \leq \frac{8\log n}{\log^{8}n}\leq \log^{-6}n,
\]
where $p'_z:=\bP_y^{\cM_n}(X_{\tau'}=z)$. By (\ref{eqn:leavingtree}), $\lv \sum_{z\in \partial T_y}p'_z-\frac{d-2}{d-1}\rv\leq \log^{-6}n.$ Therefore, 
\[
\lv\frac{\bE_{y}^{\cM_n}[H_{A}]}{\bE_{\pi_n}^{\cM_n}[H_{A}]}-\frac{d-2}{d-1}\rv   \leq     2\log^{-6}n +\max_{z\in \partial T_y}\lv\frac{\bE_z^{\cM_n}\left[H_A \right]}{\bE_{\pi_n}^{\cM_n}[H_A]} -1 \rv.
\]
Since $\max_{z\in A}\vert\psimn(z)\vert\leq \log^{2/3}n$,  (\ref{eqn:condexpsplit2}) and thus \eqref{eqn:condexpapprox} follows once we show that 
\begin{equation}\label{eqn:ratioExEpi}
\max_{z\in \partial T_y}\lv\frac{\bE_z^{\cM_n}\left[H_A \right]}{\bE_{\pi_n}^{\cM_n}[H_A]} -1 \rv\leq 5\log^{-6}n.
\end{equation}
Here lies the main difference with Proposition 2.7 in~\cite{ACregulgraphs}, which makes the stronger assumption that $\vert A\vert =O(\log n)$. 
For all $z\in \partial T_y$, following the first part of the proof of Proposition 3.5 in \cite{CTW}, we obtain the upper bound: 
%
%
\begin{equation}\label{eqn:CTW35upperbound}
\bE_z^{\cM_n}\left[H_A \right]\leq \log^2n+\hspace{-1mm}\sum_{z'\in V_n}\hspace{-1mm}\hspace{-1mm}\left(\pi_n(z')+e^{-\lambda_{\cM_n}\log^2n}\right)\bE_{z'}^{\cM_n}[H_A]\leq\hspace{-1mm} \log^2n+(1+ne^{-\lambda_{\cM_n}\log^2n})\bE_{\pi_n}^{\cM_n}[H_A].
\end{equation}
Recalling that $\bE_{\pi_n}^{\cM_n}[H_A]\geq \log^{8}n /4$; this yields 
\begin{equation}\label{eqn:CTW35upperboundthen}
\frac{\bE_z^{\cM_n}\left[H_A \right]}{\bE_{\pi_n}^{\cM_n}\left[H_A \right]}\leq \frac{4\log^2n}{\log^{8}n}+1+ne^{-\lambda_{\cM_n}\log^2n}\leq 1+ 5\log^{-6}n.
\end{equation}
\noindent
Conversely, the second part of the proof of~\cite{CTW} gives us that
\begin{align*}
\frac{\bE_z^{\cM_n}\left[H_A \right]}{\bE_{\pi_n}^{\cM_n}\left[H_A \right]}\geq &1 -ne^{-\lambda_{\cM_n}\log^2n}-\bP_z^{\cM_n}(H_A\leq \log^2n)(1+5\log^{-6}n).
\end{align*}
By Lemma~\ref{lem:repulsion} applied as below (\ref{eqn:condexpsplit2}), we have $\sup_{z\in \partial T_y}\bP_z^{\cM_n}(H_A\leq \log^2n)\leq \log^{-6}n$. Together with \ref{expander1} and \eqref{eqn:CTW35upperboundthen}, this yields \eqref{eqn:ratioExEpi} and 
the proof of (\ref{eqn:condexpapprox}) is complete (note that the required lower bounds on $\kappa$ given by Lemma~\ref{lem:repulsion} are uniform in $y$ and $A$). 

\noindent
We prove (\ref{eqn:condvarapprox}) in the same fashion. Since $H_A\geq \tau'$ a.s., (\ref{eqn:condvar}) yields
\begin{equation}\label{eqn:condvarsplit}
\begin{split}
\left\vert \text{Var}^{\cM_n}(\psimn(y)\vert \sigma(A))-\frac{d}{d-1}\right \vert &\leq  \lv\gmn(y,y) -\bE_{y}^{\cM_n}\left[\gmn\left(y,X_{H_A}\right)\mathbf{1}_{H_A= \tau'}\right]  -\frac{d}{d-1}\rv    
\\
+& \lv \bE_{y}^{\cM_n}\hspace{-1mm}\left[\gmn\hspace{-1mm}\left(y,X_{H_A}\right)\mathbf{1}_{H_A > \tau'}\right]  -  \frac{\bE_{y}^{\cM_n}[H_A]}{\bE_{\pi_n}[H_A]}\bE_{\pi_n}^{\cM_n}\hspace{-1mm}\left[\gmn\hspace{-1mm}\left(y,X_{H_A}\right)\right] \rv.
\end{split}
\end{equation}
\noindent
Following the reasoning at (2.47) and below in Proposition 2.7 of~\cite{ACregulgraphs}, one shows that the first term of the RHS is  $O(\log^{-5}n)$.
Then, we apply to the second term of the RHS of (\ref{eqn:condvarsplit}) the same reasoning as (\ref{eqn:condexpsplit1}) and (\ref{eqn:condexpsplit2}), the inequality $\max_{x,y\in V_n}\vert\gmn(x,y)\vert\leq\civ$ (by (\ref{eqn:greenfunctionGn}), since $\cM_n$ is a good graph) replacing the inequality $\max_{z\in A}\vert \psimn(z)~\vert~\leq~\log^{2/3}~n$.
\end{proof}

\subsection{Proof of Lemmas~\ref{lem:couplinggffadapted} and~\ref{lem:couplinggffadapted1cycle}}\label{subsec:AppendixSection7}

\begin{proof}[Proof of Lemma~\ref{lem:couplinggffadapted}] We follow the argument of Proposition~\ref{prop:couplinggffsexplo}, with a few adjustments.
\\
First, the bounds in the first inequality of (\ref{eqn:leavingtree}) and in (\ref{eqn:condexpsplit1}) are in fact $e^{-ca_n}$ for some constant $c>0$, and one can replace $a_n$ by $r_n$. 
\\
Second, the proof of (\ref{eqn:condexpsplit2}) can be adapted by noting the following two facts. On one hand, he condition $\vert A\vert \leq n^{2/3}$ implies that $\bE_{\pi_n}[H_A]\geq n^{1/3}/4$. On the other hand,  the second inequality of (\ref{eqn:leavingtree}) still holds for large enough $n$, using the same comparison with a biased random walk as in (2.33) in~\cite{ACregulgraphs}, since for $k\geq 0.3\log n$ and a constant $\gamma >0$ that depends neither on $k$ nor on $n$, $\dP(\tau'\geq k)\leq e^{-\gamma k}$.
\end{proof}

\begin{proof}[Proof of Lemma~\ref{lem:couplinggffadapted1cycle}]
We will only show (\ref{eqn:condexpapproxadaptedcase2}) and (\ref{eqn:condvarapproxadaptedcase2}). The other proofs are very similar and left to the reader.
\\
We start with the proof of (\ref{eqn:condexpapproxadaptedcase2}). We follow the proof scheme of (\ref{eqn:condexpapprox}). Let $\tau$ be the hitting time of $\partial B_{\cM_n}(y,\overline{y},r_n)\setminus \{y\}$ by a SRW $(X_k)_{k\geq 0}$. Note that $\{H_A\leq \tau\}\subseteq\{X_{H_A}=y\}$. We write
\begin{align*}
\dE^{\cM_n}[\psimn(y)\vert \sigma(A)]=&\bP_{y}^{\cM_n}(X_{H_A}=\overline{y},H_A\leq \tau)\psimn(\overline{y})+\bE_y^{\cM_n}[\psimn(X_{H_A})\mathbf{1}_{H_A>\tau}]\\&-\frac{\bE_y^{\cM_n}[H_A]}{\bE_{\pi_n}^{\cM_n}[H_A]}\bE_{\pi_n}^{\cM_n}[\psimn(X_{H_A})].
\end{align*} 

\noindent
Since $B_{\cM_n}(y,\overline{y},r_n)$ and $B_{G_m}(z_k,\overline{z_{k}},r_n)$ are isomorphic, 
\begin{center}
$\bP_{y}^{\cM_n}(X_{H_A}=\overline{y},H_A\leq \tau) =\bP_{z_{k}}^{G_m}(X_{H_{\partial B_{G_m}(z_k,\overline{z_{k}},r_n)}}=\overline{z_{k}})=:\alpha''_k$.
\end{center}
Since 
\begin{center}
$\{X_{H_{\partial B_{G_m}(z_k,\overline{z_{k}},r_n)}}=\overline{z_{k}}\}\subseteq \{H_{\{\overline{z_k}\}}<+\infty \}$, we have $\alpha''_k\leq \alpha'_k$. 
\end{center}
Reciprocally, on $ \{H_{\{\overline{z_{k}}\}}<+\infty \}\setminus \{X_{H_{\partial B_{G_m}(z_k,\overline{z_{k}},r_n)}}=\overline{z_{k}}\}$, a SRW starting at distance $r_n$ of $C_m$ has to reach $C_m$. As in the proof of (\ref{eqn:leavingtree}), a comparison with a biased SRW on $\dZ$ shows that this happens with a probability $O(e^{-cr_n})$ for some constant $c>0$ uniquely depending on $d$ and we get that if $a$ small enough, then for large enough $n$, $\vert\bP_{y}^{\cM_n}(X_{H_A}=\overline{y},H_A\leq \tau) -\alpha'_k \vert\leq n^{-3a}$.
\\
It remains to establish 
\begin{equation}\label{eqn:condexpadaptedcase2compens}
\bigg|\bE_y^{\cM_n}[\psimn(X_{H_A})\mathbf{1}_{H_A>\tau}]-\frac{\bE_y^{\cM_n}[H_A]}{\bE_{\pi_n}^{\cM_n}[H_A]}\bE_{\pi_n}^{\cM_n}[\psimn(X_{H_A})]\bigg| \leq n^{-3a}.
\end{equation}
To do so, one adapts the proofs of (\ref{eqn:condexpsplit1}) and (\ref{eqn:condexpsplit2}) exactly as in the proof of Lemma~\ref{lem:couplinggffadapted}: if $H_A>\tau$, $(X_k)_{k\geq 0}$ leaves $B_{\cM_n}(A\cup C\cup P,r_n)$ before hitting $A$. By Lemma~\ref{lem:repulsion} with $s=r_n$, if $a$ is small enough, $(X_k)$ does not hit $A$ within the next $\log^2n$ steps with probability at least 
\begin{center}
$(1-(d-1)^{-r_n})^{\log^2n}\geq 1-2\log^2n\,\,\,(d-1)^{-r_n}\geq 1-n^{-4a}$. 
\end{center}
Then, we use Corollary 2.1.5 of~\cite{Saloff-Coste1997} as in the proof of (2.36) of~\cite{ACregulgraphs}: 
after $\lfloor \log^2n\rfloor$ steps, the fact that $\cM_n$ is an expander (i.e. $\lambda_{\cM_n}>\ciii>0$) forces the empirical distribution of $X_k$ to be very close to the uniform distribution $\pi_n$. (\ref{eqn:condexpapproxadaptedcase2}) follows.
\\
For (\ref{eqn:condvarapproxadaptedcase2}), we follow the proof scheme of (\ref{eqn:condvarapprox}). If $\tau'$ is the exit time of $B_{\cM_n}(y,\overline{y},r_n)$, we have
\begin{equation}\label{eqn:condvaradaptedcase2split}
\begin{split}
\left\vert \text{Var}^{\cM_n}(\psimn(y)\vert \sigma(A))-\gamma'_k\right \vert &\leq  \lv\gmn(y,y) -\bE_{y}^{\cM_n}\left[\gmn\left(y,X_{H_A}\right)\mathbf{1}_{H_A= \tau'}\right]  -\gamma'_k\rv    
\\
 +& \lv \bE_{y}^{\cM_n}\left[\gmn\left(y,X_{H_A}\right)\mathbf{1}_{H_A > \tau'}\right]  -  \frac{\bE_{y}^{\cM_n}[H_A]}{\bE_{\pi_n}[H_A]}\bE_{\pi_n}^{\cM_n}\left[\gmn\left(y,X_{H_A}\right)\right] \rv.
\end{split}
\end{equation}
We deal with the second term of the RHS as (\ref{eqn:condexpadaptedcase2compens}) to show that it is $O(n^{-3a})$. As for the first term, we have as in (2.47) and below in~\cite{ACregulgraphs}:
\begin{center}
$\bE_{y}^{\cM_n}\left[\gmn\left(y,X_{H_A}\right)\mathbf{1}_{H_A= \tau'}\right]
=\bE_{y}^{\cM_n}\left[\gmn\left(y,X_{\tau'}\right)\right]-O(n^{-3a})$
\end{center}
if $a$ is small enough, by (\ref{eqn:greenfunctionGn}). Now, by (\ref{eqn:lem14Aba}) applied to $D:=B_{\cM_n}(y,\overline{y},r_n)$ (note that $T_{D}=\tau'$) and the second inequality of (\ref{eqn:leavingtree}) (which still holds, as remarked in the proof of Lemma~\ref{lem:couplinggffadapted}), we get 
\begin{center}
$\vert \gmn(y,y) - \bE_{y}^{\cM_n}\left[\gmn\left(y,X_{\tau'}\right)\right] -G_{\cM_n}^{D}(y,y) \vert  \leq \frac{\bE_y^{\cM_n}[\tau']}{n}=O(n^{-3a})$.
\end{center}
But $D$ is isomorphic to $E:=B_{G_m}(z_k,\overline{z_k},r_n)$, so that 
\begin{align*}
G_{\cM_n}^{D}(y,y)&=G_{G_m}^E(z_k,z_k)
\\
&=G_{G_m}(z_k,z_k)-\bP_{z_k}^{G_m}(T_E=\overline{z_k})G_{G_m}(z_k,\overline{z_k})-\bP_{z_k}^{G_m}(T_B\neq\overline{z_k})G_{G_m}(z_k,z)
\end{align*}
for any $z\in \partial B_{G_m}(E,1)\setminus\{\overline{z_k}\}$, by cylindrical symmetry of $B_{G_m}(E,1)$. One checks easily that if $a$ is small enough, then for $n$ large enough, 
\begin{center}
$ G_{\cM_n}^{D}(y,y)=G_{G_m}^E(z_k,z_k)=G_{G_m}(z_k,z_k)-\bP_{z_k}^{G_m}(T_E=\overline{z_k})G_{G_m}(z_k,\overline{z_k})+O(n^{-3a})$.
\end{center}
One easily adapts the reasoning leading to (\ref{eqn:leavingtree}), despite the presence of one cycle, to get 
\begin{center}
$\bP_{z_k}^{G_m}(T_E=\overline{z_k})=\bP_{z_k}^{G_m}(H_{\{\overline{z_k}\}}<+\infty)+O(n^{-3a})$ for $a$ small enough. 
\end{center}
Note indeed that $\{T_E=\overline{z_k}\}\subseteq \{H_{\{\overline{z_k}\}}<+\infty\}$. Reciprocally, if $z\in \partial B_{G_m}(E,1)\setminus\{\overline{z_k}\}$, a SRW starting at $z$ has a probability decaying exponentially with $r_n$ to reach $\overline{z_k}$, since there are at most two injective paths from $z$ to $\overline{z_k}$, and each contains at least $r_n-3$ vertices where the SRW has a positive probability (only depending on $d$) to enter a subtree isomorphic $\Td^+$ and to never leave it. 
\\
Since $\gamma'_k=G_{G_m}(z_k,z_k)- \bP_{z_k}^{G_m}(H_{\{\overline{z_k}\}}<+\infty)  G_{G_m}(z_k,\overline{z_k}) $, we obtain
\begin{center}
$\vert G_{\cM_n}^{D}(y,y)-\gamma'_k\vert =O(n^{-3a})$.
\end{center} 
All in all, we get that the first term of the RHS of (\ref{eqn:condvaradaptedcase2split}) is $O(n^{-3a})$, and (\ref{eqn:condvarapproxadaptedcase2}) follows.
\end{proof}

\end{appendix}

\end{document}